\documentclass[a4paper,11pt]{amsart}
\usepackage{amsmath,amsthm,amssymb,amsfonts,enumerate,xcolor,esint,wasysym}
\usepackage[pdftex]{graphicx}
\usepackage{float}
\usepackage{tikz}
\usetikzlibrary{cd}
\usepackage{rotating}
\usetikzlibrary{positioning}
\usepackage{subfig}
\usepackage{multicol}
\usepackage[shortlabels]{enumitem}
\usepackage{mathtools}
\usepackage[colorlinks=true, allcolors=blue]{hyperref}
\usepackage{amsrefs}
\newcommand{\N}{\mathbb{N}}

\newcommand{\R}{\mathbb{R}}
\newcommand{\Z}{\mathbb{Z}}

\newcommand{\ext}{\operatorname{Ext}}
\newcommand{\avg}{\operatorname{Avg}}

\newcommand{\lip}{\operatorname{Lip}}

\newtheorem{thm}{Theorem}[section]
\newtheorem{prop}[thm]{Proposition}

\newtheorem{lem}[thm]{Lemma}

\theoremstyle{definition}
\newtheorem{defn}[thm]{Definition}
\newtheorem{rem}[thm]{Remark}

\newtheorem{assumption}[thm]{Assumption}
\theoremstyle{definition}

\numberwithin{equation}{section}

\allowdisplaybreaks

\author[A. Aboud]{\href{https://www.westmont.edu/people/anna-aboud-phd}{Anna Aboud}}
\address{Department of Mathematics\\
Westmont College\\ 
955 La Paz Road\\
Santa Barbara CA 93108,
United States of America}
\email{aaboud@westmont.edu}

\author[P.~Alonso Ruiz]{\href{https://www.math.tamu.edu/~paruiz/}{Patricia Alonso Ruiz}}
\address{Institute of Mathematics\\ Friedrich Schiller University Jena\\ 07743 Jena, Germany}
\email{patricia.alonso.ruiz@uni-jena.de}

\author[M.~Vaughan]{\href{https://maryvaughan.github.io/}{Mary Vaughan}}
\address{Department of Mathematics and Statistics\\
The University of Western Australia\\
35 Stirling HWY\\
Crawley WA 6009, Australia}
\email{mary.vaughan@uwa.edu.au}


 %Methods involving semicontinuity and convergence; relaxation

\begin{document}
%%%%%%%%%%%%%%%%%%%%%%%%%%%%%%%%%%%%%%%%%%%%%%%%

%%%%%%%%%%%%%%%%%%%%%%%%%%%%%%%%%%%%%%%%%%%%%%%%
\title[Nonlocal energies on the unit interval]{A new graph-directed construction \\of nonlocal energies on the unit interval}
%%%%%%%%%%%%%%%%%%%%%%%%%%%%%%%%%%%%%%%%%%%%%%%%

%%%%%%%%%%%%%%%%%%%%%%%%%%%%%%%%%%%%%%%%%%%%%%%%
\begin{abstract}
We present an analytic construction of nonlocal energies on the unit interval. 
The energies are defined using a new graph-directed construction of discrete energies on dyadic approximations of the interval. 
When the discrete jump kernels are comparable to the kernel of the fractional discrete Laplacian, we prove that the discrete energies Mosco converge and the limiting energy is equivalent to the fractional Gagliardo seminorm. 
\end{abstract}
%%%%%%%%%%%%%%%%%%%%%%%%%%%%%%%%%%%%%%%%%%%%%%%%

\maketitle
%%%%%%%%%%%%%%%%%%%%%%%%%%%%%%%%%%%%%%%%%%%%%%%%

%%%%%%%%%%%%%%%%%%%%%%%%%%%%%%%%%%%%%%%%%%%%%%%%
\section{Introduction}
%%%%%%%%%%%%%%%%%%%%%%%%%%%%%%%%%%%%%%%%%%%%%%%%

In a continuum space, such as the unit interval, nonlocal energies naturally arise as functionals associated with operators describing processes that exhibit long-range interactions. These appear in applications across many areas of science, such as mathematical finance, materials science, fluid mechanics, and social science.  
Roughly speaking, (purely) nonlocal processes are associated with jumps. This is in contrast to local processes, where interactions happen only within small neighborhoods and whose associated processes are of diffusion type. Nonlocal processes can be obtained from local ones through subordination, as is the case of L\'evy processes constructed from Brownian motion.

When the underlying space is discrete, as happens with graphs and networks, the concept of nonlocality refers to the fact that interactions (wires) are allowed beyond graph neighbors. Such models appear in the context of anomalous diffusions in complex systems; see e.g.~\cites{MK00,RM14} and references therein.

Both discrete and continuous settings come together when a continuum, as the unit interval, is approximated by a sequence of discrete sets, for example dyadic points. Discrete approximations are the basis for practical simulation of processes taking place in a theoretic continuum. They are also key in the construction of local processes on porous media modeled by post-critically finite (p.c.f.) self-similar fractals; see~\cites{Kigami01,Str06} for a general introduction to the topic.

Motivated by the local construction in the fractal setup, the present paper explores an analogous analytic construction of \emph{nonlocal} energies in the continuum space $[0,1]$ as limits of 
energies on discrete dyadic approximations. The focus is thus on the ability to directly construct a nonlocal process in a limit space \emph{without using the (a priori) existence of such a process}. This approach stands in contrast to the typical subordination method used in the literature, see e.g.~\cites{Hin09,CKK}. 

More precisely, we propose a way to construct  
energy forms on finite (dyadic) graph approximations of $[0,1]$ whose limit will yield nonlocal Dirichlet forms on the unit interval. The approach draws from the theory of electrical networks that is used in the fractal context to construct local Dirichlet forms~\cite{Kigami01}*{Chapter 2}. As such, we refer to the graph approximations as networks.
Following the precedent of~\cite{Str06}, the interval $[0,1]$ serves as model space to test the construction. 
Discrete approximations of nonlocal energies and operators have appeared in \cite{Stinga} to compare discrete fractional Laplacians and discretized fractional Laplacians, 
and in \cite{CKK} to study discrete approximations of jump processes on metric measure spaces. 
Discrete approximations are also useful in numerical approximations; our approach is partially motivated by the possibility of creating new % a 
numerical schemes incorporating boundary conditions in the future.

An important observation in the nonlocal setting is that, to encode nonlocality into the resulting Dirichlet form, the finite networks (i.e.~graphs) must be complete, that is any two nodes (i.e.~vertices) in the network
are connected by a wire (i.e.~edge), see Figure \ref{fig:intro} below.   
This precludes both finite ramification and self-similarity, 
and it is not obvious how to implement the analogous steps to those in the standard local construction where it is crucial to understand how the resistances between wires change from one approximation stage to the next. Other works in the literature dealing with approximation of nonlocal processes leave this question open~\cite{CKK}.

\begin{figure}[htb]
\begin{center}
\begin{tikzpicture}[thick,scale=.7, main/.style ={circle, draw, fill=black!50,
                        inner sep=0pt, minimum width=4pt}]
%n=0
\node[main]  (0) at (-10,3.5) {};
\node[main]  (4) at (-2,3.5) {};
    \draw (0) to (4);
\node at (-6,2.8) {\small Stage $0$};
%n=1
\node[main]  (0) at (0,3.5) {};
\node[main]  (2) at (4,3.5) {};
\node[main]  (4) at (8,3.5) {};
    \draw (0) to (2) to (4);
    \draw[blue] (0) to  [looseness=1]  (4);
\node at (4, 2.8) {\small Stage $1$};
%n=2
\node[main]  (0) at (-10,0) {};
\node[main]  (1) at (-8,0) {};
\node[main]  (2) at (-6,0) {};
\node[main]  (3) at (-4,0) {};
\node[main]  (4) at (-2,0) {};
% length 4
\draw[orange] (0) to  [looseness=1]  (4);
% length 3
\draw[red] (0) to  [looseness=1]  (3);
\draw[red] (1) to  [looseness=1]  (4);
% length 2
\draw[blue] (0) to  [looseness=1]  (2) to   [looseness=1] (4);
\draw[blue] (1) to  [looseness=1]  (3);
    % length 1
\draw (0) to (1) to (2) to (3) to (4);
\node at (-6, -.7) {\small Stage $2$};
%%n=3
\foreach \x in {0,1,2,3,4,5,6,7,8}{
	\node[main]  (\x) at (\x,0) {};
    }
% length 1
\foreach \x in {0,1,2,3,4,5,6,7}{
    \draw (\x) to (\x+1,0);
    }
% length 8
\draw[violet] (0) to  [looseness=1]  (8);
% length 7
\draw[magenta] (0) to  [looseness=1]  (7);
\draw[magenta] (1) to  [looseness=1]  (8);
% length 6
\foreach \x in {0,1,2}{
    \draw[cyan] (\x) to [looseness=1] (\x+6,0);
    }
% length 5
\foreach \x in {0,1,2,3}{
    \draw[green!85!blue] (\x) to [looseness=1] (\x+5,0);
    }
% length 4
\foreach \x in {0,1,2,3,4}{
    \draw[orange] (\x) to [looseness=1] (\x+4,0);
    }
% length 3
\foreach \x in {0,1,2,3,4,5}{
    \draw[red] (\x) to [looseness=1] (\x+3,0);
    }
% length 2
\foreach \x in {0,1,2,3,4,5,6}{
    \draw[blue] (\x) to [looseness=1] (\x+2,0);
    }
\foreach \x in {0,1,2,3,4,5,6,7,8}{
	\node[main]  (\x) at (\x,0) {};
    }
\node at (4, -.7) {\small Stage $3$};
\end{tikzpicture}
\end{center}
\caption{Electrical networks (also called finite graph approximations)}
\label{fig:intro}
\end{figure}
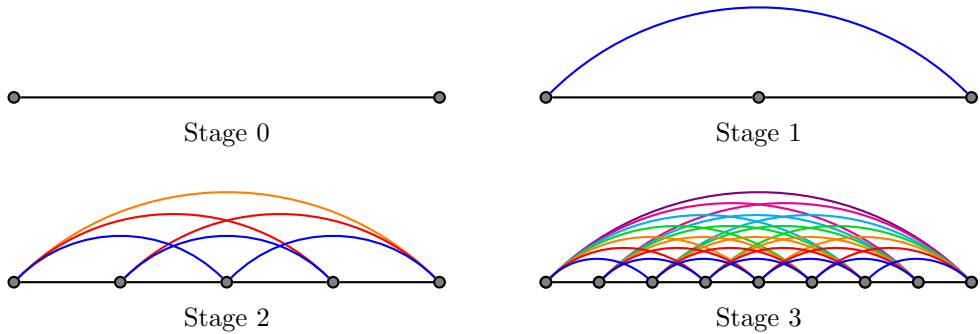

The solution we propose is a new \emph{graph-directed} construction of discrete energies on the discrete approximations of the unit interval. 
An extension to more general spaces, including p.c.f.~fractals, will be the subject of future work.

\subsection{Graph-directed constructions}\label{subsec:GD_construction}

Graph-directed constructions were  
introduced by Maudin and Williams in \cite{MauldinWilliams} as a way to obtain new deterministic fractals. 
Roughly speaking, such a graph is a geometric construction based on the following ingredients: 
\begin{enumerate}
    \item a (finite) index set $\{1,\dots, N\}$,
    \item a  sequence of compact and connected subsets $J_i \subset \R^d$, $i \in \{1,\dots, N\}$,
    \item a weighted, directed graph consisting of $N$ vertices, indexed by $i \in \{1,\dots,N\}$,
    \item a (contraction) map from $J_j$ to $J_i$ for each directed edge from $i$ to $j$.
\end{enumerate}
Under appropriate assumptions on the compact sets and maps, the construction by Mauldin and Williams produces a fractal in $\mathbb{R}^d$ based on that index set $\{1,\dots, N\}$.  
Later on, Hambly and Nyberg developed a general analytic approach to diffusion processes on finitely ramified, graph-directed self-similar  
fractals that followed the Mauldin--Williams construction~\cite{Hambly-Nyberg}. Their methods rely on the theory of Dirichlet forms on p.c.f.~fractals due to Kigami~\cites{Kigami01}. 
We refer to~\cites{Hybrids,NX20,CQTY22} for further analytic investigations on graph-directed fractals. 

While the unit interval as a set can be obtained from a graph-directed construction with $N=1$ 
(see Appendix \ref{sec:appendixA}),
our goal is rather to introduce a graph-directed construction as in \cite{Hambly-Nyberg} to obtain a nonlocal energy on $[0,1]$ as the  limit of discrete energies featuring the long-range interactions at each dyadic approximation of $[0,1]$. 
We highlight some unique features of our construction here and refer to Section \ref{sec:construction} for details.

From Figure \ref{fig:intro}, recall that the approximating electrical networks must be \emph{complete}, meaning all nodes are connected by a wire, to effectively encode all nonlocal  interactions in the resulting  energy on $[0,1]$. 
Consequently, the number of wires connected to a given node grows exponentially like $2^n$ as the stage $n$ of approximation increases.  
In this regard, it is helpful to view the wires in Figure \ref{fig:intro} as intervals in $\R$:
\begin{itemize}
    \item At stage 0, view the single black wire as $[0,1]$. 
    \item At stage 1, the black wires 
    correspond to $[0,\frac12]$ and $[\frac12,1]$, regarded as contractions of $[0,1]$. The blue wire corresponds to a copy of $[0,1]$ (drawn with an arc for visual purposes), and it is the new wire introduced at this stage. 
\end{itemize}
The colored wires in Figure \ref{fig:intro} each arise from a copy of the unit interval and make up the sequence of compact, connected subsets of $\R$ in our graph-directed construction. Creation of a complete network necessitates an index set of $\N$, rather than the finite index set of size $N$. 
To the best of our knowledge, 
graph-directed constructions with infinitely many vertices are new in the literature and present a significant distinction from \cites{Hambly-Nyberg,MauldinWilliams}. 

Another important difference from previous graph-directed works is that not all of the maps involved in the construction have contraction ratio strictly less than one. Indeed, returning to Figure \ref{fig:intro}, the blue wire at Stage 2 arises from the identity map (i.e.~contraction ratio $1$), and we have nontrivial overlap with the two black wires (again viewed as the intervals $[0,1]$, $[0,\frac12]$, and $[\frac12,1]$ respectively). 
Since we are not interested in developing the unit interval itself, this does not pose a problem in our setting.

As a bi-product of the graph-directed construction, each index $i \in \N$ yields 
a unique sequence of electrical networks  that give rise to 
energies on the compact sets $[0,\frac{1}{i}] \cup [1 - \frac{1}{i},1]$.
We call these the \emph{physical spaces} (see Section \ref{sec:physical space}). 
The electrical networks for $i>1$ are complete bipartite graphs, namely, the wires connect dyadic rationals in $[0, \frac{1}{i}]$ to dyadic rationals in $[1 - \frac{1}{i},1]$. 
In a sense, these characterize the long-range interactions in the electrical networks and physical space for $i=1$ that are absent in the classical construction of local energies on the unit interval.

Using the electrical networks created by the graph-directed construction, we obtain discrete energies $\{\mathcal{E}_i^{(n)}\}_{n \geq 0, i \geq 1}$ that are related through a so-called graph-directed self-similar relationship, see Lemma \ref{L:graph_directed_ss_energy}. 

%%%%%%%%%%%%%%%%%%
\subsection{Obtaining nonlocal energies on the physical spaces}
%%%%%%%%%%%%%%%%%%

To define a meaningful nonlocal energy on $[0,1]$ from the discrete energies $\{\mathcal{E}_1^{(n)}\}_{n \geq 0}$, we start by following the 
\emph{finite energy approach} from~\cite{Kig89} and consider
\begin{equation}\label{eq:energy-intro}
\mathcal{E}_1^{(\infty)}(u) := \lim_{n \to \infty} \mathcal{E}_1^{(n)}(u \big|_{V_1^{(n)}}),
\end{equation}
where $u\big|_{V_1^{(n)}}$ denotes the usual restriction of $u \in L^2([0,1],dx)$ to the set of dyadic rationals in $[0,1]$ at stage $n$. 
The general idea of this approach is to use the theory of electrical networks to determine compatibility conditions on the weights of the wires in the discrete approximations that allow one to appropriately move between stages, see \cites{Kigami01,Tetali,DoyleSnell}. 
Since our electric networks are complete, we cannot repeatedly apply the usual network reduction rules (e.g.~series, parallel, delta-wye). Nevertheless, we present some necessary conditions for compatibility in Section \ref{subsec: compat} and provide a more comprehensive explanation in Appendix \ref{sec:appendixA}. 

Of particular interest and one of the main results in the paper is to determine the Mosco convergence of the sequence $\{\mathcal{E}_1^{(n)}\}_{n \geq 0}$. Mosco convergence has a number of valuable implications, including the convergence of related processes, of energy minimizers, and of the associated semigroups. We are able overcome the aforementioned compatibility challenges and construct a meaningful nonlocal energy by assuming that the jump kernels of the discrete energies are comparable to the kernel of the $s$-fractional discrete Laplacian for $s \in (0,\frac12)$ studied in \cite{Stinga}. 
A key observation towards the proof is the ability 
to characterize the limiting energy~\eqref{eq:energy-intro} by means of a \emph{density approach} as was done in~\cite{CKK}. Indeed, we prove in Theorem~\ref{thm:DF-main} that both approaches yield the same energy functionals with the same domains under reasonable assumptions on the discrete jump kernels. 
One benefit to the density approach over the finite energy approach is the simplicity of the domain and the availability of tools from \cite{CKK} to study the (Mosco) convergence as $n \to \infty$. That being said, the main convergence result in \cite{CKK} does not immediately apply to our setting as not all assumptions hold.
In addition, we prove that the energy \eqref{eq:energy-intro} is equivalent to the square of the fractional seminorm for the Gagliardo space $H^s([0,1])$, $s \in (0,1)$.

\medskip

Although our main objective  
 is the construction of nonlocal energies on the interval, we also analyze the energies on the physical spaces $[0,\frac{1}{i}] \cup [1 - \frac{1}{i},1]$ associated with $i>1$ in the graph-directed construction. As it turns out, less regularity is required to obtain similar results to the case of the interval, including the $\Gamma$-convergence of the discrete energies.

Since we are dealing with infinitely many energies, both on each level of approximation $\{\mathcal{E}_i^{(n)}\}_{i\geq 1}$ and on the physical spaces $\{\mathcal{E}_i^{(\infty)}\}_{i \geq 1}$, we also explore the result when we send $i \to \infty$. 
Evaluated along a fixed function $u$, we prove in Section \ref{sec:i-limit} that both limits result in the same energy $|u(0) - u(1)|^2$. 
Studying more general notions of convergence from this point of view is left as the subject of future research.

\medskip

%%%%%%%%%%
\subsection{Organization of the paper}
%%%%%%%%%%

The paper is organized as follows:
Section \ref{sec:construction} presents the graph-directed construction that models the discrete approximations of the physical spaces and their associated electric networks.
The latter are used in Section \ref{sec:discrete}
to define discrete energies on the discrete approximations of the physical space. 
In Section \ref{sec:limit}, 
we discuss both the finite energy approach and the density approach and show their equivalence. 
The proof of the Mosco convergence of the nonlocal discrete energies under suitable assumptions is presented in Section \ref{sec:mosco}. 
Finally, 
Section \ref{sec:i-limit} addresses the question of taking the limit of the nonlocal energies as $i \to \infty$. 
For expository reasons, we present the analogous construction of local energies on the unit interval and showcase the challenges that arise in the nonlocal setting in Appendix \ref{sec:appendixA}. 
In Appendix \ref{sec:appendixB}, we establish some estimates on the effective resistance distance which  
may be useful for future study.

%%%%%%%%%%%%%%%%%%%%%%%%%%%%%%%%%%%%%%%%%%%%%%%%
\section{Graph-directed construction of electrical networks}\label{sec:construction}
%%%%%%%%%%%%%%%%%%%%%%%%%%%%%%%%%%%%%%%%%%%%%%%%

In this section, we explain how to use a graph-directed construction that  gives rise to
a sequence of electrical networks on the discrete approximations of the so-called physical spaces $\overline{V}_i^*$, see~\eqref{E:Vstar_def}. The unit interval corresponds to the case $i=1$.  
First, we introduce a weighted, directed graph that we call the \emph{index space}. 
Second, the directed graph is used to define finite approximations of the physical spaces, and finally to construct the associated electrical networks.

%%%%%%%%%%%%%%%%%%%%%%%%%%%%%%%%%%%%%%%%%%%%%%%%
\subsection{Index space}\label{sec:symbolic space}
%%%%%%%%%%%%%%%%%%%%%%%%%%%%%%%%%%%%%%%%%%%%%%%%

We consider the directed graph $(\N,E)$ described in Figure~\ref{fig:graph-directed}.
Its edge set is given by
\begin{equation*}\label{E:edge_set_def}
E := \bigcup_{i,j\in\N} E_{i,j},
\end{equation*}
where $E_{i,j}$ is the set of all directed edges connecting vertex $i$ to vertex $j$ and is given by
\begin{equation}\label{E:Eij_def}
    E_{i,j}=\begin{cases}
    \{e_{i,j}\}&\text{if }j=2i-2 \text{ and }i\geq 2;\\
    \{e_{i,j},e_{i,j}'\}&\text{if }j=2i-1;\\
    \{e_{i,j}\}&\text{if }j=2i;\\
    \emptyset&\text{ else (including $i=1,j=0$)}.
    \end{cases}
\end{equation}
As an abuse of notation, we will sometimes write $e_{i,j}$ to denote an arbitrary edge from $i$ to $j$, i.e.,~we do not always include the prime notation when $j=2i-1$. 
In addition, for $i\in\N$, the set $E_i$ will denote the set of all edges that exit the vertex $i$, that is
\begin{equation}\label{E:Ei_def}
    E_i:=\bigcup_{j\in\N}E_{i,j}.
\end{equation}
From \eqref{E:Eij_def} and Figure \ref{fig:graph-directedA}, note that $e_{1,0}$ does not exist and $E_1$ only has three elements, whereas $E_i$ has four elements for all $i \geq 2$.

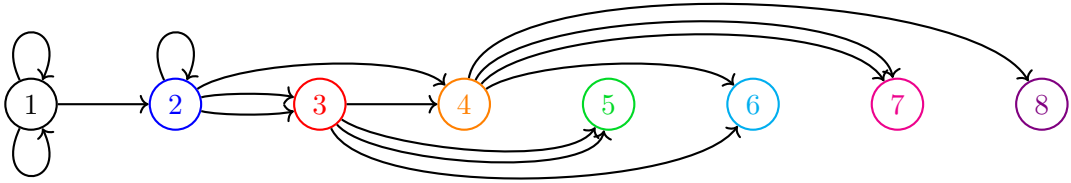
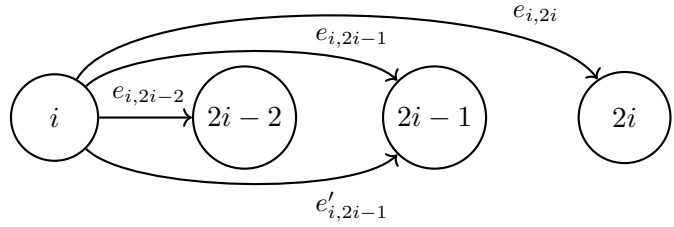
\begin{figure}[hbt]
\begin{center}
\subfloat[Directed edges from vertices $i=1,2,3,4$]{
\label{fig:graph-directedA}
\begin{tikzpicture}[ node distance={19mm}, thick, main/.style = {draw, circle}] 
\node[main] (1) {$1$}; 
\node[main,blue] (2) [right of=1] {$2$};
\node[main,red] (3) [ right of=2] {$3$};
\node[main,orange] (4) [right of=3] {$4$};
\node[main,green!85!blue] (5) [right of=4] {$5$};
\node[main,cyan] (6) [right of=5] {$6$};
\node[main,magenta] (7) [right of=6] {$7$};
\node[main,violet] (8) [right of=7] {$8$};
% \node[main,brown] (9) [right of=8] {$9$};
% \node[main,gray] (10) [right of=9] {$10$};
\draw[->] (1) -- (2);
\draw[->] (1) to [out=115,in=65,looseness=8] (1);
\draw[->] (1) to [out=-115,in=-65,looseness=8] (1);
\draw[->] (2) to [out=115,in=65,looseness=8]  (2);
\draw[->] (2) to [out=15,in=165,looseness=.5](3);
\draw[->] (2) to [out=-15,in=-165,looseness=.5] (3);
\draw[->] (2)  to [out=35,in=135,looseness=.5] (4);
\draw[->] (3) --(4);
\draw[->] (3) to [out=-35,in=-120,looseness=.5]  (5);
\draw[->] (3) to [out=-50,in=-100,looseness=.5] (5);
\draw[->] (3) to [out=-65,in=-125,looseness=.5]  (6);
\draw[->] (4)  to [out=35,in=135,looseness=.5] (6);
\draw[->] (4) to [out=50,in=120,looseness=.5]  (7);
\draw[->] (4) to [out=65,in=100,looseness=.5] (7);
\draw[->] (4) to [out=80,in=120,looseness=.5] (8);
\end{tikzpicture} }

\subfloat[Directed edges from vertex $i \geq 2$]{
\label{fig:graph-directedB}
\begin{tikzpicture}[node distance={25mm}, thick, main/.style = {draw, circle}] 
\node[main] (3) {$\,\,\,\,\,i\,\,\,\,\,$}; 
\node[main] (4) [right of=3] {$2i-2$};
\node[main] (5) [right of=4] {$2i-1$};
\node[main] (6) [right of=5] {$\,\,\,\,2i\,\,\,\,$};
\draw[->] (3) -- (4);
\draw[->] (3) to [out=45,in=135,looseness=.5] (5);
\draw[->] (3) to [out=-45,in=-135,looseness=.5] (5);
\draw[->] (3) to [out=60,in=125,looseness=.5] (6);
%weights
\node [above right=-.4cm and .2cm of 3] {\small $e_{i,2i-2}$}; 
\node [above right=.3cm and .3cm of 4] {\small $e_{i,2i-1}$}; 
\node [below right=.35cm and .3cm of 4] {\small $e_{i,2i-1}'$};
\node [above right=.6cm and .4cm of 5] {\small $e_{i,2i}$}; 
\end{tikzpicture} 
}
\end{center}
\caption{Directed graph $(\N,E)$ in the index space}
\label{fig:graph-directed}
\end{figure}

A path $\sigma$ of length $n\geq 1$ is a concatenation of $n$ consecutive edges as traversed along $(\N,E)$ in Figure \ref{fig:graph-directed} and may be represented as
\begin{equation}\label{eq:path}
\sigma = e_{i_0, i_1} e_{i_1,i_2} \dots e_{i_{n-1}, i_n}.
\end{equation}
In particular, the above path 
originates at the vertex $i_0$,  passes through the vertices $i_k$, $1 \leq k \leq n-1$, and terminates at vertex $i_n$.
Furthermore, in view of~\eqref{E:edge_set_def}, 
\begin{equation}\label{eq:allowable}
i_{k} \in \{ 2i_{k-1}-2,  2i_{k-1}-1, 2i_{k-1} \}, \quad k=1,\dots, n.
\end{equation}
Given a path $\sigma$ as in \eqref{eq:path}, 
we denote the $j$th vertex it passes through by
\[
\sigma(j) := i_j,\qquad 0 \leq j \leq n
\]
and write $\sigma = e_{\sigma(0), \sigma(1)} \cdots e_{\sigma(n-1), \sigma(n)}$. 
Analogous to~\eqref{E:Eij_def} and~\eqref{E:Ei_def}, we define the collection of paths $\sigma$ of length $n$ starting at $i$ and ending at $j$ by
\begin{equation*}
    E_{i,j}^{(n)}:=\{\sigma\text{ of length }n~\colon~\sigma(1)=i,~\sigma(n)=j\}.
\end{equation*}
In particular, $E_{i,j}^{(1)}=E_{i,j}$. 

We now describe the sets $E_i^{(n)}$ in terms of the sets $E_{i,j}^{(n)}$. 

\begin{lem}\label{lem:path sets}
Fix $n \in \N$ and let $|E|$ denote the cardinality of a countable set $E$. 
Then
\begin{equation}\label{eq:by-object}
E_1^{(n)} = \bigcup_{j=1}^{2^n} E_{1,j}^{(n)} \quad \hbox{and} \quad
E_i^{(n)} =\bigcup_{j=2^n(i-2)+2}^{2^ni} E_{i,j}^{(n)} \quad \hbox{for}~i>1.
\end{equation}
\end{lem}

\begin{proof}
Let $i \in \N$. 
Let $j_{\min}$ and $j_{\max}$ be smallest and largest possible $j \in \N$ such that $\sigma \in E_i^{(n)}$ with $\sigma(n) = j$, respectively. 
From \eqref{eq:allowable}, we observe that $j_{\max} = 2^ni$ as seen from the path described by the sequence of indices
\[
i_0 := i, \quad i_{k} := 2i_{k-1} = 2^k i, \quad k=1,\dots n. 
\]
On the other hand, $j_{\min}$ is described by the allowable sequence of indices
\[
i_0 := i>1, \quad i_k := 2i_{k-1}-2 \quad \hbox{for}~k=1,\dots, n. 
\]
One can show inductively that $i_k = 2^k(i-2)+2$. 
Consequently, $j_{\min} = 2^k(i-2)+2$ for $i >1$. 
When $i=1$, we have $j_{\min} = 1$. 
Hence, 
\[
E_{i,j}^{(n)} = \emptyset \quad \hbox{for}~j< 2^n(i-2)+2~\hbox{and}~j>2^ni
\]
which gives \eqref{eq:by-object}.
\end{proof}

On each edge $e \in E$, we assign a positive weight $r_{e}>0$. 
For ease of the notation, we may write $r_{i,j} := r_{e_{i,j}}$ and assume that 
\begin{equation*}
r_{i,2i-1}= r_{e_{i, 2i-1}} = r_{e_{i, 2i-1}'} \quad
\text{for all}~i \in \N.
\end{equation*}
The weighted directed graph $(\N,E)$ will be used in the next two subsections to construct a physical space for each $i\in\N$ as well an associated sequence of finite electrical networks approximating a nonlocal energy on said space. 

%%%%%%%%%%%%%%%%%%%%%%%%%%%%%%%%%%%%%%%%%%%%%%%%
\subsection{Physical spaces}\label{sec:physical space}
%%%%%%%%%%%%%%%%%%%%%%%%%%%%%%%%%%%%%%%%%%%%%%%%

What is the physical space associated with the index $i\in\N$? 
To any edge $e\in E_{i,j}$, $j \in \{2i-2,2i-1,2i\}$, we associate the mapping $\phi_e: [0,1] \to [0,1]$ 
given by
\begin{equation*}
\phi_e(x) := \frac{j}{2i}x + s_e,
\end{equation*}
where $s_e \geq 0$ denotes the shift
\[
s_e := \begin{cases}
0 & \hbox{if}~e  \in \{e_{i,2i-1}, e_{i,2i}\}\\
\frac{1}{2i} & \hbox{if}~e \in \{e_{i,2i-2}, e_{i,2i-1}'\}.
\end{cases}
\]
In particular, 
\begin{equation}\label{E:def_phi_e}
\phi_e([0,1]) = 
\begin{cases}
[0,1] & \hbox{if}~e = e_{i,2i}, \\
[0,1-\frac{1}{2i}] & \hbox{if}~e = e_{i,2i-1}, \\
[\frac{1}{2i},1] & \hbox{if}~e = e_{i,2i-1}', \\
[\frac{1}{2i},1-\frac{1}{2i}] & \hbox{if}~e = e_{i,2i-2}~\text{and}~i\geq 2.
\end{cases}
\end{equation}
In addition, for any path $\sigma = e_{\sigma(0), \sigma(1)} \dots e_{\sigma(n-1), \sigma(n)}\in E_i$ starting at $\sigma(0)=i\in\N$, define
\begin{equation}\label{eq:composition}
\phi_\sigma(x):= 
\phi_{e_{\sigma(n-1), \sigma(n)}}\circ 
\dots \circ 
\phi_{e_{\sigma(1), \sigma(2)}} \circ
\phi_{e_{\sigma(0), \sigma(1)}}.
\end{equation}

\begin{figure}[htb]
\begin{center}
\subfloat[$n=0,1,2$ and $i=1$]{
\label{fig:physical-1}
\begin{tikzpicture}[thick,scale=.79, main/.style ={circle, draw, fill=black!50,
                        inner sep=0pt, minimum width=4pt},
main2/.style ={circle, draw, fill=white,
                        inner sep=0pt, minimum width=4pt}]
%level 0
\draw[thin,dashed] (0,2)--(8,2);
\node[main] at (0,2) {};
\node[main] at (8,2) {};
\node at (-1,2) {\scriptsize  $n=0$};
%level 1
\draw[thin,dashed] (0,1)--(8,1);
\node[main] at (0,1) {};
\node[main] at (4,1) {};
\node[main] at (8,1) {};
\node at (-1,1) {\scriptsize  $n=1$};
%\level 2
\draw[thin,dashed] (0,0)--(8,0);
\node[main] at (0,0) {};
\node[main] at (2,0) {};
\node[main] at (4,0) {};
\node[main] at (6,0) {};
\node[main] at (8,0) {};
%labels
\node at (-1,0) {\scriptsize  $n=2$};
\node at (0,-.5) {\scriptsize$0$};
\node at (2,-.5) {\scriptsize$\frac{1}{4}$};
\node at (4,-.5) {\scriptsize$\frac{1}{2}$};
\node at (6,-.5) {\scriptsize$\frac{3}{4}$};
\node at (8,-.5) {\scriptsize$1$};
\node at (4,-1.5) {\vdots};
%continuous space
\draw[line width=2pt] (0,-2.25)--(8,-2.25);
\node[main] at (0,-2.25) {};
\node[main] at (8,-2.25) {};
\node at (-1,-2.25) {\scriptsize $\overline{V}_i^*$};
\node at (0,-2.75) {\scriptsize$0$};
\node at (8,-2.75) {\scriptsize$1$};
\node at (0,-2.8) {\phantom{blank}};
\node at (8.5,-1) {\phantom{x}};
\end{tikzpicture}
}
%%%
\subfloat[$n=0,1,2$ and $i>2$]{
\label{fig:physical-i}
\begin{tikzpicture}[thick,scale=.79, main/.style ={circle, draw, fill=black!50,
                        inner sep=0pt, minimum width=4pt},
main2/.style ={circle, draw, fill=white,
                        inner sep=0pt, minimum width=4pt}]
%level 0
\draw[thin,dashed] (0,2)--(8,2);
\node[main] (0L) at (0,2) {};
\node[main2] at (2,2) {};
\node[main2] at (6,2) {};
\node[main] (0R) at (8,2) {};
%\node at (-1,2) {\scriptsize  $n=0$};
%level 1
\draw[thin,dashed] (0,1)--(8,1);
\node[main] (0L) at (0,1) {};
\node[main] (2L) at (1,1) {};
\node[main2] at (2,1) {};
\node[main2] at (6,1) {};
\node[main] (2R) at (7,1) {};
\node[main] (0R) at (8,1) {};
%\node at (-1,1) {\scriptsize $n=1$};
%\level 2
\draw[thin,dashed] (0,0)--(8,0);
\node[main] (0L) at (0,0) {};
\node[main] (1L) at (.5,0) {};
\node[main] (2L) at (1,0) {};
\node[main]  (3L) at (1.5,0) {};
\node[main2] at (2,0) {};
\node[main2] at (6,0) {};
\node[main] (3R) at (6.5,0) {};
\node[main] (2R) at (7,0) {};
\node[main] (1R) at (7.5,0) {};
\node[main] (0R) at (8,0) {};
%labels
%\node at (-1,0) {\scriptsize  $n=2$};
%
\node at (0,-.5) {\scriptsize$0$};
\node at (.4,-.5) {\scriptsize$\frac{1}{2^2i}$};
\node at (1,-.5) {\scriptsize$\frac{2}{2^2i}$};
\node at (1.6,-.5) {\scriptsize$\frac{3}{2^2i}$};
\node at (2,-.5) {\scriptsize$\frac{1}{i}$};
\node at (8,-.5) {};%{\small$1$};
\node at (7.5,-.5) {};
\node at (7,-.5) {};
\node at (6.5,-.5) {};
\node at (6,-.5) {};%{\small$1-\frac{1}{i}$};
\node[scale=1.15] at (.75,-1.35) {\scriptsize $\underbrace{\hphantom{\text{xxx text}}}_{ V_{i-}^{(2)}}$};
\node[scale=1.15] at (7.25,-.75) {\scriptsize $\underbrace{\hphantom{\text{xxx text}}}_{ V_{i+}^{(2)}}$};
\node at (4,-1.5) {\vdots};
%continuous space
\draw[thin,dashed] (0,-2.25)--(8,-2.25);
\draw[line width=2pt] (0,-2.25)--(2,-2.25);
\draw[line width=2pt] (6,-2.25)--(8,-2.25); 
\node[main] (0L) at (0,-2.25) {};
\node[main] at (2,-2.25) {};
\node[main] at (6,-2.25) {};
\node[main] (0R) at (8,-2.25) {};
%\node at (-1,-2.25) {\small$\overline{V}_i^*$};
\node at (0,-2.75) {\scriptsize$0$};
\node at (2,-2.75) {\scriptsize$\frac{1}{i}$};
\node at (6,-2.75) {\scriptsize$1-\frac{1}{i}$};
\node at (8,-2.75) {\scriptsize$1$};
\end{tikzpicture}
}
\end{center}
\caption{Discrete approximations $V_i^{(n)}$ of the physical spaces $\overline{V}_i^*$}
\label{fig:physical}
\end{figure}
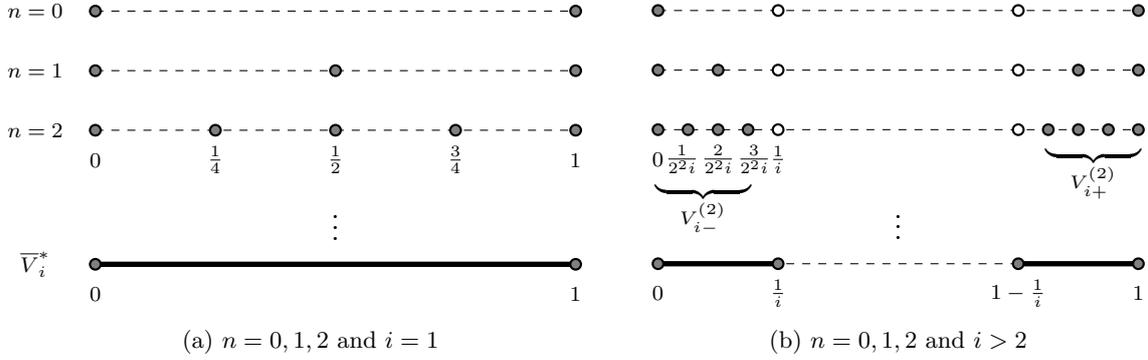

The physical space associated with the vertex $i\in\N$ of the directed graph $(\N,E)$ arises from a sequence of finite approximations $\{V_i^{(n)}\}_{n\geq 0}$ constructed as follows: Set $V_i^{(0)}:=\{0,1\}$ and for each $n\geq 1$ define
\begin{equation*}%\label{E:Vi_n_def}
V_i^{(n)} 
    :=\bigcup_{\sigma\in E_i^{(n)}} \phi_\sigma (V_i^{(0)}).
\end{equation*}
Alternatively and more explicitly,
\begin{equation}\label{eq:nodes}
 V_i^{(n)} = \bigcup_{j=2i-2}^{2i} \bigcup_{e \in E_{i,j}} \phi_e(V_i^{(n-1)})=
    \begin{cases}
    \displaystyle \left\{ \frac{k}{2^n}\right\}_{k=0}^{2^n} &\text{for }i=1,\\[.25em]
    \displaystyle\left\{ \frac{k}{i 2^n}, 1- \frac{k}{i2^n} \right\}_{k=0}^{2^n-1} & \text{for }i>1,
    \end{cases}
\end{equation} 
see Figure \ref{fig:physical}.
For indices $i >1$, it will be useful later on to partition $ V_i^{(n)}$ into the set of points to the left of $\frac{1}{2}$ and those to the right of $\frac{1}{2}$ by setting
\begin{align*}
V_{i-}^{(n)} &:= V_i^{(n)} \cap \Big[0,\frac{1}{i} \Big)= \Big\{ \frac{k}{i 2^n}\Big\}_{k=0}^{2^n-1}\\[.5em]
V_{i+}^{(n)} &:= V_i^{(n)} \cap \Big(1-\frac{1}{i},1 \Big] = \Big\{ 1-\frac{k}{i 2^n}\Big\}_{k=0}^{2^n-1},
\end{align*}
see Figure \ref{fig:physical-i}. 
Finally, the physical space associated with $i\in\N$ is obtained as the closure of
\[
V_i^{*} :=\bigcup_{n \geq 0} V_i^{(n)}.
\]
In particular,
\begin{equation}\label{E:Vstar_def}
\overline{V}_i^{*}
= \Big[0, \frac{1}{i}\Big] \cup \Big[1-\frac{1}{i}, 1\Big] \quad \text{for any}~i \in \N.
\end{equation}

\begin{rem}%\label{R:V1_V2_interval}
The indices $1$ and $2$ are both associated with the same physical space, namely 
\[
\overline{V_1}^{*}=[0,1]=\overline{V}_2^{*}.
\]
\end{rem}  

\subsection{Electric networks}\label{sec:electrical-networks}

Given $i\in\N$ and an approximation level $n\geq 0$, how do the weights $\{r_e\}_{e\in E}$ define an electric network in $V_i^{(n)}$? 
An electric network, or more simply network, consists of a set of \emph{nodes}, a set of \emph{wires} and \emph{resistances} attached to each wire. 
Networks are also referred to as resistance networks in the literature.
In the present setting, the nodes are the set $V_i^{(n)}$ given in~\eqref{eq:nodes} and the set of wires $ W_i^{(n)}$ is a collection of unordered pairs $\{x,y\}$ defined as
\begin{equation}\label{E:def_wires_n}
    W_i^{(n)}:= \bigcup_{\sigma\in E_i^{(n)}} \left\{\phi_\sigma(0),\phi_\sigma(1)\right\},
\end{equation}
since the wires are unordered pairs, $\{x,y\} \in W_i^{(n)}$ if and only if $\{y,x\} \in W_i^{(n)}$. 
In particular, the set of wires is given by 
\begin{equation}\label{eq:wires}
W_i^{(n)}
    =\begin{cases}
    (V_1^{(n)} \times V_1^{(n)}) \setminus \{\{x,x\} : x \in V_1^{(n)}\}& \hbox{if}~i=1,\\[.5em]
    (V_{i-}^{(n)} \times V_{i+}^{(n)}) \cup (V_{i+}^{(n)} \times V_{i-}^{(n)}) & \hbox{if}~i >1.
    \end{cases}
\end{equation}
We illustrate the networks $(V_i^{(n)}, W_i^{(n)})$ for $i=1,2$ and $n=0,1,2$ in Figure~\ref{fig:networks}. 
There we use arcs to visualize the wires $\{x,y\} \in W_i^{(n)}$ and color them according to the following two criteria:
\begin{enumerate}
\item In a given network, all wires $\{x,y\}$ with the same physical Euclidean distance $|x-y|$ are colored the same. 
\item A wire $\{x,y\}$ with index $\sigma\in E_i^{(n)}$ is colored the same as the wire in $(V_j^{(0)}, W_j^{(0)})$ such that $\sigma(n) = j$. For example, the red wire $\{0, \frac34\}$ in $(V_1^{(2)}, W_1^{(2)})$ corresponds to the path $\sigma = e_{1,2} e_{2,3}$, so that $\sigma(2)=3$. Thus, its color is the same as the red wire in $(V_3^{(0)}, W_3^{(0)})$. 
\end{enumerate}

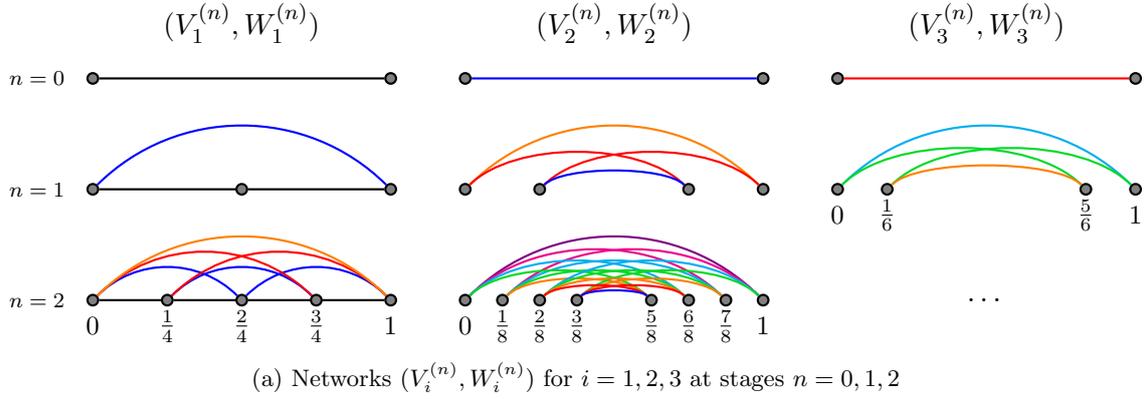
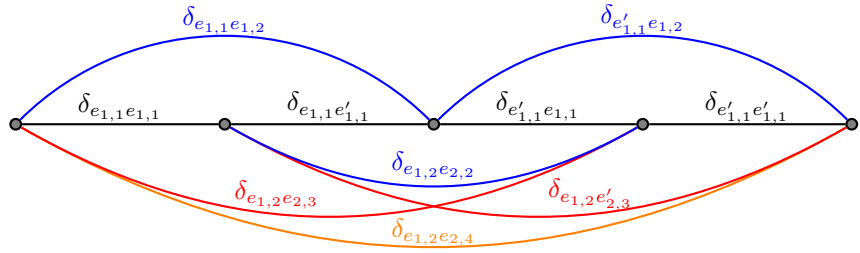
\begin{figure}[htb]
\begin{center}
\subfloat[Networks $(V_i^{(n)}, W_i^{(n)})$ for  $i=1,2,3$ at stages $n=0,1,2$]{\label{fig:networks}
\begin{tikzpicture}[thick,scale=.98, main/.style ={circle, draw, fill=black!50,
                        inner sep=0pt, minimum width=4pt}]
%%labels
\node at (2,.75) {$(V_1^{(n)}, W_1^{(n)})$};
\node at (7,.75) {$(V_2^{(n)}, W_2^{(n)})$};
\node at (12,.75) {$(V_3^{(n)}, W_3^{(n)})$};
\node at (-.75,0) {\scriptsize $n=0$};
\node at (-.75,-1.5) {\scriptsize $n=1$};
\node at (-.75,-3) {\scriptsize $n=2$};
%%i=1,n=0
\node[main]  (0) at (0,0) {};
\node[main]  (4) at (4,0) {};
    \draw (0) to (4);
%%i=2,n=0
\node[main]  (0') at (5,0) {};
\node[main]  (4') at (9,0) {};
    \draw[blue] (0') to (4');
%%i=3,n=0
\node[main]  (0'') at (10,0) {};
\node[main]  (4'') at (14,0) {};
    \draw[red] (0'') to (4'');
%%i=1, n=1
\node[main]  (0) at (0,-1.5) {};
\node[main]  (2) at (2,-1.5) {};
\node[main]  (4) at (4,-1.5) {};
    \draw (0) to (2) to (4);
    \draw[blue] (0) to  [looseness=1]  (4);
%%i=2, n=1
\node[main]  (0') at (5,-1.5) {};
\node[main]  (1') at (6,-1.5) {};
\node[main]  (3') at (8,-1.5) {};
\node[main]  (4') at (9,-1.5) {};
    \draw[orange] (0') to  [looseness=1]  (4');
    \draw[red] (0') to  [looseness=.75]  (3');
    \draw[red] (1') to  [looseness=.75]  (4');
    \draw[blue] (1') to  [looseness=.5]  (3');
%%i=3, n=1
\node[main]  (0'') at (10,-1.5) {};
\node[main]  (1'') at (10.666,-1.5) {};
\node[main]  (3'') at (13.333,-1.5) {};
\node[main]  (4'') at (14,-1.5) {};
    \draw[cyan] (0'') to  [looseness=1]  (4'');
    \draw[green!85!blue] (0'') to  [looseness=.75]  (3'');
    \draw[green!85!blue] (1'') to  [looseness=.75]  (4'');
    \draw[orange] (1'') to  [looseness=.5]  (3'');
%%i=1, n=2
\node[main]  (0) at (0,-3) {};
\node[main]  (1) at (1,-3) {};
\node[main]  (2) at (2,-3) {};
\node[main]  (3) at (3,-3) {};
\node[main]  (4) at (4,-3) {};
    % length 1
    \draw (0) to (1) to (2) to (3) to (4);
    % length 2
    \draw[blue] (0) to  [looseness=1]  (2) to   [looseness=1] (4);
    \draw[blue] (1) to  [looseness=1]  (3);
    % length 3
    \draw[red] (0) to  [looseness=1]  (3);
    \draw[red] (1) to  [looseness=1]  (4);
    % length 4
    \draw[orange] (0) to  [looseness=1]  (4);
%%i=2, n=2
\node[main]  (0L) at (5,-3) {};
\node[main]  (1L) at (5.5,-3) {};
\node[main]  (2L) at (6,-3) {};
\node[main]  (3L) at (6.5,-3) {};
\node[main]  (3R) at (7.5,-3) {};
\node[main]  (2R) at (8,-3) {};
\node[main]  (1R) at (8.5,-3) {};
\node[main]  (0R) at (9,-3) {};
    %from 4:
    \draw[violet] (0L) to  [looseness=1]  (0R);
    \draw[magenta] (0L) to  [looseness=.9]  (1R);
    \draw[magenta] (1L) to  [looseness=.9]  (0R);
    \draw[cyan] (1L) to  [looseness=.8]  (1R);
    %From 3L
    \draw[cyan] (0L) to  [looseness=.8]  (2R);
    \draw[green!85!blue] (0L) to  [looseness=.7]  (3R);
    \draw[green!85!blue] (1L) to  [looseness=.7]  (2R);
    \draw[orange] (1L) to  [looseness=.6]  (3R);
    %From 3R
    \draw[cyan] (2L) to  [looseness=.8]  (0R);
    \draw[green!85!blue] (2L) to  [looseness=.7]  (1R);
    \draw[green!85!blue] (3L) to  [looseness=.7]  (0R);
    \draw[orange] (3L) to  [looseness=.6]  (1R);
    %From 2:
    \draw[orange] (2L) to  [looseness=.6]  (2R);
    \draw[red] (2L) to  [looseness=.5]  (3R);
    \draw[red] (3L) to  [looseness=.5]  (2R);
    \draw[blue] (3L) to  [looseness=.4]  (3R);
%%%
\node at (0,-3.35) {\small $0$};
\node at (1,-3.35) {\small $\frac{1}{4}$};
\node at (2,-3.35) {\small $\frac{2}{4}$};
\node at (3,-3.35) {\small $\frac{3}{4}$};
\node at (4,-3.35) {\small $1$};
%%%
\node at (5,-3.35) {\small $0$};
\node at (5.5,-3.35) {\small $\frac{1}{8}$};
\node at (6,-3.35) {\small $\frac{2}{8}$};
\node at (6.5,-3.35) {\small $\frac{3}{8}$};
\node at (7.5,-3.35) {\small $\frac{5}{8}$};
\node at (8,-3.35) {\small $\frac{6}{8}$};
\node at (8.5,-3.35) {\small $\frac{7}{8}$};
\node at (9,-3.35) {\small $1$};
% \node at (12,-2.25) {\vdots};
% \node at (10.5,-3) {\dots};
\node at (12,-3) {\dots};
%%%
\node at (10,-1.85) {\small $0$};
\node at (10.666,-1.85) {\small $\frac{1}{6}$};
\node at (13.333,-1.85) {\small $\frac{5}{6}$};
\node at (14,-1.85) {\small $1$};
\end{tikzpicture}
}

\vspace{.2cm}

\subfloat[The network $(V_1^{(2)}, W_1^{(2)})$ with resistances]{\label{fig:resistances}
\begin{tikzpicture}[thick,scale=2.75, main/.style ={circle, draw, fill=black!50,
                        inner sep=0pt, minimum width=4pt}]
% nodes
\node[main]  (0) at (0,0) {};
\node[main]  (1) at (1,0) {};
\node[main]  (2) at (2,0) {};
\node[main]  (3) at (3,0) {};
\node[main]  (4) at (4,0) {};
% wires
\draw (0) to (1) to (2) to (3) to (4);
\draw[blue] (0) to  [looseness=1]  (2) to   [looseness=1] (4);
\draw[orange] (0) to  [bend right,looseness=1]  (4);
\draw[red] (0) to  [bend right,looseness=1]  (3);
\draw[red] (1) to  [bend right,looseness=1]  (4);
\draw[blue] (1) to  [bend right,looseness=1]  (3);
%resistances
\node at (.5,.08) {\small $\delta_{e_{1,1} e_{1,1}}$};
\node at (1.5,.08) {\small $\delta_{e_{1,1} e_{1,1}'}$};
\node at (2.5,.08) {\small $\delta_{e_{1,1}' e_{1,1}}$};
\node at (3.5,.08) {\small $\delta_{e_{1,1}' e_{1,1}'}$};
\node[blue] at (1,.5) {\small $\delta_{e_{1,1}e_{1,2}}$};
\node[blue] at (3,.5) {\small $\delta_{e_{1,1}'e_{1,2}}$};
\node[blue] at (2,-.2) {\small $\delta_{e_{1,2}e_{2,2}}$};
\node[red] at (1.25,-.34) {\small $\delta_{e_{1,2}e_{2,3}}$};
\node[red] at (2.75,-.34) {\small $\delta_{e_{1,2}e_{2,3}'}$};
\node[orange] at (2,-.52) {\small $\delta_{e_{1,2}e_{2,4}}$};
\end{tikzpicture}
}
\end{center}
\caption{Electrical networks}
\label{fig:physical construction}
\end{figure}

\begin{rem}\label{rem:copies}
In Figure \ref{fig:physical construction}, notice that $(V_1^{(2)}, W_1^{(2)})$ contains two scaled copies of $(V_1^{(1)}, W_1^{(1)})$ and one copy of $(V_2^{(1)}, W_2^{(1)})$. This is a consequence of the graph-directed construction and holds in general: for $i,n \in \N$, the network 
$(V_i^{(n)}, W_i^{(n)})$ contains a scaled copy of $(V_{2i-2}^{(n-1)}, W_{2i-2}^{(n-1)})$ (if $i>1$), two copies of $(V_{2i-1}^{(n-1)}, W_{2i-1}^{(n-1)})$, and one copy of $(V_{2i}^{(n-1)}, W_{2i}^{(n-1)})$. 
\end{rem}

The resistance $\delta_\sigma>0$ of a wire in $W_i^{(n)}$ indexed by $\sigma\in E_i^{(n)}$ is determined by the weights of the directed graph $(\N,E)$, see Figure \ref{fig:resistances}. 
For $\sigma = e_{\sigma(0), \sigma(1)} \dots e_{\sigma(n-1), \sigma(n)}$, we define
\begin{equation}\label{E:def_resistance_wire}
    \delta_\sigma:=r_{e_{\sigma(0), \sigma(1)}} \cdots r_{e_{\sigma(n-1), \sigma(n)}},
\end{equation}
where we recall that $r_{i,j}$ is edge weight on the edge $e_{i,j}$ in the index space.
By convention, for any $i \in \N$, the resistance on the wires $\{0,1\}$ at stage $n=0$ is always $1$.

In the next lemma, we will see that, in the case $i=1$, at each fixed approximation level $n\geq 0$, 
there is a one-to-one correspondence between wires $W_i^{(n)}$ and paths $E_i^{(n)}$.

\begin{lem}\label{lem:last edge}
Let $n \in \N$ and 
$\{x,y\} \in W_i^{(n)}$. 
Assume that $\sigma$ is the unique path associated to the wire $\{x,y\}$. % \in W_1^{(n)}$. 
Then $\sigma(n) = 2^ni|x-y|$ and, setting $j = \sigma(n)$, we have
\[
e_{\sigma(n-1),\sigma(n)}
    = \begin{cases}
        e_{\frac{j+1}{2}, j} & \hbox{if $j$ is odd and}~x \in V_i^{(n-1)} \\
        e_{\frac{j+1}{2}, j}' & \hbox{if $j$ is odd and}~x \notin V_i^{(n-1)} \\
        e_{\frac{j+2}{2}, j} & \hbox{if $j$ is even and}~x \notin V_i^{(n-1)} \\
        e_{\frac{j}{2}, j} & \hbox{if $j$ is even and}~x \in V_i^{(n-1)}.
    \end{cases}
\]
\end{lem}

By Lemma \ref{lem:last edge}, given any $\{x,y\} \in W_i^{(n)}$ with $x < y$,  
we can find the unique path $\sigma \in E_i^{(n)}$ such that $x=\phi_{\sigma}(0)$ and $y = \phi_{\sigma}(1)$. 
That is, each wire in the network is determined by a unique path $\sigma\in E_i^{(n)}$ with associated resistance  $\delta_{\sigma}>0$. 

\begin{rem}\label{R:complete_graph}
The networks $(V_i^{(n)},W_i^{(n)})$ can be seen as undirected, weighted graphs where vertices are nodes, edges are wires, and edge weights are resistances.
By construction, every node of $V_1^{(n)}$ is connected by a wire, so that $(V_1^{(n)},W_1^{(n)})$ is a \emph{complete graph}. 
When $i>1$, $(V_i^{(n)}, W_i^{(n)})$ is a \emph{complete bipartite graph} since every node in $V_{i-}^{(n)}$ (the left half interval) is connected to every node in $V_{i+}^{(n)}$ (the right half interval), but no two nodes in the same half interval are connected, see for example second and third columns in Figure \ref{fig:networks}.
\end{rem}

%%%%%%%%%%%%%%%%%%%%%%%%%%%%%%%%%%%%%
\subsection{Notation}
%%%%%%%%%%%%%%%%%%%%%%%%%%%%%%%%%%%%%

Let us concisely summarize the above notation. 

\begin{minipage}{\textwidth}
\begin{multicols}{3}

\noindent
\emph{Index space}.
\begin{itemize}[leftmargin=10pt]
\item $(\N,E)$: directed graph
\item $e$, $e_{i,j}$: edge 
\item $\sigma$: path
\item $E_i$: edges from $i$
\item $E_{i,j}$: edges from $i$ to $j$
\item \mbox{$E_{i}^{(n)}$: paths of length $n$ from $i$}
\item $r_e$: edge weights 
\end{itemize}
\vfill
\hphantom{ghost}

\columnbreak

\noindent
\emph{Physical spaces}.
\begin{itemize}[leftmargin=10pt]
\item $\phi_e$, $\phi_\sigma$: affine maps
\item $V_i^{(n)}$: discrete approx.
\item $V_{i\pm }^{(n)}$: partition of $V_i^{(n)}$ 
\item $\overline{V}_i^{*}$: physical space
\end{itemize}
\vfill
\hphantom{ghost}

\columnbreak

\noindent
\emph{Electric networks}. 
\begin{itemize}[leftmargin=10pt]
\item $(V_i^{(n)}, W_i^{(n)})$: networks
\item $\delta_{\sigma}$: resistance
\end{itemize}
\hphantom{ghost}
\end{multicols}
\end{minipage}

%%%%%%%%%%%%%%%%%%%%%%%%%%%%%%%%%%%%%%%%%%%%%%%%
\section{Approximating energies}\label{sec:discrete}
%%%%%%%%%%%%%%%%%%%%%%%%%%%%%%%%%%%%%%%%%%%%%%%

In this section, we use the networks $(V_i^{(n)},W_i^{(n)})$ to define associated discrete energies that will give rise in Section~\ref{sec:limit} to a nonlocal energy on the physical spaces $\overline{V}_i^{*}$, and in particular on the unit interval when $i=1$. 
While the construction is inspired by the graph-directed approach of Hambly-Nyberg~\cite{Hambly-Nyberg} in the context of local energies on fractals, there are fundamental differences, as has been pointed out in Section~\ref{subsec:GD_construction}. 

%%%%
\subsection{Defining discrete energies}
%%%%

For any $i,n\in\mathbb{N}$, let $\ell(V_i^{(n)})$ denote the set of functions on the nodes of the physical space network, i.e.,
\begin{equation*}
    \ell(V_i^{(n)}):=\{u\colon V_i^{(n)}\to\mathbb{R}\}.
\end{equation*}
The basic building block to define an energy on a finite network is to start with the energy defined on the network $(V_i^{(0)},W_i^{(0)})$, given by 
\begin{equation*}%\label{eq:initial-energy}
    \mathcal{E}_i^{(0)}(u):= (u(1)-u(0))^2, \qquad 
    u\in\ell(V_i^{(0)}) = \ell(\{0,1\}).
\end{equation*}
As explained in Section~\ref{sec:electrical-networks}, each wire in a generic level $n$ is determined by a path $\sigma \in E_i^{(n)}$ and has an associated resistance $\delta_\sigma$. To define the energy associated with the network $(V_i^{(n)},W_i^{(n)})$, we make use of the resistances $\delta_\sigma$ and the mappings $\phi_\sigma$ from~\eqref{E:def_phi_e}. 

\begin{defn}
For each $i,n \in \N$, the energy $(\mathcal{E}^{(n)}_i,\ell(V_i^{(n)}))$ associated with the resistance network $(V_i^{(n)},W_i^{(n)})$ is given by 
\begin{equation}\label{E:def_finite_energy}
\mathcal{E}^{(n)}_i(u)
    := \sum_{\sigma \in E_i^{(n)}} \delta_\sigma^{-1} \mathcal{E}_{\sigma(n)}^{(0)}(u \circ \phi_\sigma),\qquad u\in\ell(V_i^{(n)}).
\end{equation}
\end{defn}

\begin{rem}\label{R:energy_n_alt}
One can use Lemma~\ref{lem:path sets} to rewrite the energy~\eqref{E:def_finite_energy}  
as
\begin{equation*}%\label{E:energy_1n_alt}
    \mathcal{E}_1^{(n)}(u)=\sum_{j=1}^{2^n}\sum_{\sigma\in E_{i,j}^{(n)}}\delta_\sigma^{-1}\mathcal{E}_j^{(0)}(u\circ\phi_\sigma),\qquad u\in\ell(V_1^{(n)})
\end{equation*}
and 
\begin{equation*}%\label{E:energy_in_alt}
    \mathcal{E}_i^{(n)}(u)=\sum_{j=2^n(i-2)+2}^{2^ni}\sum_{\sigma\in E_{i,j}^{(n)}}\delta_\sigma^{-1}\mathcal{E}_j^{(0)}(u\circ\phi_\sigma),\qquad u\in\ell(V_i^{(n)}), \quad i \geq 2.
\end{equation*}
\end{rem}

In view of the definition of the resistances in \eqref{E:def_resistance_wire}, 
we can also represent the energy recursively with what we call a graph-directed self-similar structure. This is consistent with Remark \ref{rem:copies}.

\begin{lem}\label{L:graph_directed_ss_energy}
For $i,n \geq 1$, it holds that
\begin{equation}\label{E:graph_directed_ss_energy}
\mathcal{E}^{(n)}_i(u)
    = \sum_{j=2i-2}^{2i} \sum_{e \in E_{ij}} r_e^{-1}\mathcal{E}_j^{(n-1)}(u \circ \phi_e),\quad u\in\ell(V_i^{(n)}).
\end{equation}
\end{lem}
\begin{proof}
Note first that any $\sigma\in E_i^{(n)}$ may be written as $\sigma=e_{i,\sigma(1)}\tilde{\sigma}$, where $\tilde{\sigma}\in E_{\sigma(1)}^{(n-1)}$ satisfies $\tilde{\sigma}(k)=\sigma(k+1)$ for each $k=0,\ldots,n-1$. Moreover, by construction, $\sigma(1)\in\{2i-2,2i-1,2i\}$, 
\begin{align*}
    \mathcal{E}^{(n)}_i(u)
    =\sum_{\sigma\in E_i^{(n)}}\delta_\sigma^{-1}\mathcal{E}_{\sigma(n)}^{(0)}(u\circ\phi_\sigma)
    &=\sum_{j=2i-1}^{2i}\sum_{e\in E_{i,j}}\sum_{\tilde{\sigma}\in E_j^{(n-1)}}\delta_{e\tilde{\sigma}}^{-1}\mathcal{E}_{\tilde{\sigma}(n-1)}^{(0)}(u\circ\phi_{e\tilde{\sigma}})\\
    &=\sum_{j=2i-1}^{2i}\sum_{e\in E_{i,j}}r_e^{-1}\sum_{\tilde{\sigma}\in E_j^{(n-1)}}\delta_{\tilde{\sigma}}^{-1}\mathcal{E}_{\tilde{\sigma}(n-1)}^{(0)}(u\circ\phi_e\circ\phi_{\tilde{\sigma}})\\
    &=\sum_{j=2i-1}^{2i}\sum_{e\in E_{i,j}}r_e^{-1}\mathcal{E}_j^{(n-1)}(u\circ\phi_e),
\end{align*}
as we wanted to prove.
\end{proof}

%%%%%%%%%%%%%%%%%%%%%%%%%%%%%%%%%%%%%%%%%%%%%%%%
\begin{rem}%\label{R:gdss_alt}
The same relationship~\eqref{E:graph_directed_ss_energy} can be alternatively obtained 
with the expression of the energy from Remark~\ref{R:energy_n_alt}. In the case $i=1$, writing $\sigma=e_{1,k}\tilde{\sigma}$ for any $\sigma\in E_{1,j}^{(n)}$,
\begin{equation}\label{E::gdss_alt_1}
\begin{split}
    \mathcal{E}_1^{(n)}(u)&=\sum_{j=1}^{2^n}\sum_{\sigma\in E_{i,j}^{(n)}}\delta_\sigma^{-1}\mathcal{E}_j^{(0)}(u\circ\phi_\sigma)\\
    &=\sum_{j=1}^{2^n}\sum_{\tilde{\sigma}\in E_{1,j}^{(n-1)}}\delta_{e_{1,1}\tilde{\sigma}}^{-1}\mathcal{E}_j^{(0)}(u\circ\phi_{e_{1,1}\tilde{\sigma}})+
    \sum_{j=1}^{2^n}\sum_{\tilde{\sigma}\in E_{1,j}^{(n-1)}}\delta_{e'_{1,1}\tilde{\sigma}}^{-1}\mathcal{E}_j^{(0)}(u\circ\phi_{e'_{1,1}\tilde{\sigma}})\\
    &\quad+\sum_{j=1}^{2^n}\sum_{\tilde{\sigma}\in E_{1,j}^{(n-1)}}\delta_{e_{2,1}\tilde{\sigma}}^{-1}\mathcal{E}_j^{(0)}(u\circ\phi_{e_{2,1}\tilde{\sigma}}).
\end{split}
\end{equation}
By virtue of Lemma~\ref{lem:path sets}, there are no paths of length $n-1$ starting at $i=1$ that end at $i=2^{n-1}$, whence $E_{1,j}^{(n-1)}$ is empty for $j>2^{n-1}$. Similarly, no backtracking implies $E_{2,j}^{(n-1)}$ is empty if $j<2$. Therefore,~\eqref{E::gdss_alt_1} becomes
\begin{align*}
    &r_{1,1}^{-1}\sum_{j=1}^{2^{n-1}}\sum_{\tilde{\sigma}\in E_{1,1}^{(n-1)}}\delta_{\tilde{\sigma}}^{-1}\mathcal{E}_j^{(0)}(u\circ\phi_{e_{1,1}}\circ\phi_{\tilde{\sigma}})+r_{1,1}^{-1}\sum_{j=1}^{2^{n-1}}\sum_{\tilde{\sigma}\in E_{1,j}^{(n-1)}}\delta_{\tilde{\sigma}}^{-1}\mathcal{E}_j^{(0)}(u\circ\phi_{e'_{1,1}}\circ\phi_{\tilde{\sigma}})\\
    &+r_{1,2}^{-1}\sum_{j=2}^{2^{n-1}\cdot 2}\sum_{\tilde{\sigma}\in E_{2,j}^{(n-1)}}\delta_{\tilde{\sigma}}^{-1}\mathcal{E}_j^{(0)}(u\circ\phi_{e_{1,1}}\circ\phi_{\tilde{\sigma}})
    =\sum_{e\in E_1^{(1)}}r_e^{-1}\mathcal{E}_{e(1)}^{(n-1)}(u\circ\phi_e)\\
    &=\sum_{j=1}^2\sum_{e\in E_{i,j}}r_{1,j}^{-1}\mathcal{E}_j^{(n-1)}(u\circ\phi_e).
\end{align*}
Similar computations lead to~\eqref{E:graph_directed_ss_energy} for any $i\geq 2$. 
\end{rem}

\subsection{Dirichlet forms and discrete jump kernels}%\label{sec:DF-jumps}

The discrete energies defined in the previous section give rise here to Dirichlet forms, which will later on approximate nonlocal energies on the physical spaces $\overline{V}_i^{*}$. We refer the reader to~\cite{Keller-Lenz-Wojciechowski} for basic results and details regarding Dirichlet forms on graphs.

\medskip

To properly regard the energy $(\mathcal{E}_i^{(n)},\ell(V_i^{(n)}))$ as a discrete Dirichlet form with an associated jump kernel, we first equip $V_i^{(n)}$ with the following measure: Fix $i \in \N$ and $n \geq 0$ and consider the partition 
\[
\overline{V}_i^{*} = \bigcup_{x \in V_i^{(n)}} \overline{U_i^{(n)}(x)} 
\]
where, for $x \in V_i^{(n)}$, 

\begin{equation}\label{eq:partiton}
\begin{aligned}
\hbox{for}~i=1, \quad U_1^{(n)}(x) 
    &:= \bigg[x - \frac{1}{2^{n+1}}, x + \frac{1}{2^{n+1}}\bigg) \cap [0,1]\\[.75em]
\quad \hbox{and, for}~i>1, \quad U_i^{(n)}(x) 
   & := \begin{cases} 
        \displaystyle \bigg[x, x+ \frac{1}{2^ni}\bigg) & \hbox{if}~x < \frac{1}{i} \\[1em]
        \displaystyle \bigg(x - \frac{1}{2^n i}, x \bigg] & \hbox{if}~x > 1-\frac{1}{i}.
    \end{cases}
    \quad \quad \quad \mbox{}
\end{aligned}
\end{equation}
Notice that the intervals $U_1^{(n)}(x)$ are centered at points in $x \in V_1^{(n)} \setminus \{0,1\}$ while, for $i>1$, the intervals $U_i^{(n)}(x)$ use $x \in V_i^{(n)}$ as endpoints. 
For $x \in V_i^{(n)}$, we have that 
\begin{align*}
\hbox{for}~i=1, \quad |U_1^{(n)}(x)| &= \begin{cases} \displaystyle\frac{1}{2^n} & \hbox{if}~x \in V_1^{(n)} \setminus \{0,1\} \\[.5em]
\displaystyle\frac{1}{2^{n+1}} &\hbox{if}~x \in  \{0,1\}
\end{cases}\\[.5em]
 \hbox{and, for}~i>1, \quad
|U_i^{(n)}(x)| &= \frac{1}{2^ni}
\end{align*}
where in this context $|\cdot|$ denotes the one-dimensional Lebesgue measure. 
On the set $V_i^{(n)}$, we define the measure $\mu_i^{(n)}$ to be 
\begin{equation}\label{eq:discrete-measure}
\mu_i^{(n)}(B)  
    := \sum_{x\in V_i^{(n)}}\mathbf{1}_B(x) |U_i^{(n)}(x)| \quad \hbox{where}~B \subset V_i^{(n)},
\end{equation}
where $\mathbf{1}_B$ denotes the characteristic function of the set $B$.
Consequently, 
\[
\mu_i^{(n)}(V_i^{(n)}) = |\overline{V}_i^{*}| = \begin{cases}
1 & \hbox{if}~i=1 \\
\frac{2}{i} & \hbox{if}~i>1. 
\end{cases}
\]

With a measure now at hand, we view $(V_i^{(n)},\mu_i^{(n)})$ as a discrete measure space and define a (weighted) graph in the sense of~\cite{Keller-Lenz-Wojciechowski}*{Definition 1.1} by introducing the weight function $j^{(n)}_i\colon V_i^{(n)} \times V_i^{(n)}\to [0,\infty]$ given by
\begin{equation}\label{E:def_jump_kernel}
j^{(n)}_i(x,y):=
\begin{cases}(\delta_{\sigma_{x,y}^{(n)}}\mu_i^{(n)}(x)\mu_i^{(n)}(y))^{-1}
&\hbox{if}~\{x,y\} \in W_i^{(n)}\\
0 & \hbox{otherwise},
\end{cases}
\end{equation}
where 
\begin{equation}\label{eq:path-notation}
\sigma_{x,y}^{(n)} \in E_i^{(n)}
\end{equation}
denotes the unique path of length $n$ associated with $\{x,y\} \in W_i^{(n)}$. 
In particular, $x = \phi_{\sigma_{x,y}^{(n)}}(0)$ and $y = \phi_{\sigma_{x,y}^{(n)}}(1)$.
With this notation, the discrete energy can now be written as
\begin{equation}\label{E:energy_as_DF}
    \mathcal{E}_i^{(n)}(u)=\sum_{x,y\in V_i^{(n)}}(u(x)-u(y))^2j_i^{(n)}(x,y)\mu_i^{(n)}(x)\mu_i^{(n)}(y),
\end{equation}
and we will therefore refer to the function $j^{(n)}_i(x,y)$ as the \emph{jump kernel} associated with $(\mathcal{E}_i^{(n)},\ell(V_i^{(n)}))$. 

\begin{rem}\label{rem:split}
Note that in the case $i>1$, we can write the energy as
\begin{align*}
\mathcal{E}_i^{(n)}(u)
    &= 2\sum_{x\in V_{i-}^{(n)}} \,\, \sum_{y\in V_{i+}^{(n)}} (u(x)-u(y))^2j_i^{(n)}(x,y)\mu_i^{(n)}(y)\mu_i^{(n)}(x)
\end{align*}
because $j_i^{(n)}(x,y) =0$ when $x,y \in V_{i-}^{(n)}$ or $x,y \in V_{i+}^{(n)}$.
\end{rem}

\begin{lem}%\label{lem:discrete dirichlet form}
For each fixed $i, n \in \N$, the pair $(\mathcal{E}_i^{(n)},\ell(V_i^{(n)}))$ defines a Dirichlet form. 
\end{lem}

\begin{proof}
Since $V_i^{(n)}$ is a finite set, for any fixed $x\in V_i^{(n)}$ and $n \in \N$
\[
    \sum_{y\in V_i^{(n)}}j_i^{(n)}(x,y)\mu_i^{(n)}(y)<\infty.
\]
Therefore, applying~\cite{Keller-Lenz-Wojciechowski}*{Proposition 1.14} with $b(x,y)=j_i^{(n)}(x,y)\mu_i^{(n)}(y)$ the claim follows. 
\end{proof}

We will define nonlocal energies on the physical spaces $\overline{V}_i^{*}$ by taking an appropriate limit of the sequences $(\mathcal{E}_i^{(n)},\ell(V_i^{(n)}))$ as $n \to \infty$. Sections~\ref{sec:limit} and \ref{sec:mosco} will discuss that convergence at the level of Dirichlet forms. 

\begin{rem}%\label{R:unbdd_degree}
In contrast to results regarding the convergence of Dirichlet forms on finite graphs appearing in the literature, see e.g.~\cites{Hambly-Nyberg,CKK,Hybrids}, the associated graphs in our setting do \emph{not} have uniformly bounded (unweighted) degree in $n$. Indeed, the kernel (weight) $j_i^{(n)}(x,y)$ from~\eqref{E:def_jump_kernel} is non-zero anytime $x$ is wired to $y$ (i.e.~they are neighbors in the graph sense) and the number of wires attached to a given node $x$ grows with the level $n$, c.f.~\eqref{eq:wires}.
\end{rem}

%%%
\subsection{Compatibility}\label{subsec: compat}
%%%

We now relate the behavior of the Dirichlet forms  $\{(\mathcal{E}_i^{(n)},\ell(V_i^{(n)}))\}_{n\geq 0}$ to the sequence of networks $\{(V_i^{(n)},W_i^{(n)})\}_{n\geq 0}$ using the notion of effective resistance.  

\begin{defn}\label{defn:ER}
  The \textit{effective resistance} between $x,y \in V_i^{(n)}$ is given by
\[
    R_i^{(n)}(x,y)=\sup\Big\{\frac{|u(x)-u(y)|^2}{\mathcal{E}_i^{(n)}(u,u)}\colon u \in \ell(V_i^{(n)}),~\mathcal{E}_i^{(n)}(u,u)>0\Big\}.
\]  
\end{defn} 
It is well-known that the effective resistance is a metric on $V_i^{(n)}$. Moreover, it provides a useful tool for comparing electrical networks.

\begin{defn}\label{defn:EE}
Let $i\in\N$ and $n>k \geq 0$. 
The networks $(V_i^{(n)}, W_i^{(n)})$ and $(V_i^{(k)}, W_i^{(k)})$ are \emph{electrically equivalent} if  $V_i^{(k)} \subseteq V_i^{(n)}$ and
\[
R_i^{(n)}(x,y)=R_i^{(k)}(x,y) \quad \hbox{for any}~x,y\in V_i^{(k)}.
\]
We say the sequence $\{(V_i^{(n)}, W_i^{(n)})\}_{n \geq 0}$ 
of networks
is  electrically equivalent if $(V_i^{(n)}, W_i^{(n)})$ and $(V_i^{(n-1)}, W_i^{(n-1)})$ are electrically equivalent for all $n \geq 1$.
\end{defn}

Given an electrical network $(V_i^{(n)}, W_i^{(n)})$, one can apply the standard network reduction rules (e.g.,~series, parallel, delta-wye) to find electrically equivalent networks. For \emph{any} two nodes in $(V_i^{(n)}, W_i^{(n)})$ it is always possible to find an electrically equivalent network consisting solely of those two nodes connected with a single wire. The resistance on that wire is precisely the effective resistance between the original nodes. Indeed, the effective resistance is sometimes defined in this manner. 

The notion of electrically equivalent networks is closely connected to the concept of compatible energies.

\begin{defn}\label{defn:Compatible}
Let $n,i \geq 1$. The energy forms $(\mathcal{E}_i^{(n)},\ell(V_{i}^{(n)}))$ and $(\mathcal{E}_i^{(n-1)},\ell(V_{i}^{(n-1)}))$ are \emph{compatible} if $V_{i}^{(n-1)} \subseteq V_{i}^{(n)}$ and
\begin{equation*}%\label{eq:compatibility}
\mathcal{E}_i^{(n-1)}(u,u) = \min \{\mathcal{E}_i^{(n)}(v,v): v \in \ell(V_i^{(n)})~\hbox{and}~v \big|_{V_{i}^{(n-1)}} = u\}.
\end{equation*}
We say the sequence $\{(\mathcal{E}_i^{(n)},\ell(V_{i}^{(n)}))\}_{n \geq 0}$ of energy forms is compatible if 
$(\mathcal{E}_i^{(n)},\ell(V_{i}^{(n)}))$ and $(\mathcal{E}_i^{(n-1)},\ell(V_{i}^{(n-1)}))$ are compatible for all $n \geq 1$. 
\end{defn}

The sequence of energy forms $\{(\mathcal{E}_i^{(n)},\ell(V_i^{(n)}))\}_{n\geq 0}$ are compatible precisely when 
the sequence of networks 
 $\{(V_i^{(n)},W_i^{(n)})\}_{n\geq 0}$ are electrically  equivalent. In this case, one can take a limit of $(\mathcal{E}_i^{(n)},\ell(V_i^{(n)}))$ as $n \to \infty$ to find a nonlocal energy on the physical space. 
For more information on equivalent networks, compatibility, and the associated limiting process, we refer the reader to ~\cites{Kigami01,Tetali,DoyleSnell}.

\begin{lem}\label{lem:compat-1}
Fix $i \geq 1$. 
The networks $(V_i^{(0)}, W_i^{(0)})$ and $(V_i^{(1)}, W_i^{(1)})$ are electrically equivalent if and only if 
the triplet $(r_{i, 2i-2}, r_{i, 2i-1}, r_{i, 2i})$ satisfies the matching condition
\begin{equation}\label{eq:matching}
1 = \frac{(r_{i, 2i-2} + 2r_{i, 2i-1})r_{i,2i}}{r_{i, 2i-2} + 2r_{i, 2i-1}+r_{i,2i}}, \quad i \geq 1,
\end{equation}
where we use the convention $r_{1,0} = 0$. 
\end{lem}

\begin{proof}
Electrical equivalence can be determined by  following the  usual series and parallel reduction techniques as illustrated in Figure \ref{fig:stage 1 reduction} with 
\begin{equation}\label{eq:greek}
(\alpha, \beta, \gamma)= (r_{i,2i-2}, r_{i, 2i-1}, r_{2,2i}), \quad \hbox{where}~\alpha=0~\hbox{for}~i=1,
\end{equation}
to renormalize the resistances.

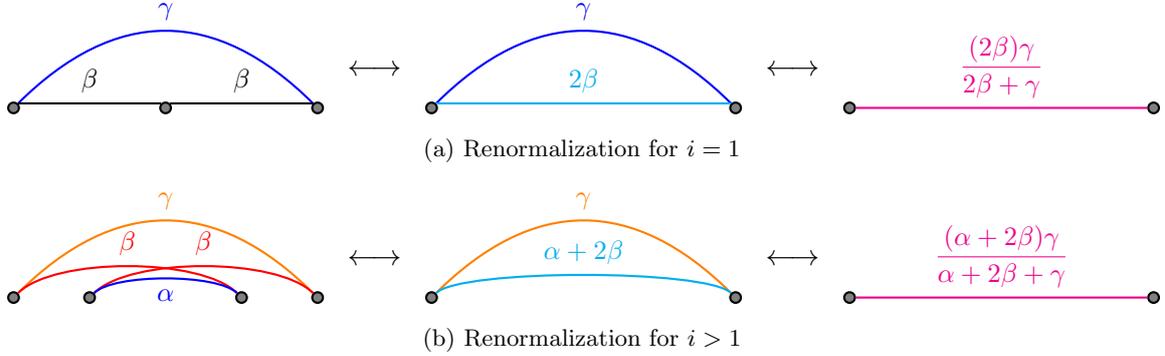
\begin{figure}[H]
\begin{center}
\subfloat[Renormalization for $i=1$]{
\label{fig:reduction-1}
\begin{tikzpicture}[thick,scale=1, main/.style ={circle, draw, fill=black!50,
                        inner sep=0pt, minimum width=4pt}]
%%%first graph
\node[main]  (0) at (0,0) {};
\node[main]  (2) at (2,0) {};
\node[main]  (4) at (4,0) {};
\path[blue]
	(0) edge[looseness=1.2] node[above] {\small $\gamma$} (4);
\path[black]
	(0) edge[looseness=0] node[above] {\small $\beta$} (2);
\path[black]
	(2) edge[looseness=0] node[above] {\small $\beta$} (4);
%%%second graph
\node[main]  (0) at (5.5,0) {};
% \node[main]  (1) at (1,0) {};
% \node[main]  (3) at (3,0) {};
\node[main]  (4) at (9.5,0) {};
\path[blue]
	(0) edge[looseness=1.2] node[above] {\small $\gamma$} (4);
\path[cyan]
	(0) edge[looseness=0] node[above] {\small $2\beta$} (4);
%%%third graph
\node[main]  (0) at (11,0) {};
\node[main]  (4) at (15,0) {};
\node at (0,-.1) {};
\draw[magenta] (0) -- node[above] {\small $\displaystyle \frac{(2 \beta)\gamma}{2 \beta + \gamma}$} ++ (4);
%%%arrows
\node at (4.75,.5) {$\longleftrightarrow$};
\node at (10.25,.5) {$\longleftrightarrow$};
\end{tikzpicture} 
}
\end{center}

\subfloat[Renormalization for $i>1$]{
\label{fig:reduction-i}
\begin{tikzpicture}[thick,scale=1, main/.style ={circle, draw, fill=black!50,
                        inner sep=0pt, minimum width=4pt}]
%%%first graph
\node[main]  (0) at (0,0) {};
\node[main]  (1) at (1,0) {};
\node[main]  (3) at (3,0) {};
\node[main]  (4) at (4,0) {};
\path[orange]
	(0) edge[looseness=1.2] node[above] {\small $\gamma$} (4);
\path[red]
	(0) edge[looseness=.6] node[above] {\small $\beta$} (3);
\path[red]
	(1) edge[looseness=.6] node[above] {\small $\beta$} (4);
\path[blue]
	(1) edge[looseness=.5] node[below] {\small $\alpha$} (3);
%%%second graph
\node[main]  (0) at (5.5,0) {};
% \node[main]  (1) at (1,0) {};
% \node[main]  (3) at (3,0) {};
\node[main]  (4) at (9.5,0) {};
\path[orange]
	(0) edge[looseness=1.2] node[above] {\small $\gamma$} (4);
\path[cyan]
	(0) edge[looseness=.3] node[above] {\small $\alpha + 2\beta$} (4);
%%%third graph
\node[main]  (0) at (11,0) {};
\node[main]  (4) at (15,0) {};
\node at (0,-.1) {};
\draw[magenta] (0) -- node[above] {\small $\displaystyle \frac{(\alpha + 2 \beta)\gamma}{\alpha + 2 \beta + \gamma}$} ++ (4);
%%%arrows
\node at (4.75,.5) {$\longleftrightarrow$};
\node at (10.25,.5) {$\longleftrightarrow$};
\end{tikzpicture}
}
\caption{
Renormalization of resistances
from $(V_i^{(1)}, W_i^{(1)})$ to $(V_i^{(0)}, W_i^{(0)})$}
\label{fig:stage 1 reduction}
\end{figure} 

Indeed for $i>1$, we first ``untwist'' as in Figure \ref{fig:reduction-i} to see that the wire with resistance $\alpha$ and the two wires with resistances $\beta$ are connected in series.
After reducing, the remaining two wires are connected in parallel. 
Since the wire $\{0,1\}$ has resistance 1 at stage $n=0$, electrical equivalence is characterized by \eqref{eq:matching}.
From Figure \ref{fig:reduction-1}, the case $i=1$ is similar.   
\end{proof}

\begin{rem}
The matching condition \eqref{eq:matching} only guarantees electrical equivalence from stage 0 to stage 1. 
However, due to the high degree of connectedness of our network (complete for $i=1$ and compete bipartite for $i>1$), we cannot iterate the same renormalization steps outlined in Figure \ref{fig:stage 1 reduction} to obtain explicit expressions for the resistances ensuring that $(V_i^{(n-1)},W_i^{(n-1)})$ and 
$(V_i^{(n)},W_i^{(n)})$ are electrically equivalent for $n \geq 2$. 
This problem is highly nontrivial even in the local context; we refer to~\cites{Kigami01,Str06} for some examples where the reduction is possible. 
Nevertheless, we know appropriate resistances exist, and we discuss this further in Appendix \ref{sec:appendixA}.
\end{rem}

\begin{rem}
Notice that the matching condition \eqref{eq:matching} is the same for all $i>1$. 
Thus, additional constraints may be imposed to distinguish $(V_i^{(n)}, W_i^{(n)})$ from $(V_j^{(n)}, W_j^{(n)})$ for $i \not= j$ and $n \geq 1$. 
For example, one might take $r_{i,2i-2}< r_{i+1, 2(i+1)-2}$ for all $i \in \N$.
\end{rem}

From \eqref{eq:matching}, if $r_{i,2i-2}+2r_{i,2i-1}=1$, then $r_{i,2i}=+\infty$.  
For $i=1$, this means that $r_{1,1} = \frac12$ and $r_{1,2}=+\infty$ which corresponds to the standard discrete approximation of the classical local energy on the unit interval, see e.g.~\cite{Str06}*{Section 1.3}. 
In our setting, we assume that
\begin{equation}\label{eq:not1}
r_{i,2i-2}+2r_{i,2i-1}\not=1 \quad \hbox{and} \quad r_{i,2i}\not=1. 
\end{equation}

We draw the reader's attention to the fundamental 
challenging novelties of our setting: 
In general, for the construction of \emph{local} energies by discrete approximation, the resistances attached to the wires are strictly less than one, see Appendix \ref{sec:appendixA}. The intuition is that only ``close-by'' nodes (meaning~the Euclidean distance between them is minimal) are wired together, so that the resistance on the wire between these 
nodes is small. This is the idea of a ``short-range interaction.'' 
However, in the \emph{nonlocal} setting, one has to take on account ``long-range interactions'', in the sense that 
 ``far-away'' nodes 
(meaning the Euclidean distance between them is \emph{not} minimal)
will also be connected by a wire.
Nevertheless, because the nodes are far, the resistance on the wire may be strictly larger than one. We see this in the next result.

\begin{lem}\label{lem:big-r}
If the matching condition \eqref{eq:matching} holds for $i \in \N$, then 
\[
r_{1,1} > \frac12 \quad \hbox{and} \quad r_{1,2} > 1,
\]
and for $i>1$,
\[
r_{i,2i-2} + 2r_{i,2i-1}>1 \quad \hbox{and} \quad r_{i, 2i} >1.
\]
\end{lem}

\begin{proof}
Recall the notation \eqref{eq:greek} from the proof of Lemma \ref{lem:compat-1} and set
$\xi := \alpha + 2 \beta$.
From \eqref{eq:not1}, we have $\xi, \gamma \in (0,\infty) \setminus \{1\}$. 
Observe that \eqref{eq:matching} can be written as
\[
    \frac{\xi\gamma}{\xi+\gamma}=1
    \quad\Leftrightarrow\quad
    \gamma=\frac{\xi}{\xi-1}=1+\frac{1}{\xi-1} 
    \quad \Leftrightarrow \quad 
    \xi = \frac{\gamma}{\gamma-1}= 1 + \frac{1}{\gamma-1}.
\]
If $\xi \in (0,1)$, then $\gamma < 0$, a contradiction. Therefore, $\xi>1$, and consequently $\gamma>1$. 
For $i=1$, the condition $\xi>1$ is equivalent to $2r_{1,1} = 2 \beta >1$. 
\end{proof}

Note in Lemma \ref{lem:big-r} that one could still have $r_{i, 2i-2}, r_{i, 2i-1} \in (0,1)$.

Another distinguishing feature from the classical 
construction is that if $\{x,y\} \in W_i^{(n_0)}$ for some $n_0 \in \N$, then $\{x,y\} \in W_i^{(n)}$ for all $n \geq n_0$. Consequently, we may take a limit of the resistance between two points $x$ and $y$ as $n$ approaches infinity. 
Since our networks are complete/complete bipartite (recall Remark \ref{R:complete_graph}), 
the resistance on the wire $\{x,y\} \in W_i^{(n)}$ wire must become infinitely large as $n$ increases in order to maintain electrical equivalence. 
We see this in the next result, for which we recall the notation \eqref{eq:path-notation}.

\begin{lem}%\label{lem:limit of delta}
Assume that the matching condition \eqref{eq:matching} holds and 
that
\begin{equation}\label{eq:log-sum}
\sum_{k=0}^\infty\log r_{2^ki_0, 2^{k+1}i_0} = \infty~\hbox{for all}~i_0 \in \N. 
\end{equation} 
If $x$ and $y$ are such that $\{x,y\} \in W_i^{(n_0)}$ for some $n_0,i \in \N$, 
then
\[
\lim_{n \to \infty} \delta_{\sigma_{x,y}^{(n)}}^{-1} = 0.
\]
\end{lem}

\begin{proof}
Set $i_0 = \sigma_{x,y}^{(n_0)}(n_0)$.  
Since the Euclidean distance $|x-y|$ is fixed at each stage $n \geq n_0$, we recall \eqref{E:def_phi_e} to determine that
\[
\sigma_{x,y}^{(n)} = \sigma_{x,y}^{(n_0)} \rho^{(n-n_0)}_{i_0}
\]
where $\rho = \rho^{(n-n_0)}_{i_0}$ is a path of length $n-n_0$ satisfying
\[
\rho(0)=i_0 \quad \hbox{and} \quad \rho(k) = 2\rho(k-1) \quad \hbox{for}~1 \leq k \leq n-n_0.
\]
Recursively, we can show that $\rho(k) = 2^{k} i_0$ for all $0 \leq k \leq n-n_0$. 
By definition,
\begin{equation}\label{eq:delta-n0-rho}
\delta_{\sigma_{x,y}^{(n)}} = \delta_{\sigma_{x,y}^{(n_0)}} \delta_{\rho^{(n-n_0)}_{i_0}}
\end{equation}
where 
\[
\delta_{\rho^{(n-n_0)}_{i_0}}
    = \prod_{k=0}^{n-n_0-1} r_{2^ki_0, 2^{k+1}i_0}.
\]
By \eqref{eq:log-sum}, 
\[
\lim_{n \to \infty} \delta_{\rho^{(n-n_0)}_{i_0}}
    = \prod_{k=0}^{\infty} r_{2^ki_0, 2^{k+1}i_0} = \infty
\]
which, together with \eqref{eq:delta-n0-rho}, gives the desired result.  
\end{proof}

We remark that the conclusion of Lemma \ref{lem:big-r} does not necessarily imply \eqref{eq:log-sum}. 
For this one would need additional assumptions on the sequence $\{r_{i, 2i}\}_{i \in \N}$, such as a uniform lower bound greater than one.

%%%%%%%%%%%%%%%%%%%%%%%%%%%%%%%%%%%%%%%%%%%%%%%%
\section{Energy in the limit}\label{sec:limit}
%%%%%%%%%%%%%%%%%%%%%%%%%%%%%%%%%%%%%%%%%%%%%%%%

In this section, we present two ways of taking the limit of energies on the discrete approximations and prove in Theorem~\ref{thm:DF-main} that, under certain assumptions on the discrete jump kernels, both approaches yield the same nonlocal energy on the physical space. 

%%%%%%%%%%%%%
\subsection{Preliminaries}
%%%%%%%%%%%%%

We begin by setting notation. 
For a measurable set $A \subset \R$, let $L^2(A,dx)$ denote the space of $L^2$ functions on $A$ with respect to the 1-dimensional Lebesgue measure. 
For $s \in (0,1)$ and $A_1,A_2 \subset \R$, $H^s(A_1,A_2)$ denotes the class of functions $u \in L^2(A_1\cup A_2, dx)$ such that
\begin{equation}\label{def:Gaglicardo_seminorm}
[u]_{H^s(A_1,A_2)}^2 := \int_{A_1} \int_{A_2} \frac{|u(x) - u(y)|^2}{|x-y|^{1+2s}} \, dy \, dx<\infty. 
\end{equation}
When $A_1=A_2=A$ we simply write $H^s(A)$ and $[u]_{H^s(A)}$ instead of $H^s(A,A)$ and $[u]_{H^s(A,A)}$.   
In this case, note that~\eqref{def:Gaglicardo_seminorm} corresponds to the standard Gagliardo seminorm on the fractional Sobolev space $H^s(A)$. 

Further, let $L^{\infty}(A)$ denote the space of bounded functions on $A \subset \R$. 
For $\alpha \in (0,1)$, the class of $\alpha$-H\"older continuous functions  $C^\alpha(A)$ is the set functions $u \in L^{\infty}(A)$ such that
\[
[u]_{C^\alpha(A)} = \sup_{\substack{x,y \in A \\x\not=y}} \frac{|u(x) - u(y)|}{|x-y|^\alpha} <\infty.
\]
Finally, $\lip(A)$ will denote the class of Lipschitz functions on $A \subset \R$.

\medskip

At the discrete level, for any $i\geq 1$, $n\geq 0$ and $B\subseteq V_i^{(n)}$, we denote by $\ell^2(B, \mu_i^{(n)})$ the set of functions $u:B \to \R$ satisfying 
\[
\|u\|_{\ell^2(B, \,\mu_i^{(n)})}^2 = \sum_{x \in B} |u(x)|^2 \mu_i^{(n)}(x)<\infty,
\]
where $\mu_i^{(n)}$ is the discrete measure defined in~\eqref{eq:discrete-measure}. 
Since the cardinality of $V_i^{(n)}$ is finite, note that $\ell^2(B, \mu_i^{(n)}) = \ell(B)=:\{u:B \to \R\}$. 
We next introduce two operators based on the partition $\{U_i^{(n)}(x)\}_{x\in V_i^{(n)}}$ of $\overline{V}_i^{*}$ into the intervals defined in~\eqref{eq:partiton}.  

\begin{defn} \label{defn:operators}
Let $i \in \N$ and $n \geq 0$.
\begin{enumerate}
\item The extension operator $\ext_i^{(n)}: \ell^2(V_i^{(n)}, \mu_i^{(n)}) \to L^2(\overline{V_i^{*}},dx)$ is defined as 
\[
\ext_i^{(n)} u(x) = u(\bar{x}) \quad \hbox{where}~\bar{x}\in V_i^{(n)}~\hbox{is such that} ~x \in U_i^{(n)}(\bar{x}).
\]
\item The average $\avg_i^{(n)}: L^2(\overline{V_i^{*}},dx)\to \ell^2(V_i^{(n)}, \mu_i^{(n)})$ is defined as
\[
\avg_i^{(n)}u(\bar{x}) = \frac{1}{|U_i^{(n)}(\bar{x})|} \int_{U_i^{(n)}(\bar{x})} u(x) \, dx \quad \hbox{for}~\bar{x} \in V_i^{(n)}. 
\]
\end{enumerate}
\end{defn}

As a brief aside, we note that that $\avg_i^{(n)}$ is the left inverse of $\ext_i^{(n)}$ for a fixed $i \in \N$, $n \geq 0$, namely, for all $v \in \ell^2(V_i^{(n)}, \mu_i^{(n)})$ and $u \in L^2(\overline{V}_i^*, dx)$, 
\[
\avg_i^{(n)}  \ext_i^{(n)} v = v 
\]
and
\[
\langle \avg_i^{(n)} u, v \rangle_{\ell^2(V_i^{(n)}, \mu_i^{(n)})} 
    = \langle  u, \ext_i^{(n)} v \rangle_{L^2(\overline{V_i^{*}},dx)}. 
\]

%%%%%%%%%%%%%%%%%%%%%%%%%%
\subsection{Two approaches} 
%%%%%%%%%%%%%%%%%%%%%%%%%%

We now present two approaches for taking the limit as $n \to \infty$ of the discrete energies $(\mathcal{E}_i^{(n)}, \ell(V_i^{(n)}))$ from \eqref{E:def_finite_energy}. 

\subsubsection*{Finite energy approach.}
Following the standard construction of local energies on p.c.f.~self-similar sets, see e.g.~\cite{Kig89}, 
we define for each $i \in \N$ the energy form $(\mathcal{E}_i^{(\infty)}, \mathcal{F}_i^{(\infty)})$ by setting
\begin{equation}\label{E:energy_limit_fin_energy}
\begin{aligned}
&\mathcal{E}_i^{(\infty)}(u):=\lim_{n\to\infty}\mathcal{E}_i^{(n)}(u|_{V_i^{(n)}})\\
&\mathcal{F}_i^{(\infty)}:=\big\{u\in L^2(\overline{V}_i^*,dx)\colon \mathcal{E}_i^{(\infty)}(u)<\infty\big\},
\end{aligned}
\end{equation}
where $u |_{V_i^{(n)}}$ denotes the usual restriction of $u$ to the set $V_i^{(n)}$. The main advantage of this \emph{finite energy approach} is the straightforward definition of both the functional and the domain. 
However, common tools such as those in~\cite{Kigami01} to analyze the limiting form $(\mathcal{E}_i^{(\infty)}, \mathcal{F}_i^{(\infty)})$, and in particular to verify its closability, cannot be practically applied in our setting.

%%%%%%%%%
\subsubsection*{Density approach}
%%%%%%%%%

To simplify the question of closability, one can follow an approach discussed by Chen, Kim, and Kumagai~in~\cite{CKK}, and define for each $i \in \N$ an energy $(\widetilde{\mathcal{E}}_i^{(\infty)}, \widetilde{\mathcal{F}}_i^{(\infty)})$ by setting
\begin{equation}\label{E:energy_limit_Lip}
\begin{aligned}
&\widetilde{\mathcal{E}}_i^{(\infty)}(u):=\lim_{n\to\infty}\mathcal{E}_i^{(n)}(\avg_i^{(n)}u)\\[.5em]
&\widetilde{\mathcal{F}}_i^{(\infty)}:=
\begin{cases}
\overline{{\rm Lip}([0,1])}^{\widetilde{\mathcal{E}}} & \hbox{if}~i=1,\\[.5em]
\overline{C^{\frac{1}{2}}([0,1])}^{\widetilde{\mathcal{E}}} & \hbox{if}~i = 2,\\[.5em]
\overline{L^{\infty}(\overline{V_i^*})}^{\widetilde{\mathcal{E}}} & \hbox{if}~i>2,
\end{cases}
\end{aligned}
\end{equation}
where the latter closure is taken with respect to $\widetilde{\mathcal{E}}=(\widetilde{\mathcal{E}}_i^{(\infty)}(\cdot)+\|\cdot\|_{L^2(\overline{V_i^*},dx)}^2)^{1/2}$.
In this \emph{density approach},
the definition of the energy and its domain are less straightforward than in the finite energy approach.

%%%%%%%%%%
\subsection{Connecting the two approaches}
%%%%%%%%%%

In the next paragraphs, we will show that both constructions give rise to the same energy in the limit under reasonable assumptions on the jump kernels.

%%%%%%%%%%%%%%%%%%%%%%%%%%%%%%%%%%%%%%%%%%%%%%%%
\subsubsection{Main assumptions}%\label{sec:jump}
%%%%%%%%%%%%%%%%%%%%%%%%%%%%%%%%%%%%%%%%%%%%%%%%
For the remainder of the paper, we will make the following 
assumptions on the jump kernels $j_i^{(n)}(x,y)$ defined in~\eqref{E:def_jump_kernel}.

\begin{assumption}\label{A:kernel}
Fix $s \in (0,1)$ and let $i \in \N$. 
\begin{enumerate}[label=(\alph*)]
\item \label{item:upper} (Upper bound)
There exists $\Lambda_i>0$ such that, for all $n\geq 0$ and $\{x,y\} \in W_i^{(n)}$,
\begin{equation}\label{eq:kernel-estimate-new}
0 < j_i^{(n)}(x,y) \leq \Lambda_i|x-y|^{-(1+2s)}.
\end{equation}
\item \label{item:lower} (Lower bound)
There exists $\lambda_i>0$  such that, for all $n\geq 0$ and $\{x,y\} \in W_i^{(n)}$,
\begin{equation}\label{eq:kernel-lower-new}
\lambda_i |x-y|^{-(1+2s)} \leq j_i^{(n)}(x,y).
\end{equation}
\end{enumerate}
\end{assumption}

\begin{rem}\label{rem:conductance-bound}
In view of~\eqref{E:def_jump_kernel}, Assumption \ref{A:kernel} for $\{x,y \} \in W_i^{(n)}$ is equivalent to 
\[
\frac{\lambda_i}{|x-y|^{1+2s}}\mu_i^{(n)}(x) \mu_i^{(n)}(y)
\leq 
\delta_{\sigma_{x,y}^{(n)}}^{-1} \leq \frac{\Lambda_i}{|x-y|^{1+2s}}\mu_i^{(n)}(x) \mu_i^{(n)}(y). 
\]
\end{rem}

\begin{rem} 
Our results under Assumption \ref{A:kernel} also hold for more regular kernels with powers $\alpha\in (0,1]$. 
The power $1+2s$ corresponds to the kernel for the fractional Laplacian in $\R$. 
Fractional discrete Laplacians on the infinite lattice $h \Z$, $h>0$, were studied in \cite{Stinga}, where it was shown that the discrete jump kernel $K^{(h)}$ satisfies
\[
\frac{\lambda}{|x|^{1+2s}} h \leq K^{(h)}(x) \leq \frac{\Lambda}{|x|^{1+2s}}h, \quad x \in h \Z \setminus \{0\},
\]
for constants $\lambda,\Lambda>0$ depending only on $s \in (0,1)$. 
Our situation corresponds to $h = 2^{-n}= \mu_1^{(n)}(x)$ when $x\not=0,1$ and $i=1$, with the kernel $K^{(h)}$ including the measure. 
We emphasize that our discrete operators are different than those in~\cite{Stinga} as we work in bounded domains and use a graph-directed construction. 
\end{rem}

\begin{rem}\label{rem:nonsingular}
The kernel is not singular in the case $i>2$. Indeed, 
 if $x \in [0,\frac{1}{i}]$ and $y \in [1-\frac{1}{i},1]$, then for all $s \in (0,1)$ 
\begin{equation}\label{eq:boundedkernel}
(y-x)^{-(1+2s)}\leq  \left(1 - \frac{2}{i}\right)^{-(1+2s)} =: b_i < \infty.
\end{equation}
When $i=2$, there is a singularity at $\frac{1}{2}$ since the physical gap between $V_{2-}^{(n)}$ and $V_{2+}^{(n)}$ disappears in the limit.
\end{rem}

\begin{rem}Under Assumption \ref{A:kernel}, we can show that the effective resistance distance is bounded above by the Euclidean distance to a power when $i=1$ and by a geodesic distance if $i>1$. 
A precise statement and proof is given in Appendix \ref{sec:appendixB} (see Lemma \ref{lem:R-euclidean}).
\end{rem}

%%%%%%%%
\subsubsection{Statement of the main result} 
%%%%%%%%

To rigorously connect the two approaches, we first prove that the domains $\mathcal{F}_i^{(\infty)}$ and $\widetilde{\mathcal{F}}_i^{(\infty)}$ coincide for all $i \in \N$ (see Section \ref{subsec:domains}). 
Then, we prove that the forms $\mathcal{E}_i^{(\infty)}$ and $\widetilde{\mathcal{E}}_i^{(\infty)}$ coincide in a dense set of functions (see Section \ref{subsec: forms}). 
In the proofs of these steps, we will explicitly distinguish the cases $i=1$ and $i>1$, and treat $i=2$ separately, pointing out subtle but important differences.
The main result of this section is the following. 

\begin{thm}\label{thm:DF-main}
Let $i \in \N$ be fixed and let Assumption \ref{A:kernel} hold for some fixed $s \in (0,1)$. 
Then, $(\mathcal{E}_i^{(\infty)},\mathcal{F}_i^{(\infty)})=(\widetilde{\mathcal{E}}_i^{(\infty)},\widetilde{\mathcal{F}}_i^{(\infty)})$ where
\begin{equation}\label{eq:domain-description}
\mathcal{F}_i^{(\infty)} = \begin{cases}
    H^s([0,1]) & \hbox{if}~i=1\\[.5em]
    H^s\left(\left[0,\frac{1}{2}\right],\left[\frac{1}{2},1\right]\right)  & \hbox{if}~i=2\\[.5em]
    L^2(\overline{V}_i^*,dx) & \hbox{if}~i>2.
\end{cases}
\end{equation}
\end{thm}

\begin{proof}
For $i \in \N$, the equivalence of domains $\mathcal{F}_i^{(\infty)} = \widetilde{\mathcal{F}}_i^{(\infty)}$ and the description \eqref{eq:domain-description} follows from Lemmas \ref{lem:domains} and \ref{lem:domains-tilde}.
The equivalence of energies $\mathcal{E}_i^{(\infty)}(u) = \widetilde{\mathcal{E}}_i^{(\infty)}(u)$ for all $u \in \mathcal{F}_i^{(\infty)}$ then follows from Lemmas \ref{lem:samelimit}, \ref{lem:samelimit-2}, and \ref{lem:samelimit-i}.   
\end{proof}

\begin{rem}\label{rem:ellipticity}
As a consequence of the proof of Theorem \ref{thm:DF-main}, if $u \in \mathcal{F}_i^{(\infty)}$, then 
\begin{equation}\label{eq:ellipticity}
\begin{array}{rcccll} 
\lambda_1 [u]_{H^s([0,1])}^2 
&\leq& \mathcal{E}_1^{(\infty)}(u) &\leq&\Lambda_1 [u]_{H^s([0,1])}^2 
& \hbox{if}~i=1, 
 \\[1em] 
 2\lambda_2 [u]_{H^s([0,\frac12],[\frac12,1])}^2 &\leq&
\mathcal{E}_2^{(\infty)}(u) 
&\leq& 2\Lambda_2 [u]_{H^s([0,\frac12],[\frac12,1])}^2 
& \hbox{if}~i=2, 
 \\[1em]
0 &\leq&  \mathcal{E}_i^{(\infty)}(u) 
&\leq&  \frac{4\Lambda_i b_i}{i}  \|u\|_{L^2(\overline{V}_i^*)}^2 & \hbox{if}~i>2,
\end{array}
\end{equation}
where $b_i \in (0,\infty)$, $i>2$, is defined in  \eqref{eq:boundedkernel}. 
\end{rem}

%%%%%%%%%%%%%%%%%%%%%%%%%%%%%%%%
\subsection{Equivalence of domains}\label{subsec:domains}
%%%%%%%%%%%%%%%%%%%%%%%%%%%%%%%%

The goal of this section is to prove that $\mathcal{F}_i^{(\infty)} = \widetilde{\mathcal{F}}_i^{(\infty)}$ for all $i \in \N$. 
We begin by proving the following description of $\mathcal{F}_i^{(\infty)}$. 

\begin{lem}\label{lem:domains}
Let $i \in \N$ be fixed and let Assumption~\ref{A:kernel}\ref{item:upper} hold for some $s\in (0,1)$. Then,
\begin{equation}\label{eq:inclusion_domains}
\begin{array}{ll}
\lip([0,1])\subset H^s([0,1])\subset \mathcal{F}_1^{(\infty)}  & \hbox{if}~i=1,\\[.75em]
C^{\frac12}([0,1])\subset H^s([0,\frac12],[\frac12,1])\subset\mathcal{F}_2^{(\infty)} & \hbox{if}~i=2,\\[.75em]
L^{2}(\overline{V}_i^*,dx) =\mathcal{F}_i^{(\infty)} & \hbox{if}~i>2.
\end{array}
\end{equation}
If in addition~Assumption~\ref{A:kernel}\ref{item:lower} holds, then $\mathcal{F}_i^{(\infty)}$ is given by \eqref{eq:domain-description} for all $i \in \N$.
\end{lem}

\begin{proof}

\underline{\bf Step 1}. We begin by proving \eqref{eq:inclusion_domains}. 
First, consider $i=1$. 
By Assumption~\ref{A:kernel}\ref{item:upper} and the definition of Riemann integration, for $u\in H^s([0,1])$,
\begin{align*}
\mathcal{E}_1^{(\infty)}(u)
    &\leq \lim_{n \to \infty}  \Lambda_1 \sum_{x, y \in V_1^{(n)}} \frac{|u(x) - u(y)|^2}{|x-y|^{1+2s}} \, \mu_1^{(n)}(y) \, \mu_1^{(n)}(x)\notag\\
    &= \Lambda_1 \int_0^1 \int_0^1 \frac{|u(x) - u(y)|^2}{|x-y|^{1+2s}} \, dy \, dx = \Lambda_1[u]_{H^s([0,1])}^2. 
\end{align*} 
Moreover, since $\lip([0,1]) \subset H^s([0,1])$, it follows that
\begin{equation*}%\label{eq:inclusions-1}
 \lip([0,1]) \subset
 H^s([0,1]) \subset \mathcal{F}_1^{(\infty)}
\end{equation*}
which proves \eqref{eq:inclusion_domains} for $i=1$. 

For $i\geq 2$, we similarly estimate 
\begin{align*}
\mathcal{E}_i^{(\infty)}(u)
    &\leq \lim_{n \to \infty}  2\Lambda_i \sum_{x \in V_{i-}^{(n)}} \,\, \sum_{y\in V_{i+}^{(n)}} \frac{(u(x)-u(y))^2}{(y-x)^{1+2s}} \mu_i^{(n)}(y)\mu_i^{(n)}(x)\\
    &= 2\Lambda_i \int_0^{\frac{1}{i}} \int_{1-\frac{1}{i}}^1 \frac{|u(x) - u(y)|^2}{(y-x)^{1+2s}} \, dy \, dx.
\end{align*}
For $i=2$, it is easy to check that if $u \in C^{\frac{1}{2}}([0,1])$, then 
\begin{align*}
[u]_{H^s([0,\frac12],[\frac12,1])}^2 =\int_0^{\frac{1}{2}} \int_{\frac{1}{2}}^1 \frac{|u(x) - u(y)|^2}{(y-x)^{1+2s}} \, dy \, dx &\leq \int_0^{\frac{1}{2}} \int_{\frac{1}{2}}^1 \frac{[u]_{C^{\frac{1}{2}}([0,1])}^2(y-x)}{(y-x)^{1+2s}} \, dy \, dx
     <\infty.
\end{align*}
Thus,
\[
C^{\frac12}([0,1]) \subset H^s\left(\left[0,\frac12\right],\left[\frac12,1\right]\right) \subset \mathcal{F}_2^{(\infty)}
\]
and \eqref{eq:inclusion_domains} holds for $i=2$. 
On the other hand, for $i>2$, 
if $u \in L^2(\overline{V}_i^*, dx)$, then~\eqref{eq:boundedkernel} implies
\begin{align*}
2\Lambda_i \int_0^{\frac{1}{i}} \int_{1-\frac{1}{i}}^1 \frac{|u(x) - u(y)|^2}{(y-x)^{1+2s}} \, dy \, dx
    &\leq 2\Lambda_i b_i \int_0^{\frac{1}{i}} \int_{1-\frac{1}{i}}^1 |u(x) - u(y)|^2 \, dy \, dx\\
    &\leq 4\Lambda_i b_i \int_0^{\frac{1}{i}} \int_{1-\frac{1}{i}}^1 (|u(x)|^2 + |u(y)|^2) \, dy \, dx\\
    &= \frac{4\Lambda_i b_i}{i} \|u\|_{L^2(\overline{V}_i^*, dx)}^2.
\end{align*}
Therefore, 
\[
L^2(\overline{V}_i^*, dx) \subset \mathcal{F}_i^{(\infty)} \subset L^2(\overline{V}_i^*, dx) 
\]
and \eqref{eq:inclusion_domains} holds for $i>2$. 

\medskip

\noindent
\underline{\bf Step 2}.
We now turn our attention to 
\eqref{eq:domain-description}.
Let $i=1$. 
By Assumption~\ref{A:kernel}\ref{item:lower}, we can show as above that
\[
\mathcal{E}_1^{(\infty)}(u)
    \geq \lambda_1 \int_0^1 \int_0^1 \frac{|u(x) - u(y)|^2}{|x-y|^{1+2s}} \, dy \, dx = \lambda_1[u]_{H^s([0,1])}^2.
\]
Consequently,
\begin{equation*}
\mathcal{F}_1^{(\infty)} = H^s([0,1])
\end{equation*}
which proves \eqref{eq:domain-description} for $i=1$. 
Similarly, for $i=2$ we find
\[
\mathcal{E}_2^{(\infty)}(u)
    \geq 2\lambda_2 \int_0^{\frac{1}{2}} \int_{\frac{1}{2}}^1 \frac{|u(x) - u(y)|^2}{(y-x)^{1+2s}} \, dy \, dx,
\]
so that
\begin{equation*}
\mathcal{F}_2^{(\infty)} = H^s\left(\left[0,\frac{1}{2}\right],\left[\frac{1}{2},1\right]\right),
\end{equation*}
which is \eqref{eq:domain-description} for $i=2$. 
The result for $i>2$ was already established in \eqref{eq:inclusion_domains}. 
\end{proof}

%%%%%

In the next lemma, we describe the domain $\widetilde{\mathcal{F}}_i^{(\infty)}$ and consequently prove that it coincides with the domain $\mathcal{F}_i^{(\infty)}$. 

\begin{lem}\label{lem:domains-tilde}
Let $i \in \N$ be fixed and let Assumption~\ref{A:kernel}\ref{item:upper} hold for some $s\in (0,1)$. Then,
\begin{equation}\label{eq:inclusion_domains-tilde}
\begin{array}{ll}
\lip([0,1])\subset H^s([0,1])\subset\widetilde{\mathcal{F}}_1^{(\infty)}  
& \hbox{if}~i=1,\\[.75em]
C^{\frac12}([0,1])\subset H^s([0,\frac12],[\frac12,1])\subset\widetilde{\mathcal{F}}_2^{(\infty)} 
& \hbox{if}~i=2,\\[.75em]
L^{2}(\overline{V}_i^*,dx) =\widetilde{\mathcal{F}}_i^{(\infty)} & \hbox{if}~i>2.
\end{array}
\end{equation}
If in addition~Assumption~\ref{A:kernel}\ref{item:lower} holds, 
then
\begin{equation}\label{eq:equal_domains-tilde}
\widetilde{\mathcal{F}}_i^{(\infty)}
    = \begin{cases}
    H^s([0,1])& \hbox{if}~i=1,\\[.75em]
    H^s([0,\frac12],[\frac12,1]) & \hbox{if}~i=2,\\[.75em]
    L^2(\overline{V}_i^*, dx) & \hbox{if}~i>2,
    \end{cases}
\end{equation}
and, consequently, $\mathcal{F}_i^{(\infty)} = \widetilde{\mathcal{F}}_i^{(\infty)}$ for all $i \in \N$.
\end{lem}

\begin{proof}

The proof is broken up into several steps. 

\medskip
\noindent
\underline{\bf Step 1}. We begin by proving \eqref{eq:inclusion_domains-tilde} in the case $i=1$. 
The first inclusion was established in  \eqref{eq:inclusion_domains}, 
hence we will prove that, for $u \in H^s([0,1])$,
\begin{equation}\label{eq:tilde-Hs-1}
\widetilde{\mathcal{E}}_i^{(\infty)}(u) \leq \Lambda_1 [u]_{H^s([0,1])}^2. 
\end{equation}

Fix $\delta>0$ and define for $v \in \ell(V_1^{(n)})$
\begin{equation*}
\mathcal{E}_{1,\delta}^{(n)}(v)
    := \sum_{\bar{x} \in V_1^{(n)}} \sum_{\bar{y} \in V_1^{(n)}} j_1^{(n)}(\bar{x},\bar{y}) \mathbf{1}_{\{|\bar{x} - \bar{y}|>\delta\}} |v(\bar{x}) - v(\bar{y})|^2 \mu_1^{(n)}(\bar{x}) \mu_1^{(n)}(\bar{y}),
\end{equation*}
where $\mathbf{1}_A$ denotes the characteristic function of a measurable set $A \subset \R$. 

By Assumption~\ref{A:kernel}\ref{item:upper}, we get for any $v \in \ell(V_1^{(n)})$
\begin{align*}
\mathcal{E}_{1,\delta}^{(n)}(v)
    &\leq \Lambda_1 \sum_{\bar{x} \in V_1^{(n)}} \sum_{\bar{y} \in V_1^{(n)}} |\bar{x} - \bar{y}|^{-1-2s} \mathbf{1}_{\{|\bar{x} - \bar{y}|>\delta\}} |v(\bar{x}) - v(\bar{y})|^2 \mu_1^{(n)}(\bar{x}) \mu_1^{(n)}(\bar{y})\\
    &= \Lambda_1 \sum_{\bar{x} \in V_1^{(n)}} \sum_{\bar{y} \in V_1^{(n)}}\int_{U_1^{(n)}(\bar{x})}\int_{U_1^{(n)}(\bar{y})} \frac{ |v(\bar{x}) - v(\bar{y})|^2}{|\bar{x} - \bar{y}|^{1+2s}} \mathbf{1}_{\{|\bar{x} - \bar{y}|>\delta\}} \, dy \, dx.
\end{align*}
For $x \in [0,1]$, consider the function
\begin{equation}\label{eq:l}
\avg_1^{(n)}(x)  := \bar{x}
\end{equation}
where $\bar{x} \in V_1^{(n)}$ is the unique point such that $x \in U_1^{(n)}(\bar{x})$. 
With this, we write
\begin{align*}
\mathcal{E}_{1,\delta}^{(n)}(v)
    &\leq \Lambda_1 \int_0^1 \int_0^1 \frac{ |v(\bar{x}) - v(\bar{y})|^2}{|\bar{x} - \bar{y}|^{1+2s}} \mathbf{1}_{\{|\bar{x} - \bar{y}|>\delta\}} \, dy \, dx
\end{align*}
with the understanding that $\bar{x} = \avg_1^{(n)}(x)$ and $\bar{y} = \avg_1^{(n)}(y)$ are functions of $x$ and $y$, respectively.

Consider now $v = \avg_1^{(n)}u$ and observe that
\begin{multline*}
\frac{ |\avg_1^{(n)}u(\bar{x}) - \avg_1^{(n)}u(\bar{y})|^2}{|\bar{x} - \bar{y}|^{1+2s}} \mathbf{1}_{\{|\bar{x} - \bar{y}|>\delta\}} 
    \leq \frac{ |\avg_1^{(n)}u(\bar{x}) - \avg_1^{(n)}u(\bar{y})|^2}{\delta^{1+2s}}\\
    \leq 2\delta^{-1-2s} ( |\avg_1^{(n)}u(\bar{x})|^2 + | \avg_1^{(n)}u(\bar{y})|^2).
\end{multline*}
Then, for $x \in [0,1]$ and $\bar{x} = \avg_1^{(n)}(x)$, 
\begin{align*}
|\avg_1^{(n)}u(\bar{x})|
    &\leq \frac{1}{|U_1^{(n)}(\bar{x})|} \int_{U_1^{(n)}(\bar{x})} |u(z)| \, dz\\
    &\leq \sup \left\{\frac{1}{|B|} \int_{B} |u(z)| \, dz : B\subset [0,1]~\hbox{is a ball containing}~x\right\} =: Mu(x),
\end{align*}
where $M$ is the restricted, un-centered maximal function on $[0,1]$. By virtue of the Hardy-Littlewood maximal theorem, $M$ is a bounded operator on $L^2([0,1],dx)$ and 
\begin{align*}
\frac{ |\avg_1^{(n)}u(\bar{x}) - \avg_1^{(n)}u(\bar{y})|^2}{|\bar{x} - \bar{y}|^{1+2s}} \mathbf{1}_{\{|\bar{x} - \bar{y}|>\delta\}} 
    \leq \frac{2}{\delta^{1+2s}} ( (Mu(x))^2 + (Mu(y))^2) \in L^1([0,1] \times [0,1]). 
\end{align*}
The Dominated Convergence Theorem and the Lebesgue Differentiation Theorem thus yield
\begin{align*}
\lim_{n \to \infty} \mathcal{E}_{1,\delta}^{(n)}(\avg_1^{(n)}u)
    &\leq  \Lambda_1 \int_0^1 \int_0^1  \lim_{n \to \infty} \frac{ |\avg_1^{(n)}u(\bar{x}) - \avg_1^{(n)}u(\bar{y})|^2}{|\bar{x} - \bar{y}|^{1+2s}} \mathbf{1}_{\{|\bar{x} - \bar{y}|>\delta\}} \, dy \, dx\\
    &=\Lambda_1 \int_0^1 \int_0^1  \frac{ |u(x) - u(y)|^2}{|x-y|^{1+2s}} \mathbf{1}_{\{|x-y|>\delta\}} \, dy \, dx\\
    &\leq \Lambda_1 \int_0^1 \int_0^1   \frac{ |u(x) - u(y)|^2}{|x-y|^{1+2s}} \, dy \, dx = \Lambda_1 [u]_{H^s([0,1])}^2. 
\end{align*}
Since $u \in H^s([0,1])$, Fubini's Theorem implies
\begin{align*}
\widetilde{\mathcal{E}}_1^{(\infty)}(u)
    = \lim_{n \to \infty} \mathcal{E}_{1}^{(n)}(\avg_1^{(n)}u)
    &= \lim_{n \to \infty} \sup_{\delta>0} \mathcal{E}_{1,\delta}^{(n)}(\avg_1^{(n)}u)\\
    &= \sup_{\delta>0} \lim_{n \to \infty} \mathcal{E}_{1,\delta}^{(n)}(\avg_1^{(n)}u)\\
    &\leq \sup_{\delta>0}\Lambda_1 [u]_{H^s([0,1])}^2 = \Lambda_1 [u]_{H^s([0,1])}^2,
\end{align*}
and \eqref{eq:tilde-Hs-1} holds. 

Let $u \in H^s([0,1])$. 
By density of smooth functions in $H^s([0,1])$ (see e.g.~\cite{Mclean}*{Theorem 3.25}), we have
\[
C^\infty([0,1])\cap H^s([0,1])
\subset \lip([0,1])\cap H^s([0,1])
=\lip([0,1])\subset 
H^s([0,1])
\]
with both inclusions dense. 
Hence, there exists a sequence $u_k \in \lip([0,1])$ that converges to $u$ in $H^s([0,1])$ as $k \to \infty$. By \eqref{eq:tilde-Hs-1},
\[
\widetilde{\mathcal{E}}_1^{(\infty)}(u_k-u)
    + \|u_k-u\|_{L^2([0,1],dx)}^2 
        \leq \Lambda_1 [u_k-u]_{H^s([0,1])}^2 + \|u_k-u\|_{L^2([0,1],dx)}^2 \to 0
\]
as $k \to \infty$. Therefore, $u \in \widetilde{\mathcal{F}}_1^{(\infty)}$ and \eqref{eq:inclusion_domains-tilde} holds for $i=1$. 

\medskip
 
Next, we let $i=2$. 
It is enough to prove 
that
\begin{equation}\label{eq:tilde-Hs-2}
\widetilde{\mathcal{E}}_2^{(\infty)}(u) \leq 2\Lambda_2 [u]_{H^s([0,\frac12],[\frac12,1])}^2
\end{equation}
for $u \in H^s([0,\frac12],[\frac12,1])$. Then \eqref{eq:inclusion_domains-tilde} follows from the density of $C^{\frac12}([0,1])$ in $H^s([0,\frac12],[\frac12,1])$, which follows as in \cite{Mclean}.

First fix $\delta>0$ and define for $v \in \ell(V_2^{(n)})$
\begin{equation*}
\mathcal{E}_{2,\delta}^{(n)}(v)
    := 2\sum_{\bar{x} \in V_{2-}^{(n)}} \sum_{\bar{y} \in V_{2+}^{(n)}} j_2^{(n)}(\bar{x},\bar{y}) \mathbf{1}_{\{|\bar{x} - \bar{y}|>\delta\}} |v(\bar{x}) - v(\bar{y})|^2 \mu_2^{(n)}(\bar{x}) \mu_2^{(n)}(\bar{y}).
\end{equation*}
By Assumption~\ref{A:kernel}\ref{item:upper}, we get for any $v \in \ell(V_2^{(n)})$
\begin{align*}
\mathcal{E}_{2,\delta}^{(n)}(v)
    &\leq 2\Lambda_2 \sum_{\bar{x} \in V_{2-}^{(n)}} \sum_{\bar{y} \in V_{2+}^{(n)}} (\bar{y} - \bar{x})^{-1-2s} \mathbf{1}_{\{(\bar{y} - \bar{x})>\delta\}} |v(\bar{x}) - v(\bar{y})|^2 \mu_2^{(n)}(\bar{x}) \mu_2^{(n)}(\bar{y})\\
    &= 2\Lambda_2 \sum_{\bar{x} \in V_{2-}^{(n)}} \sum_{\bar{y} \in V_{2+}^{(n)}}\int_{U_2^{(n)}(\bar{x})}\int_{U_2^{(n)}(\bar{y})} \frac{ |v(\bar{x}) - v(\bar{y})|^2}{(\bar{y} - \bar{x})^{1+2s}} \mathbf{1}_{\{(\bar{y} - \bar{x})>\delta\}} \, dy \, dx.
\end{align*}
For $x \in [0,1] \setminus \{\frac12\}$, consider the function 
\begin{equation}\label{eq:l-2}
\avg_2^{(n)}(x) := \bar{x},
\end{equation}
where $\bar{x} \in V_2^{(n)}$ is the unique point for which $x \in U_2^{(n)}(\bar{x})$. Now, it follows from the above that
\begin{align*}
\mathcal{E}_{2,\delta}^{(n)}(v)
    &\leq 2\Lambda_2 \int_0^{\frac12} \int_{\frac12}^1 \frac{ |v(\bar{x}) - v(\bar{y})|^2}{(\bar{y} - \bar{x})^{1+2s}} \mathbf{1}_{\{(\bar{y} - \bar{x})>\delta\}} \, dy \, dx
\end{align*}
with the understanding that $\bar{x} = \avg_2^{(n)}(x)$ and $\bar{y} = \avg_2^{(n)}(y)$ are functions of $x$ and $y$, respectively.

Consider now $v = \avg_2^{(n)}u$ and let $M_2^-$ and $M_2^+$ denote the restricted, un-centered maximal functions on $[0,\frac12)$ and $(\frac12,1]$, respectively.
For $x \in [0,\frac12)$, $y \in (\frac12,1]$ and  $\bar{x} = \avg_2^{(n)} (x)$, $\bar{y} = \avg_2^{(n)}(y)$, the Hardy-Littlewood maximal theorem yields
\begin{multline*}
\frac{ |\avg_2^{(n)}u(\bar{x}) - \avg_2^{(n)}u(\bar{y})|^2}{(\bar{y} - \bar{x})^{1+2s}} \mathbf{1}_{\{(\bar{y} - \bar{x})>\delta\}} 
    \leq \frac{2}{\delta^{1+2s}} ( |\avg_2^{(n)}u(\bar{x})|^2 + | \avg_2^{(n)}u(\bar{y})|^2)\\
    \leq \frac{2}{\delta^{1+2s}} \left((M_2^-u(x))^2 + (M_2^+u(y))^2\right) \in L^1([0,\frac12) \times (\frac12,1]).
\end{multline*}
Hence, by the Dominated Convergence Theorem and the Lebesgue Differentiation Theorem, 
\begin{align*}
\lim_{n \to \infty} \mathcal{E}_{2,\delta}^{(n)}(\avg_2^{(n)}u)
    &\leq  2\Lambda_2 \int_0^{\frac12} \int_{\frac12}^1  \lim_{n \to \infty} \frac{ |\avg_2^{(n)}u(\bar{x}) - \avg_2^{(n)}u(\bar{y})|^2}{(\bar{y} - \bar{x})^{1+2s}} \mathbf{1}_{\{(\bar{y} - \bar{x})>\delta\}} \, dy \, dx\\
    &=2\Lambda_2 \int_0^{\frac12} \int_{\frac12}^1  \frac{ |u(x) - u(y)|^2}{(y-x)^{1+2s}} \mathbf{1}_{\{(y-x)>\delta\}} \, dy \, dx
    \leq 2\Lambda_2 [u]_{H^s([0,\frac12],[\frac12,1])}^2
\end{align*}
and by Fubini's Theorem we obtain \eqref{eq:tilde-Hs-2}. 

\medskip

Finally, let $i>2$. 
From the definition, it is clear that 
$\widetilde{F}_i^{(\infty)} \subset L^2(\overline{V}_i^*, dx)$.
We will prove that
\begin{equation}\label{eq:tilde-l2-i}
\widetilde{\mathcal{E}}_i^{(\infty)}(u) \leq \frac{4\Lambda_ib_i}{i} \|u\|^2_{L^2(\overline{V}_i^*, dx)}
\end{equation}
where $b_i$ is given in \eqref{eq:boundedkernel}.

For any $v \in \ell(V_2^{(n)})$, we use Assumption~\ref{A:kernel}\ref{item:upper} and \eqref{eq:boundedkernel} to estimate
\begin{align*}
\mathcal{E}_i^{(n)}(v)  
    &\leq 2\Lambda_i \sum_{\bar{x} \in V_{i+}^{(n)}} \sum_{\bar{y} \in V_{i-}^{(n)}} |\bar{x}- \bar{y}|^{-1-2s} |v(\bar{x}) - v(\bar{y})|^2 \mu_i^{(n)}(\bar{x})\mu_i^{(n)}(\bar{y})\\
    &\leq 2\Lambda_ib_i \sum_{\bar{x} \in V_{i+}^{(n)}} \sum_{\bar{y} \in V_{i-}^{(n)}}  |v(\bar{x}) - v(\bar{y})|^2 \mu_i^{(n)}(\bar{x})\mu_i^{(n)}(\bar{y})\\
    &\leq 4\Lambda_ib_i \sum_{\bar{x} \in V_{i+}^{(n)}} \sum_{\bar{y} \in V_{i-}^{(n)}} (|v(\bar{x})|^2 + |v(\bar{y})|^2) \mu_i^{(n)}(\bar{x})\mu_i^{(n)}(\bar{y})
    = \frac{4\Lambda_ib_i}{i} \sum_{\bar{x} \in V_{i}^{(n)}} |v(\bar{x})|^2 \mu_i^{(n)}(\bar{x}).
\end{align*}

Consider now $v = \avg_i^{(n)}u$. 
By Jensen's inequality,
\begin{align*}
|\avg_i^{(n)}u(\bar{x})|^2
    \leq \frac{1}{|U_i^{(n)}(\bar{x})|} \int_{U_i^{(n)}(\bar{x})} |u(x)|^2 \, dx,
\end{align*}
so that
\begin{align*}
\mathcal{E}_i^{(n)}(\avg_1^{(n)}u) 
&\leq  \frac{4\Lambda_ib_i}{i} \sum_{\bar{x} \in V_{i}^{(n)}} \int_{U_i^{(n)}(\bar{x})} |u(x)|^2 \, dx = \frac{4\Lambda_ib_i}{i} \int_{\overline{V}_i^*}|u(x)|^2 \, dx
    =\frac{4\Lambda_ib_i}{i} \|u\|_{L^2(\overline{V}_i^*, dx)}^2. 
\end{align*}
Taking the limit as $n \to \infty$ gives \eqref{eq:tilde-l2-i}. 

Since $L^\infty(\overline{V}_i^*)$ is dense in $L^2(\overline{V}_i^*, dx)$, we can show using \eqref{eq:tilde-l2-i} that $L^2(\overline{V}_i^*, dx) \subset \widetilde{\mathcal{F}}_i^{(\infty)}$. 
Therefore, we have \eqref{eq:inclusion_domains-tilde} for $i>2$. 

\medskip

\noindent 
\underline{\bf Step 2}. We now prove \eqref{eq:equal_domains-tilde}. 
Assume first that $i=1$. We will show that, for all $u \in \widetilde{\mathcal{F}}_1^{(\infty)}$,
\begin{equation}\label{eq:tilde-hs-lower}
\lambda_1[u]_{H^s([0,1])}^2 \leq \widetilde{\mathcal{E}}_1^{(\infty)}(u).
\end{equation}

Fix $\delta>0$. 
Recalling \eqref{eq:l-2}, the Lebesgue differentiation theorem and Fatou's lemma give
\begin{align*}
&\int_0^1 \int_0^1 \frac{|u(x) - u(y)|^2}{|x-y|^{1+2s}} \mathbf{1}_{\{|x-y|>\delta\}}\, dy \, dx\\
    &=  \int_0^1 \int_0^1 \lim_{n \to \infty} \frac{|\avg_1^{(n)}u(\bar{x}) - \avg_1^{(n)}u(\bar{y})|^2}{|\bar{x}-\bar{y}|^{1+2s}} \mathbf{1}_{\{|\bar{x}-\bar{y}|>\delta\}}\, dy \, dx\\
    &\leq \liminf_{n \to \infty}  \int_0^1 \int_0^1 \frac{|\avg_1^{(n)}u(\bar{x}) - \avg_1^{(n)}u(\bar{y})|^2}{|\bar{x}-\bar{y}|^{1+2s}} \mathbf{1}_{\{|\bar{x}-\bar{y}|>\delta\}}\, dy \, dx\\
    &= \lim_{n \to \infty}  \sum_{\bar{x} \in V_1^{(n)}}\sum_{\bar{y} \in V_1^{(n)} }\frac{|\avg_1^{(n)}u(\bar{x}) - \avg_1^{(n)}u(\bar{y})|^2}{|\bar{x}-\bar{y}|^{1+2s}} \mathbf{1}_{\{|\bar{x}-\bar{y}|>\delta\}}\, \int_{U_1^{(n)}(\bar{x})} \int_{U_1^{(n)}(\bar{y})} dy \, dx\\
    &= \lim_{n \to \infty}   \sum_{\bar{x} \in V_1^{(n)}}\sum_{\bar{y} \in V_1^{(n)}}\frac{|\avg_1^{(n)}u(\bar{x}) - \avg_1^{(n)}u(\bar{y})|^2}{|\bar{x}-\bar{y}|^{1+2s}} \mu_1^{(n)}(\bar{x}) \mu_1^{(n)}(\bar{y}).
\end{align*}
By Assumption \ref{A:kernel}\ref{item:lower},
\begin{align*}
\lambda_1 &\int_0^1 \int_0^1 \frac{|u(x) - u(y)|^2}{|x-y|^{1+2s}} \mathbf{1}_{\{|x-y|>\delta\}}\, dy \, dx\\
    &\leq \lim_{n \to \infty}   \sum_{\bar{x} \in V_1^{(n)}}\sum_{\bar{y} \in V_1^{(n)}} j_1^{(n)}(\bar{x}, \bar{y})|\avg_1^{(n)}u(\bar{x}) - \avg_1^{(n)}u(\bar{y})|^2\mu_1^{(n)}(\bar{x}) \mu_1^{(n)}(\bar{y}) = \widetilde{\mathcal{E}}^{(\infty)}_1(u).
\end{align*}
Taking the supremum over all $\delta>0$ gives  \eqref{eq:tilde-hs-lower}. 
Therefore, \eqref{eq:equal_domains-tilde} holds for $i=1$. 

\medskip

Next let $i=2$. 
We will show that, for all $u \in \widetilde{\mathcal{F}}_2^{(\infty)}$,
\begin{equation}\label{eq:tilde-hs-lower-2}
2\lambda_2[u]_{H^2([0,\frac12],[\frac12,1])}^2 \leq \widetilde{\mathcal{E}}_2^{(\infty)}(u).
\end{equation}

For this, fix $\delta>0$. 
Recalling \eqref{eq:l}, the Lebesgue differentiation theorem and Fatou's lemma give
\begin{align*}
 &\int_0^\frac12 \int_\frac12^1 \frac{|u(x) - u(y)|^2}{(y-x)^{1+2s}} \mathbf{1}_{\{(y-x)>\delta\}}\, dy \, dx\\
    &= \int_0^\frac12 \int_\frac12^1 \lim_{n \to \infty} \frac{|\avg_2^{(n)}u(\bar{x}) - \avg_2^{(n)}u(\bar{y})|^2}{(\bar{y}-\bar{x})^{1+2s}} \mathbf{1}_{\{(\bar{y}-\bar{x})>\delta\}}\, dy \, dx\\
    &\leq \liminf_{n \to \infty} \int_0^\frac12 \int_\frac12^1 \frac{|\avg_2^{(n)}u(\bar{x}) - \avg_2^{(n)}u(\bar{y})|^2}{(\bar{y}-\bar{x})^{1+2s}} \mathbf{1}_{\{(\bar{y}-\bar{x})>\delta\}}\, dy \, dx\\
    &= \lim_{n \to \infty}  \sum_{\bar{x} \in V_{2-}^{(n)}}\sum_{\bar{y} \in V_{2+}^{(n)} }\frac{|\avg_2^{(n)}u(\bar{x}) - \avg_2^{(n)}u(\bar{y})|^2}{(\bar{y}-\bar{x})^{1+2s}} \mathbf{1}_{\{(\bar{y}-\bar{x})>\delta\}}\, \int_{U_2^{(n)}(\bar{x})} \int_{U_2^{(n)}(\bar{y})} dy \, dx\\
    &= \lim_{n \to \infty}   \sum_{\bar{x} \in V_{2-}^{(n)}}\sum_{\bar{y} \in V_{2+}^{(n)}}\frac{|\avg_2^{(n)}u(\bar{x}) - \avg_2^{(n)}u(\bar{y})|^2}{(\bar{y}-\bar{x})^{1+2s}} \mu_2^{(n)}(\bar{x}) \mu_2^{(n)}(\bar{y}).
\end{align*}
By Assumption \ref{A:kernel}\ref{item:lower},
\begin{align*}
2\lambda_2 &\int_0^\frac12 \int_\frac12^1 \frac{|u(x) - u(y)|^2}{(y-x)^{1+2s}} \mathbf{1}_{\{(y-x)>\delta\}}\, dy \, dx\\
    &\leq \lim_{n \to \infty}  \sum_{\bar{x} \in V_{2-}^{(n)}}\sum_{\bar{y} \in V_{2+}^{(n)}}j_2^{(n)}(\bar{x},\bar{y})|\avg_2^{(n)}u(\bar{x}) - \avg_2^{(n)}u(\bar{y})|^2 \mu_2^{(n)}(\bar{x}) \mu_2^{(n)}(\bar{y}) = \widetilde{\mathcal{E}}^{(\infty)}_2(u).
\end{align*}
Taking the supremum over all $\delta>0$ gives  \eqref{eq:tilde-hs-lower-2}. 
Therefore, \eqref{eq:equal_domains-tilde} holds for $i=1$.

\medskip

Finally, we note that the result for $i>2$ was established in \eqref{eq:inclusion_domains-tilde}. 
 
\end{proof}

%%%%%

\begin{rem}
In the case $i=1$, $\widetilde{\mathcal{F}}^{(\infty)}_i$ could also be defined by taking 
the closure with respect to $\widetilde{\mathcal{E}}$ of the set $C^{s+\varepsilon}([0,1])$ for a fixed $\varepsilon>0$. 
We use Lipschitz functions for ease and to connect with previous work by Chen-Kim-Kumagai~\cite{CKK}. 
Similarly, for $i=2$, one could instead use $C^{\max\{s-\frac12,0\}}([\frac12-\varepsilon,\frac12+\varepsilon]) \cap L^{\infty}([0,1])$ for any fixed $\varepsilon \in (0,\frac12]$, but we consider $C^{\frac12}([0,1])$ for simplicity. 
\end{rem}

%%%%%%%%%%%%%%%%%%%%%%%%%%%%%%%%%
\subsection{Equivalence of energies}\label{subsec: forms}
%%%%%%%%%%%%%%%%%%%%%%%%%%%%%%%%%
The goal of this section is to prove that $\mathcal{E}_i^{(\infty)}(u) = \widetilde{\mathcal{E}}_i^{(\infty)}(u)$ for all $u$ in the dense set
\begin{equation}\label{eq:dense-set}
\mathcal{D}_i
    := \begin{cases}
        \lip([0,1]) & \hbox{if}~i=1\\
        C^{\frac12}([0,1]) & \hbox{if}~i=2\\
        L^{\infty}(\overline{V}_i^*) & \hbox{if}~i>2.
    \end{cases}
\end{equation}
Once this is proved, Theorem \ref{thm:DF-main} will follow from Lemma \ref{lem:domains-tilde}.

Recall from Remark~\ref{rem:split} that the network $(V_1^{(n)}, W_1^{(n)})$ is complete in the sense that all nodes are connected by a wire.  However, when $i>1$, there are no wires connecting points on the ``same side'' of the interval, see Figure \ref{fig:bipartite}. Thus,  the case $i>1$ must be treated separately from $i=1$. Moreover, we address $i=2$ separately from $i>2$ to account for the singularity at $\frac12$ (recall Remark \ref{rem:nonsingular}).

\begin{figure}[htb]
\begin{center}
\begin{tikzpicture}[thick,scale=1.3, main/.style ={circle, draw, fill=black!50,
                        inner sep=0pt, minimum width=4pt},
main2/.style ={circle, draw, fill=white,
                        inner sep=0pt, minimum width=4pt}]
%\node at (0,.5) {\hphantom{blank}};
\draw[thin,dashed] (0,0)--(4.5,0);
\draw[thin,dashed] (5.5,0)--(10,0);
\node at (5,0) {\dots};
\node[main] (0L) at (0,0) {};
\node[main] (1L) at (.5,0) {};
\node[main] (2L) at (1,0) {};
\node[main]  (3L) at (1.5,0) {};
\node[main2] at (2,0) {};
\node[main2] at (4,0) {};
\node[main2] at (6,0) {};
\node[main2] at (8,0) {};
\node[main] (3R) at (8.5,0) {};
\node[main] (2R) at (9,0) {};
\node[main] (1R) at (9.5,0) {};
\node[main] (0R) at (10,0) {};
%labels
\node at (0,-.5) {$0$};
\node at (.5,-.5) {$\frac{1}{2^2i}$};
\node at (1,-.5) {$\frac{2}{2^2i}$};
\node at (1.5,-.5) {$\frac{3}{2^2i}$};
\node at (2,-.5) {$\frac{1}{i}$};
\node at (4,-.5) {$\frac{2}{i}$};
\node at (10,-.5) {$1$};
\node at (9.5,-.5) {};%{$\frac{1}{2^2j}$};
\node at (9,-.5) {};%{$1-\frac{2}{2^2j}$};
\node at (8.5,-.5) {};%{$1-\frac{3}{2^2j}$};
\node at (8,-.5) {$1-\frac{1}{i}$};
\node at (6,-.5) {$1-\frac{2}{i}$};
%underbrace
\node[scale=1.25] at (.75,-1.2) {$\underbrace{\hphantom{\text{Some text}}}_{ V_{i-}^{(2)}}$};
\node[scale=1.25] at (9.25,-1.2) {$\underbrace{\hphantom{\text{Some text}}}_{ V_{i+}^{(2)}}$};
%wires
%from 4:
\draw[violet] (0L) to  [looseness=1]  (0R);
\draw[magenta] (0L) to  [looseness=.9]  (1R);
\draw[magenta] (1L) to  [looseness=.9]  (0R);
\draw[cyan] (1L) to  [looseness=.8]  (1R);
%From 3L
\draw[cyan] (0L) to  [looseness=.8]  (2R);
\draw[green!85!blue] (0L) to  [looseness=.7]  (3R);
\draw[green!85!blue] (1L) to  [looseness=.7]  (2R);
\draw[orange] (1L) to  [looseness=.6]  (3R);
%From 3R
\draw[cyan] (2L) to  [looseness=.8]  (0R);
\draw[green!85!blue] (2L) to  [looseness=.7]  (1R);
\draw[green!85!blue] (3L) to  [looseness=.7]  (0R);
\draw[orange] (3L) to  [looseness=.6]  (1R);
%From 2:
\draw[orange] (2L) to  [looseness=.6]  (2R);
\draw[red] (2L) to  [looseness=.5]  (3R);
\draw[red] (3L) to  [looseness=.5]  (2R);
\draw[blue] (3L) to  [looseness=.4]  (3R);
\end{tikzpicture}
\end{center}
\caption{Resistance network $(V_i^{(n)},W_i^{(n)})$ for $i>2$ at stage $n=2$}
\label{fig:bipartite}
\end{figure}
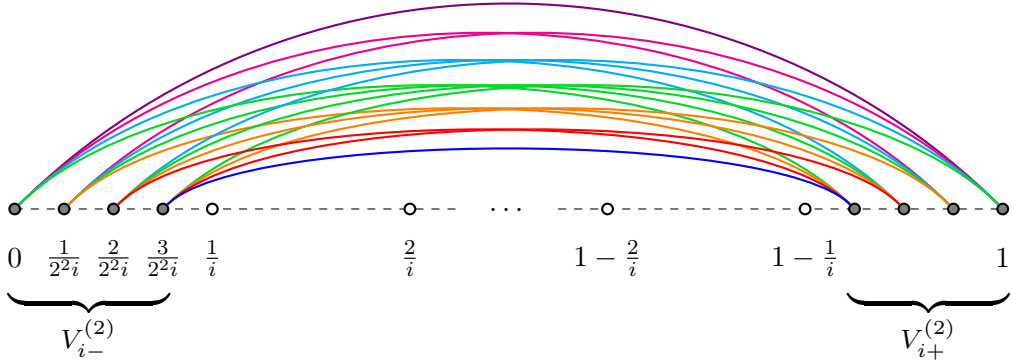

%%%%%%%%%%%%%%%%%%%%%%%%%%%%%%%%%
\subsubsection{The case $i=1$}
%%%%%%%%%%%%%%%%%%%%%%%%%%%%%%%%%

We will show $\mathcal{E}_1^{(\infty)}$ and $\widetilde{\mathcal{E}}_1^{(\infty)}$ agree on the class of Lipschitz functions. 

First, we prove a Lipschitz bound on the averaging function $\avg_i^{(n)}u$ given in Definition \ref{defn:operators}.

\begin{lem} \label{lem:lip}
If $u \in \lip([0,1])$, then for all $n \in \N$, it holds that $\avg_1^{(n)}u \in \lip(V_1^{(n)})$ with
\[
[\avg_1^{(n)}u]_{\lip(V_1^{(n)})} \leq [u]_{\lip([0,1])}. 
\]
\end{lem}

\begin{proof}
Let $x, y \in V_1^{(n)}$ and write
\[
\avg_1^{(n)}u(x) - \avg_1^{(n)}u(y)
    = \frac{1}{|U_1^{(n)}(x)|} \int_{U_1^{(n)}(x)} u(z) \, dz 
        - \frac{1}{|U_1^{(n)}(y)|}\int_{U_1^{(n)}(y)} u(z) \, dz. 
\]

Suppose first that $x,y \notin \{0,1\}$, so that $|U_i^{(n)}(x)|= |U_i^{(n)}(y)| = 2^{-n}$. Then, with the change of variables $\bar{z} = z-(y-x)$, we have
\[
\int_{U_1^{(n)}(y)} u(z) \, dz
     = \int_{U_1^{(n)}(x)} u(\bar{z}+y-x) \, d\bar{z}.
\]
Consequently, 
\begin{align*}
|\avg_1^{(n)}u(x) - \avg_1^{(n)}u(y)|
    &\leq 2^n \int_{U_1^{(n)}(x)} |u(z) -u(z+y-x)| \, dz \\
    &\leq 2^n \int_{U_1^{(n)}(x)} [u]_{\lip([0,1])} |x-y| \, dz 
    = [u]_{\lip([0,1])} |x-y|. 
\end{align*}

If $x,y \in \{0,1\}$, then $|U_i^{(n)}(x)|= |U_i^{(n)}(y)| = 2^{-n-1}$ and we can similarly show that 
\[
|\avg_1^{(n)}u(x) - \avg_1^{(n)}u(y)|
    \leq 2^{n+1} \int_{U_1^{(n)}(x)} |(u(z) -u(z+y-x))| \, dz 
    \leq  [u]_{\lip([0,1])} |x-y|.
\]

Lastly, consider $x \in \{0,1\}$ and $y \notin \{0,1\}$. For ease, assume that $x=0$ (the case $x=1$ is similar). 
Recalling that $U_1^{(n)}(0) = [0, 2^{-(n+1)})$, we write
\begin{align*}
\avg_1^{(n)}u(0)
    &= 2^{n+1} \int_{0}^{2^{-(n+1)}} u(z) \, dz
      = 2^{n} \int_{0}^{2^{-(n+1)}} u(z) \, dz
         + 2^{n} \int_{-2^{-(n+1)}}^0 u(-z) \, dz.
\end{align*}
Using the same change of variables as before, we can also write 
\begin{align*}
\avg_1^{(n)}u(y)
    &= 2^n \int_{y - 2^{-(n+1)}}^{y + 2^{-(n+1)}} u(z) \, dz 
    = 2^n \int_{ - 2^{-(n+1)}}^{ 2^{-(n+1)}} u(z+y) \, dz \\
    &= 2^n \int_{0}^{ 2^{-(n+1)}} u(z+y) \, dz 
        +  2^n \int_{- 2^{-(n+1)}}^0 u(z+y) \, dz.
\end{align*}
Therefore, 
\begin{align*}
|\avg_1^{(n)}&u(0) - \avg_1^{(n)}u(y)|\\
    &\leq 2^n \bigg[\int_{0}^{2^{-(n+1)}} |u(z) - u(z+y)| \, dz
        + \int_{-2^{-(n+1)}}^0 |u(-z) - u(z+y)| \, dz\bigg] \\
    &= 2^n \bigg[\int_{0}^{2^{-(n+1)}} [u]_{\lip([0,1])}  |y|  \, dz
        + \int_{-2^{-(n+1)}}^0 [u]_{\lip([0,1])}  (y-2|z|)  \, dz\bigg]\\
    &\leq 2^n \bigg[\int_{0}^{2^{-(n+1)}} [u]_{\lip([0,1])}  |y|  \, dz
        + \int_{-2^{-(n+1)}}^0 [u]_{\lip([0,1])}  |y|  \, dz\bigg]
    = [u]_{\lip([0,1])}  |y|
\end{align*}
as desired. 
\end{proof}

We will also need the following kernel estimate. 

\begin{lem}\label{lem:kernel-estimate1} 
Let Assumption~\ref{A:kernel}\ref{item:upper} hold for some $s\in (0,1)$. Then
\begin{equation}\label{eq:beta-kernel}
\sup_{n \in \N} \sup_{x \in V_1^{(n)}} \sum_{y \in V_1^{(n)}} j_1^{(n)}(x,y)|x-y|^{2} \mu_1^{(n)}(y) < \infty.
\end{equation}
\end{lem}

\begin{proof}
Fix $n \in \N$ and let $x \in V_1^{(n)}$. 
Using \eqref{eq:kernel-estimate-new}, we first find
\begin{equation}\label{eq:beta_kernel_p1}
\sum_{y \in V_1^{(n)}} j_1^{(n)}(x,y) |x-y|^{2} \mu_1^{(n)}(y)
    \leq \Lambda_1 \sum_{y \in V_1^{(n)} \setminus \{x\}}  |x-y|^{1-2s} \mu_1^{(n)}(y).
\end{equation}
If $0 < s \leq \frac{1}{2}$, then, 
since $|x-y| \leq 1$, 
\[
\sum_{y \in V_1^{(n)} \setminus \{x\}}  |x-y|^{1-2s} \mu_1^{(n)}(y)
    \leq \sum_{y \in V_1^{(n)} \setminus \{x\}} \mu_1^{(n)}(y) =1
\]
which, together with \eqref{eq:beta_kernel_p1}, gives \eqref{eq:beta-kernel}. Assume for the rest of the proof that $\frac{1}{2}<s<1$. 

For ease, suppose that $x \in V_1^{(n)} \setminus \{0,1\}$. The case $x \in \{0,1\}$ is similar. 
We will show that the  finite sums on the right hand side of~\eqref{eq:beta_kernel_p1} are increasing in $n$ by proving that 
\[
\frac{|x|^{1-2s} + |1-x|^{1-2s}}{2^{n+1}}+ \sum_{\substack{y \in V_1^{(n)} \\ y\notin\{ x,0,1\}} } \frac{|x-y|^{1-2s}}{2^n} 
    \leq 
     \frac{|x|^{1-2s} + |1-x|^{1-2s}}{2^{n+2}}+
    \sum_{\substack{y \in V_1^{(n+1)} \\ y\notin\{ x,0,1\}}} \frac{|x-y|^{1-2s}}{2^{n+1}},
\]
or equivalently,
\begin{equation}\label{eq:beta_kernel_p2}
   \frac{|x|^{1-2s} + |1-x|^{1-2s}}{2} +
 2\sum_{y \in V_1^{(n)} \setminus \{x,0,1\}}|x-y|^{1-2s} 
    \leq 
    \sum_{y \in V_1^{(n+1)} \setminus \{x,0,1\}}|x-y|^{1-2s}.
\end{equation}
Since
\[
\sum_{y \in V_1^{(n+1)} \setminus \{x,0,1\}}|x-y|^{1-2s}
    = \sum_{y \in V_1^{(n+1)} \setminus V_1^{(n)}}|x-y|^{1-2s}
     +\sum_{y \in V_1^{(n)} \setminus \{x,0,1\}}|x-y|^{1-2s},
\]
to obtain~\eqref{eq:beta_kernel_p2} it is enough to show
\begin{equation}\label{eq:claim-final}
\sum_{y \in V_1^{(n)} \setminus \{x,0,1\}}|x-y|^{1-2s}  + \frac{|x|^{1-2s} + |1-x|^{1-2s}}{2}
    \leq 
    \sum_{y \in V_1^{(n+1)} \setminus V_1^{(n)}}|x-y|^{1-2s}.
\end{equation}
For each $y \in V_1^{(n)}\setminus \{x\}$, there is a unique $\bar{y} \in V_1^{(n+1)}\setminus V_1^{(n)}$ such that
\[
|y - \bar{y}| = \frac{1}{2^{n+1}} \quad \hbox{and} \quad |x - \bar{y}| < |x-y|.
\]
Since $\frac{1}{2}<s<1$, we consequently find 
\[
\frac{|x-y|^{1-2s}}{2} < |x-y|^{1-2s} < |x - \bar{y}|^{1-2s}. 
\]
Thus,~\eqref{eq:claim-final} holds and so does~\eqref{eq:beta_kernel_p2}. 
In view of the latter,
\begin{align*}
\sum_{y \in V_1^{(n)} \setminus \{x\}}|x-y|^{1-2s} \mu_1^{(n)}(y)
    &\leq \lim_{k \to \infty} \sum_{y \in V_1^{(k)} \setminus \{x\}}|x-y|^{1-2s} \mu_1^{(k)}(y) \\
    &=\int_0^1 |x-y|^{1-2s} \, dy
    \leq C
\end{align*}
for some $C>0$ depending only on $s \in (\frac12,1)$. 
Combining this with \eqref{eq:beta_kernel_p1}, we have \eqref{eq:beta-kernel}.
\end{proof}

We are now prepared to prove equivalence of forms 
given by the Lipschitz density approach \eqref{E:energy_limit_fin_energy} and finite energy approach \eqref{E:energy_limit_Lip} on the class of Lipschitz functions.

\begin{lem} \label{lem:samelimit}
Let Assumption \ref{A:kernel}\ref{item:upper} hold for some $s \in (0,1)$. 
Then, for $u \in \lip([0,1])$, it holds that 
\begin{equation}\label{eq:samelimit}
\widetilde{\mathcal{E}}_1^{(\infty)}(u)
=\mathcal{E}^{(\infty)}_1(u).
\end{equation}
\end{lem}

\begin{proof}
First note by \eqref{eq:inclusion_domains},  we know that $\mathcal{E}_1^{(\infty)}(u)<\infty$. 
Recalling \eqref{E:energy_limit_fin_energy} and \eqref{E:energy_limit_Lip}, the equality \eqref{eq:samelimit} follows after showing 
\begin{equation}\label{eq:E1-difference}
\lim_{n \to \infty} |\mathcal{E}_1^{(n)}(u \big|_{V_1^{(n)}}) -\mathcal{E}_1^{(n)}(\avg_1^{(n)} u)| =0.
\end{equation}

Let $x,y \in [0,1]$ be such that $x,y \in V_1^{(n_0)}$ for some $n_0 \in \N$. Then, $x,y \in V_1^{(n)}$ for all $n \geq n_0$. 
We claim that
\begin{equation}\label{eq:zero-limit}
\lim_{n \to \infty} |(u\big|_{V_1^{(n)}}(x) - u\big|_{V_1^{(n)}}(y))^2 - (\avg_1^{(n)}u(x) -\avg_1^{(n)}u(y))^2| = 0. 
\end{equation}
Since $u$ is bounded on $[0,1]$, 
\begin{align*}
&|(u\big|_{V_1^{(n)}}(x) - u\big|_{V_1^{(n)}}(y))^2 - (\avg_1^{(n)}u(x) -\avg_1^{(n)}u(y))^2|\\
    &= | (u(x) - u(y)) + (\avg_1^{(n)}u(x) -\avg_1^{(n)}u(y))|  
    | (u(x) - u(y)) - (\avg_1^{(n)}u(x) -\avg_1^{(n)}u(y))|\\
    &\leq 
    4[u]_{L^{\infty}([0,1])}
        (| u(x)-\avg_1^{(n)}u(x)|+ | u(y)-\avg_1^{(n)}u(y)|).
\end{align*}
By the Lebesgue differentiation theorem, 
\[
| u(x)-\avg_1^{(n)}u(x)|
    = \bigg| \frac{1}{|U_1^{(n)}(x)|} \int_{U_1^{(n)}(x)} (u(x) - u(z)) \, dz \bigg| \to 0 \quad \hbox{as}~n \to \infty
\]
and similarly for $| u(y)-\avg_1^{(n)}u(y)|$. 
Consequently, \eqref{eq:zero-limit} holds. 

Next, we show that the difference in energies at level $n$ is a uniformly bounded sequence in $n$. 
For this, we recall Lemma \ref{lem:lip} and estimate
\begin{align*}
|(u\big|_{V_1^{(n)}}(x)& - u\big|_{V_1^{(n)}}(y))^2 - (\avg_1^{(n)}u(x) -\avg_1^{(n)}u(y))^2|\\
    &\leq |u(x) - u(y)|^2+|\avg_1^{(n)}u(x) -\avg_1^{(n)}u(y)|^2
    \leq 2[u]_{\lip([0,1])}^2 |x-y|^{2}. 
\end{align*}
Therefore, for $n \geq n_0$, we use Lemma \ref{lem:kernel-estimate1} to get
\begin{align*}
|\mathcal{E}_1^{(n)}&(u \big|_{V_1^{(n)}}) -\mathcal{E}_1^{(n)}(\avg_1^{(n)} u)|\\
    &\leq  \sum_{x \in V_1^{(n)}} \sum_{y \in V_1^{(n)}} 
        j_1^{(n)}(x,y)  \, \mu_1^{(n)}(y) \, \mu_1^{(n)}(x)  \\
        &\quad\quad \times \big| (u\big|_{V_1^{(n)}}(x) - u\big|_{V_1^{(n)}}(y))^2 - (\avg_1^{(n)}u(x) -\avg_1^{(n)}u(y))^2\big|\\
    &\leq 2[u]_{\lip([0,1])}^2 \sum_{x \in V_1^{(n)}} \sum_{y \in V_1^{(n)}} 
        j_1^{(n)}(x,y) |x-y|^2 \, \mu_1^{(n)}(y) \, \mu_1^{(n)}(x) \\
    &\leq C\sum_{x \in V_1^{(n)}}  \, \mu_1^{(n)}(x) = C
\end{align*}
for some $C$ is independent of $n$. With this and \eqref{eq:zero-limit}, we arrive at \eqref{eq:E1-difference}. 
\end{proof}

%%%%%%%%%%%%%%%%%%%%%%%%%%%%%%%%%%%%
\subsubsection{The case $i=2$}
%%%%%%%%%%%%%%%%%%%%%%%%%%%%%%%%%%%%

We follow the same steps as in the case $i=1$, beginning with a result analogous to Lemma \ref{lem:lip}.

\begin{lem} \label{lem:Holder}
Let $i=2$. 
If $u \in C^{\frac{1}{2}}([0,1])$, then for all $n \in \N$, it holds that $\avg_2^{(n)}u \in C^{\frac{1}{2}}(V_2^{(n)})$ with 
\[
[\avg_2^{(n)}u]_{C^{\frac{1}{2}}(V_2^{(n)})} \leq [u]_{C^{\frac12}([0,1])}.
\]
\end{lem}

\begin{proof}
Let $x,y \in V_2^{(n)}$. If $x,y \in V_{2-}^{(n)}$ or $x,y \in V_{2+}^{(n)}$, the proof is similar to the proof of Lemma \ref{lem:lip}, so we do not write the details.
Assume that $x \in V_{2-}^{(n)}$ and $y \in V_{2+}^{(n)}$. 
With the definition of $\avg_2^{(n)}$ and making a simple change of variables, we write
\begin{align*}
\avg_2^{(n)}u(x) - \avg_2^{(n)}u(y)
    &= 2^{n+1}  \left[\int_{x}^{x + \frac{1}{2^{n+1}}} u(z) \, dz 
        - \int^{y}_{y - \frac{1}{2^{n+1}}} u(z) \, dz\right]\\
     &= 2^{n+1}  \int_{x}^{x + \frac{1}{2^{n+1}}} \left(u(z) - u(z - y+x + \frac{1}{2^{n+1}})\right) \, dz.
\end{align*}
Then, we estimate 
\begin{align*}
|\avg_2^{(n)}u(x) - \avg_2^{(n)}u(y)|
    &\leq [u]_{C^{\frac12}([0,1])}\left|y-x-\frac{1}{2^{n+1}}\right|^{\frac12}.
\end{align*}
Since $(y-x) \geq \frac{1}{2^n}$, we have that
\[
0 < y-x-\frac{1}{2^{n+1}} \leq (y-x) - \frac{y-x}{2} = \frac{y-x}{2}.  
\]
Therefore, we obtain the desired H\"older bound
\[
|\avg_2^{(n)}u(x) - \avg_2^{(n)}u(y)|
    \leq 2^{-\frac{1}{2}}[u]_{C^{\frac12}([0,1])}(y-x)^{\frac12} \leq [u]_{C^{\frac12}([0,1])}|y-x|^{\frac12}.
\]
\end{proof}

In the spirit of Lemma \ref{lem:kernel-estimate1}, we prove the following kernel estimate. 
The proof is similar, but we write the details for completeness.

\begin{lem}\label{lem:kernel-estimate2} 
Let Assumption~\ref{A:kernel}\ref{item:upper} hold for some $s\in (0,1)$. Then
\begin{equation*}
\sup_{n \in \N} \sum_{x \in V_{2-}^{(n)}} \sum_{y \in V_{2+}^{(n)}} j_2^{(n)}(x,y)(y-x) \mu_2^{(n)}(y) \mu_2^{(n)}(x)< \infty.
\end{equation*}
\end{lem}

\begin{proof}
Fix $n \in \N$ and let $x \in V_{2-}^{(n)}$. 
Using \eqref{eq:kernel-estimate-new}, we first find
\begin{equation}\label{eq:beta_kernel_p1-2}
\sum_{y \in V_{2+}^{(n)}} j_2^{(n)}(x,y)(y-x) \mu_2^{(n)}(y)
    \leq \Lambda_2\sum_{y \in V_{2+}^{(n)}}  (y-x)^{-2s} \mu_2^{(n)}(y).
\end{equation}
We will show that the finite sums on the right hand side of~\eqref{eq:beta_kernel_p1-2} are increasing in $n$ by proving that
\[
\sum_{y \in V_{2+}^{(n)}}  (y-x)^{-2s} \frac{1}{2^{n+1}}
    \leq 
    \sum_{y \in V_{2+}^{(n+1)}}  (y-x)^{-2s} \frac{1}{2^{n+2}}
\]
or equivalently,
\begin{equation}\label{eq:beta_kernel_p2-2}
2\sum_{y \in V_{2+}^{(n)}}  (y-x)^{-2s} 
    \leq 
    \sum_{y \in V_{2+}^{(n+1)}}  (y-x)^{-2s}.
\end{equation}
Since
\[
\sum_{y \in V_{2+}^{(n+1)}}  (y-x)^{-2s}
    = \sum_{y \in V_{2+}^{(n+1)}\setminus V_{2+}^{(n)} }  (y-x)^{-2s}
        +  \sum_{y \in V_{2+}^{(n)} }  (y-x)^{-2s},
\]
to obtain~\eqref{eq:beta_kernel_p2-2}, it is enough to show
\begin{equation}\label{eq:claim-final-2}
\sum_{y \in V_{2+}^{(n)}}  (y-x)^{-2s}
    \leq \sum_{y \in V_{2+}^{(n+1)}\setminus V_{2+}^{(n)} }  (y-x)^{-2s}.
\end{equation}
For each $y \in V_{2+}^{(n)}$, the point $\bar{y} = y- \frac{1}{2^{n+2}}$ is the unique point in $V_{2+}^{(n+1)}\setminus V_{2+}^{(n)}$ satisfying
\[
|y - \bar{y}| = \frac{1}{2^{n+2}} \quad \hbox{and} \quad (\bar{y}-x) < y-x.
\]
In particular, 
\[
(y-x)^{-2s} < (\bar{y}-x)^{-2s}.
\]
Therefore,~\eqref{eq:claim-final-2} holds and so does~\eqref{eq:beta_kernel_p2-2}. In view of the latter,
\begin{equation}\label{eq:Lip-limit-bound-2}
 \sum_{y \in V_{2+}^{(n)}}  (y-x)^{-2s} \mu_2^{(n)}(y)
    \leq \lim_{k \to \infty}  \sum_{y \in V_{2+}^{(k)}}  (y-x)^{-2s} \mu_2^{(k)}(y)
    =\int_{\frac{1}{2}}^1 (y-x)^{-2s} \, dy.
\end{equation}

When $0 < s < \frac{1}{2}$, 
\begin{align*}
\int_{\frac{1}{2}}^1 (y-x)^{-2s} \, dy
    &\leq \frac{1}{1-2s} (1-x)^{1-2s} \leq \frac{1}{1-2s},
\end{align*}
so  with \eqref{eq:Lip-limit-bound-2}, we have
\[
\sum_{x \in V_{2-}^{(n)}} \sum_{y \in V_{2+}^{(n)}} j_2^{(n)}(x,y)(y-x) \mu_2^{(n)}(y) \mu_2^{(n)}(x)
    \leq C\sum_{x \in V_{2-}^{(n)}}  \mu_2^{(n)}(x) \leq C
\]
for $C$ independent of $n$. 
On the other hand, when $\frac{1}{2}<s<1$, 
\begin{align*}
\int_{\frac{1}{2}}^1 (y-x)^{-2s} \, dy
    &\leq \frac{1}{2s-1} \left(\frac{1}{2}-x\right)^{1-2s}
\end{align*}
and when $s = \frac{1}{2}$, 
\begin{align*}
\int_{\frac{1}{2}}^1 (y-x)^{-1} \, dy
    &\leq  
    - \ln \left(\frac{1}{2}-x\right)= \left|\ln \left(\frac{1}{2}-x\right) \right|.
\end{align*}
The previous two quantities are not uniformly bounded in $(0,\frac12)$, but they are integrable. 
Indeed, as above, we can show, for $\frac{1}{2}<s < 1$, that 
\[
\sum_{x \in V_{2-}^{(n)}}  \left(\frac{1}{2}-x\right)^{1-2s}  \mu_2^{(n)}(x)
\]
is increasing in $n \in \N$. 
Hence, 
\begin{align*}
&\sum_{x \in V_{2-}^{(n)}} \sum_{y \in V_{2+}^{(n)}} j_2^{(n)}(x,y)(y-x) \mu_2^{(n)}(y) \mu_2^{(n)}(x)
\leq C\sum_{x \in V_{2-}^{(n)}}  \left(\frac{1}{2}-x\right)^{1-2s}  \mu_2^{(n)}(x)\\
&\qquad\quad\leq C\lim_{k \to \infty} \sum_{x \in V_{2-}^{(k)}}  \left(\frac{1}{2}-x\right)^{1-2s}  \mu_2^{(k)}(x)
= C \int_0^{\frac{1}{2}} \left(\frac{1}{2}-x\right)^{1-2s}  \, dx \leq C.
\end{align*}
Similarly, for $s = \frac{1}{2}$, we can show that 
\[
\sum_{x \in V_{2-}^{(n)}}  \left|\ln \left(\frac{1}{2}-x\right) \right| \mu_2^{(n)}(x)
\]
is increasing in $n$, so we have
\begin{align*}
&\sum_{x \in V_{2-}^{(n)}} \sum_{y \in V_{2+}^{(n)}} j_2^{(n)}(x,y)(y-x) \mu_2^{(n)}(y) \mu_2^{(n)}(x)
\leq C\sum_{x \in V_{2-}^{(n)}}   \left|\ln \left(\frac{1}{2}-x\right) \right|  \mu_2^{(n)}(x)\\
&\qquad\quad\leq C\lim_{k \to \infty} \sum_{x \in V_{2-}^{(k)}}  \left|\ln \left(\frac{1}{2}-x\right) \right| \mu_2^{(k)}(x)
= C \int_0^{\frac{1}{2}} \left|\ln \left(\frac{1}{2}-x\right) \right|  \, dx \leq C.
\end{align*}
This completes the proof. 
\end{proof}

We are now able to prove equivalence of forms for $i=2$ on the set of $C^{\frac12}([0,1])$ functions. 

\begin{lem}\label{lem:samelimit-2}
Let $i=2$ and assume that \eqref{eq:kernel-estimate-new} holds.
If $u \in C^{\frac{1}{2}}([0,1])$, then
\[
\widetilde{\mathcal{E}}_2^{(\infty)}(u)
= \mathcal{E}_2^{(\infty)}(u). 
\]
\end{lem}

\begin{proof}
We first use the Lebesgue differentiation theorem (see Lemma \ref{lem:samelimit}), to show that
\[
\lim_{n \to \infty}|(u\big|_{V_2^{(n)}}(x) - u\big|_{V_2^{(n)}}(y))^2 - (\avg_2^{(n)}u(x) -\avg_2^{(n)}u(y))^2|
=0 \quad \hbox{for a.e.}~x, y \in [0,1].
\]
By Lemma \ref{lem:Holder}, we also have for a.e.~$x, y \in [0,1]$, 
\begin{align*}
|(u\big|_{V_2^{(n)}}(x)& - u\big|_{V_2^{(n)}}(y))^2 - (\avg_2^{(n)}u(x) -\avg_2^{(n)}u(y))^2|\\
    &\leq |u(x) - u(y)|^2+|\avg_2^{(n)}u(x) -\avg_2^{(n)}u(y)|^2
    \leq 2[u]_{C^{\frac12}([0,1])}^2 |x-y|. 
\end{align*}
Therefore, by Lemma \ref{lem:kernel-estimate2} and for $n \in \N$,  we have
\begin{align*}
|\mathcal{E}_2^{(n)}&(u \big|_{V_2^{(n)}}) -\mathcal{E}_2^{(n)}(\avg_2^{(n)} u)|\\
    &\leq 2 \sum_{x \in V_{2-}^{(n)} } \sum_{y \in V_{2+}^{(n)}} 
        j_2^{(n)}(x,y)  \, \mu_2^{(n)}(y) \, \mu_2^{(n)}(x)  \\
    &\quad\quad \times \big| (u\big|_{V_2^{(n)}}(x) - u\big|_{V_2^{(n)}}(y))^2 - (\avg_2^{(n)}u(x) -\avg_2^{(n)}u(y))^2\big|\\
    &\leq 4[u]_{C^{\frac{1}{2}}([0,1])}^2 \sum_{x \in V_{2-}^{(n)} } \sum_{y \in V_{2+}^{(n)}} 
        j_2^{(n)}(x,y) (y-x) \, \mu_2^{(n)}(y) \, \mu_2^{(n)}(x) 
    \leq C,
\end{align*}
for some $C>0$ independent of $n$. The conclusion follows as in the proof of Lemma \ref{lem:samelimit}.
\end{proof}

%%%%%%%%%%%%%%%%%%%%%%
\subsubsection{The case $i>2$}
%%%%%%%%%%%%%%%%%%%%%%%

Since the kernel for $i>2$ is nonsingular, we can immediately prove equivalence of forms. 

\begin{lem}\label{lem:samelimit-i}
Let $i>2$. Under Assumption~\ref{A:kernel}\ref{item:upper}, for any $u \in L^{\infty}(\overline{V}_i^{*})$, it holds that 
\[ 
    \widetilde{\mathcal{E}}_i^{(\infty)}(u)
    = \mathcal{E}_i^{(\infty)}(u).
\]
\end{lem}

\begin{proof}
By the estimates~\eqref{eq:kernel-estimate-new} and \eqref{eq:boundedkernel},
\begin{align*}
|\mathcal{E}_i^{(n)}&(u\big|_{V_i^{(n)}}) - \mathcal{E}_i^{(n)}(\avg_i^{(n)}(u))|\\
    &\leq 2\sum_{x \in V_{i-}^{(n)}} 
    \sum_{y \in V_{i+}^{(n)}} 
    j_i^{(n)}(x,y) \mu_i^{(n)}(y) \mu_i^{(n)}(x)\\
    &\quad\quad  \times \big|(u\big|_{V_i^{(n)}}(x) - u\big|_{V_i^{(n)}}(y))^2 - 
    (\avg_i^{(n)}u(x) - \avg_i^{(n)}u(y))^2\big|\\
    &\leq 2\Lambda_i\sum_{x \in V_{i-}^{(n)}} 
    \sum_{y \in V_{i+}^{(n)}} 
    (y-x)^{-(1+2s)} 2(2\|u\|_{L^{\infty}(V_i^*)})^2 \mu_i^{(n)}(y) \mu_i^{(n)}(x)\\
    &\leq C\sum_{x \in V_{i-}^{(n)}} 
    \sum_{y \in V_{i+}^{(n)}} 
    \mu_i^{(n)}(y) \mu_i^{(n)}(x) = C
\end{align*}
for some $C$, independent of $n$. 
The conclusion follows as in the proof of Lemma \ref{lem:samelimit}.
\end{proof}

%%%%%%%%%%%%%%%%%%%%%%%%%%%%%%%%
\section{Mosco convergence in the generalized sense}\label{sec:mosco}
%%%%%%%%%%%%%%%%%%%%%%%%%%%%%%%%

In this section we further analyze the convergence of the sequence of approximating forms $\{(\mathcal{E}_i^{(n)}, \ell(V_i^{(n)}))\}_{n\geq 0}$ from~\eqref{E:energy_as_DF} to the energy form $(\mathcal{E}_i^{(\infty)}, \mathcal{F}_i^{(\infty)})$ given in  \eqref{E:energy_limit_fin_energy}. 

\medskip

First, we present the notions of \emph{Mosco and $\Gamma$-convergence in the generalized sense} in our setting. See Mosco's seminal paper 
\cite{Mosco} for the original notion of what is now referred to as 
Mosco convergence and \cite{CKK} for Mosco convergence in the generalized sense.

Recall the operators $\ext_i^{(n)}$ and $\avg_i^{(n)}$ in Definition \ref{defn:operators}.

\begin{defn}\label{defn:mosco}
Let $i \in \N$. 
We say that $(\mathcal{E}_i^{(n)}, \ell(V_i^{(n)}))$ \emph{Mosco converges to $(\mathcal{E}_i^{(\infty)}, \mathcal{F}_i^{(\infty)})$ in the generalized sense} if the following conditions hold:
\begin{enumerate}[start=1,label={(\arabic*)}]
\item \label{item:mosco1}
If $u_n \in \ell^2(V_i^{(n)}, \mu_i^{(n)})$ and $u \in L^2(\overline{V}_i^{*},dx)$ are such that $\ext_i^{(n)} u_n \to u$ weakly in $L^2(\overline{V}_i^*,dx)$, then 
\[
\liminf_{n \to \infty} \mathcal{E}_i^{(n)}(u_n) \geq \mathcal{E}_i^{(\infty)}(u).
\]
\item \label{item:mosco2}
Given $u \in L^2(\overline{V}_i^{*},dx)$, there is a sequence $u_n \in \ell^2(V_i^{(n)}, \mu_i^{(n)})$ such that $\ext_i^{(n)} u_n \to u$ strongly in $L^2(\overline{V}_i^*,dx)$ and
\[
\limsup_{n \to \infty} \mathcal{E}_i^{(n)}(u_n) \leq \mathcal{E}_i^{(\infty)}(u).
\]
\end{enumerate}
We say that $(\mathcal{E}_i^{(n)}, \ell(V_i^{(n)}))$ \emph{$\Gamma$-converges to $(\mathcal{E}_i^{(\infty)}, \mathcal{F}_i^{(\infty)})$ in the generalized sense} if \ref{item:mosco1} holds under the hypothesis that $\ext_i^{(n)}u_n \to u$ strongly in $L^2(\overline{V}_i^*, dx)$ and \ref{item:mosco2} holds as written.
\end{defn}

We emphasize that the $\limsup$-inequality in Definition \ref{defn:mosco} is the same for both $\Gamma$ and Mosco convergence and also that Mosco convergence is in general stronger than $\Gamma$-convergence. 

We now present our main results.

\begin{thm}\label{thm:mosco}
Let $i = 1$ and let Assumption \ref{A:kernel} hold for some fixed $s \in (0,\frac12)$. 
Then the sequence of Dirichlet forms $(\mathcal{E}_1^{(n)}, \ell(V_1^{(n)}))$ Mosco converges to $(\mathcal{E}_1^{(\infty)}, \mathcal{F}_1^{(\infty)})$ in the generalized sense. 
\end{thm}

\begin{rem}
Theorem \ref{thm:mosco} holds for all $i \in \N$ and $s \in (0,1)$ under the following additional assumption: for every $\varepsilon>0$, there is an $N = N(\varepsilon, i,s) \in \N$ such that for all $n \geq k \geq N$ and all $u \in \ell^2(V_i^{(k)}, \mu_i^{(k)})$,
\begin{equation}\label{eq:extra}
\mathcal{E}_i^{(n)}(\avg_i^{(n)} \ext_i^{(k)} u)^{\frac12} \leq \mathcal{E}_i^{(k)}( u)^{\frac12} + \varepsilon,
\end{equation}
namely \cite{CKK}*{Assumption (A4).(ii)}. 
Alternatively, one may reframe \eqref{eq:extra} as an assumption on the jump kernel.
We emphasize that all other assumptions in \cite{CKK} are satisfied or unnecessary in our setting and also that \eqref{eq:extra} is not necessary when $i=1$ and $s \in (0,\frac12)$. 
We expect Mosco convergence to hold in general without assumption \eqref{eq:extra}, and leave the question for future investigations.  
\end{rem}

For the proof of Theorem \ref{thm:mosco}, we first prove that $\Gamma$-convergence holds for all $i \in \N$ and under Assumption \ref{A:kernel} for any $s \in (0,1)$.

\begin{thm}\label{thm:Gamma}
Fix $i \in \N$ and let Assumption \ref{A:kernel} hold for some fixed $s \in (0,1)$. 
The sequence of discrete Dirichlet forms $(\mathcal{E}_i^{(n)}, \ell(V_i^{(n)}))$ $\Gamma$-converges to $(\mathcal{E}_i^{(\infty)}, \mathcal{F}_i^{(\infty)})$ in the generalized sense. 
\end{thm}

An important consequence of Theorem \ref{thm:Gamma} is that $(\mathcal{E}_i^{(\infty)}, \mathcal{F}_i^{(\infty)})$ defines a Dirichlet form on $\overline{V}_i^*$ for each $i \in \N$, see for instance \cites{KS05, Mosco}. In particular, $(\mathcal{E}_1^{(\infty)}, \mathcal{F}_1^{(\infty)})$ defines a nonlocal Dirichlet form on $[0,1]$, c.f.~\eqref{eq:ellipticity}.

The rest of this section is organized as follows. 
First, we prove the $\limsup$-inequality in Theorems \ref{thm:mosco} and \ref{thm:Gamma}. 
Then, we introduce an extension of the discrete jump kernels $j_i^{(n)}(x,y)$ to the physical space $\overline{V}_i^*$ that will allow us to move between the electric networks and the physical spaces. 
With these extended kernels at hand, we prove the $\liminf$-inequality in Theorem \ref{thm:Gamma} by following the proof of \cite{CKK}*{Theorem 4.5}.
Finally, we prove Theorem \ref{thm:mosco} by showing that the energies are asymptotically compact in the generalized sense when $i=1$ and $s \in (0,\frac12)$.

%%%%%%%%%%%%%%%%%%%%%%%%%%%%%%%%
\subsection{Proof of $\limsup$ inequality}
%%%%%%%%%%%%%%%%%%%%%%%%%%%%%%%%

Here we prove the $\limsup$ inequality in Theorems \ref{thm:mosco} and \ref{thm:Gamma}. 

\begin{prop}\label{prop:limsup}
Fix $i \in \N$ and suppose Assumption \ref{A:kernel} holds. 
For every $u \in L^2(\overline{V}_i^*, dx)$, there is a sequence $u_n \in \ell^2(V_i^{(n)}, \mu_i^{(n)})$ such that
$\ext_i^{(n)}u_n \to u$ strongly in $L^2(\overline{V}_i^*, dx)$ and 
\[
\limsup_{n \to \infty} \mathcal{E}_i^{(n)}(u_n) \leq \mathcal{E}_i^{(\infty)}(u).
\]
\end{prop}

\begin{proof}
By Theorem \ref{thm:DF-main} and Lemmas \ref{lem:samelimit}, \ref{lem:samelimit-i}, and \ref{lem:samelimit-2} the set $\mathcal{D}_i$ in \eqref{eq:dense-set}
satisfies the following hypotheses of \cite{CKK}*{Lemma 8.2}: 
\begin{enumerate}[start=1,label={(\alph*)}]
\item \label{item:mosco2a}
$\mathcal{D}_i \subset L^2(\overline{V}_i^*, dx)$ is dense in $\mathcal{F}_i^{(\infty)}$ with respect to the norm $(\mathcal{E}_i^{(\infty)}(u) + \|u\|_{L^2(\overline{V_i}^*, dx)}^2)^{\frac12}$.
\item \label{item:mosco2b}
$\avg_i^{(n)} u \in \ell^2(V_i^{(n)}, \mu_i^{(n)})$ for all $u \in \mathcal{D}_i$.
\item \label{item:mosco2c}
For every $u \in \mathcal{D}_i$, it holds that $ \displaystyle 
\limsup_{n \to \infty} \mathcal{E}^{(n)}_i(\avg_i^{(n)} u) = \mathcal{E}^{(\infty)}(u). 
$
\end{enumerate}
Note that \ref{item:mosco2b} is obvious since $\mathcal{D}_i \subset L^{\infty}(\overline{V}_i^*)$ and $V_i^{(n)}$ is a finite set. 
(See also Lemmas \ref{lem:lip} and \ref{lem:Holder} for $i=1,2$.)
Therefore, we apply \cite{CKK}*{Lemma 8.2} to complete the proof. 
\end{proof}

%%%%%%%%%%%%%%%%%%%%%
\subsection{Extension and regularization of discrete jump kernels}
%%%%%%%%%%%%%%%%%%%%%

%%%%%%%%%%%%%%%%%%%%%
\subsubsection{Extension of discrete jump kernels}
%%%%%%%%%%%%%%%%%%%%%

To view the discrete energies as energies on the  physical space $\overline{V}_i^* = [0, \frac{1}{i}] \cup [1- \frac{1}{i},1]$, 
it will be helpful to define an extension of the discrete jump kernels $j_i^{(n)}(x,y)$. 
For $x,y \in \overline{V_{i}}^{*}$, we define 
\begin{equation*}
\overline{j}_i^{(n)}(x,y):=
\begin{cases}
j_i^{(n)}(\bar{x}, \bar{y}) & \hbox{if}~x \in \operatorname{Int}(U_i^{(n)}(\bar{x})),~
y \in \operatorname{Int}(U_i^{(n)}(\bar{y}))~\hbox{for some}~\bar{x},\bar{y} \in V_i^{(n)}\\
0 & \hbox{otherwise},
\end{cases}
\end{equation*}
where $\operatorname{Int}(U)$ denotes the interior of a set $U \subset \R$. 
For functions $u \in L^2(\overline{V}_i^*, dx)$, we now define the extended energy
\[
\overline{\mathcal{E}}_i^{(n)}(u)
:=\begin{cases}
\displaystyle \int_0^{1} \int_{0}^1 \overline{j}_1^{(n)}(x,y) |u(x) - u(y)|^2 \, dy \, dx & \hbox{if}~i=1\\[.6em]
      \displaystyle 2\int_0^{\frac{1}{i}} \int_{1- \frac{1}{i}}^1 \overline{j}_i^{(n)}(x,y) |u(x) - u(y)|^2 \, dy \, dx & \hbox{if}~i > 1.
\end{cases}
\]

For all $i,n \geq 1$, it is easy to check, using only definitions, that
\begin{equation}\label{eq:L2-ell2}
\|u\|_{\ell^2(V_i^{(n)}, \mu_i^{(n)})}^2 
=\|\ext_i^{(n)}u\|_{L^2(\overline{V}_i^*, dx)}^2 
\end{equation}
and
\begin{equation*}
\mathcal{E}_i^{(n)}(u)=\bar{\mathcal{E}}_i^{(n)}(\ext_i^{(n)}u) \quad \hbox{for all}~u \in \ell(V_i^{(n)}).
\end{equation*}

Under Assumption \ref{A:kernel}, the extended jump kernel satisfies estimates similar to the discrete jump kernel $j_i^{(n)}(x,y)$. 
When $i\in \{1,2\}$, the lower bound does not hold for $x$ and $y$ as $|x-y|^{-1-2s}$ is more singular than $\bar{j}_i^{(n)}(x,y)$.  
While here only the case $i=1$ is relevant, we present the general result for completeness.

\begin{lem}\label{lem:extended-assumption}
For $i \in \N$ and $n \in \N$, let Assumption \ref{A:kernel} hold. 
Then, if $x,y \in \overline{V}_i^*$, $x\not=y$, 
\begin{equation}\label{eq:j-bar-upper}
\bar{j}_i^{(n)}(x,y) 
\leq 
    \begin{cases}
        \Lambda_1 2^{1+2s} |x-y|^{-1-2s} & \hbox{if}~i=1\\[.5em]
        \Lambda_i  |x-y|^{-1-2s} & \hbox{if}~i>1. 
    \end{cases}
\end{equation}
Moreover, if $\bar{j}_i^{(n)}(x,y) \not=0$, then
\begin{equation}\label{eq:j-bar-lower}
\bar{j}_i^{(n)}(x,y)
\geq\begin{cases}
    \lambda_i 2^{-1-2s} |x-y|^{-1-2s} &\hbox{if}~i=1,2~\hbox{and}~|x-y|\geq 2^{-n}\\[.5em]
     \lambda_i 2^{-1-2s} |x-y|^{-1-2s}  & \hbox{if}~i>2. 
    \end{cases}
\end{equation}
\end{lem}

\begin{proof}
Let $x\not=y \in \overline{V}_i^*$. The result is trivial when $\bar{j}_i^{(n)}(x,y) = 0$, so assume without loss of generality that
$x \in \operatorname{Int}(U_i^{(n)}(\bar{x}))$ and $y \in \operatorname{Int}(U_i^{(n)}(\bar{y}))$ for some $\{\bar{x}, \bar{y}\} \in W_i^{(n)}$. 
We may further assume that $\bar{y}>\bar{x}$ and thus also have $y>x$.

\medskip

We begin by proving \eqref{eq:j-bar-upper}. First let $i=1$. 
Notice that
\begin{equation}\label{eq:difference-1}
|(y-x) - (\bar{y}-\bar{x})| \leq |x- \bar{x}| + |y - \bar{y}| \leq \frac{1}{2^n}. 
\end{equation}
Since $2^{-n} \leq \bar{y} - \bar{x} \leq 1$ and $n \geq 1$, we have
\[
\frac{y-x}{\bar{y} - \bar{x}}   
\leq \frac{ (\bar{y}- \bar{x})+\frac{1}{2^n}}{\bar{y} - \bar{x}} 
= 1 + \frac{1}{2^n(\bar y - \bar x)}
\leq 2.
\]
Therefore, by Assumption \ref{A:kernel},
\begin{align*}
\bar{j}_1^{(n)}(x,y) 
    = j_1^{(n)}(\bar{x},\bar{y})
    \leq \Lambda_1 (\bar{y} - \bar{x})^{-1-2s} 
    \leq \Lambda_1 2^{1+2s}(y-x)^{-1-2s},
\end{align*}
so that \eqref{eq:j-bar-lower} holds for $i=1$. 

Next, let $i\geq2$. In this case, 
\begin{equation}\label{eq:difference-i}
0 < (\bar{y} - \bar{x}) - (y-x) \leq \frac{1}{i2^{n-1}}.
\end{equation}
By the first inequality and Assumption \ref{A:kernel},
\[
\bar{j}_i^{(n)}(x,y)
    = j_i^{(n)}(\bar{x}, \bar{y})
    \leq \Lambda_i (\bar{y} - \bar{x})^{-1-2s}
    \leq \Lambda_i (y-x)^{-1-2s},
\]
so \eqref{eq:j-bar-upper} holds for $i \geq 2$. 

\medskip

Now we prove \eqref{eq:j-bar-lower}. 
Let $i \in \{1,2\}$ and $y-x\geq 2^{-n}$. If $\bar{y} - \bar{x} \geq 2^{-n+1}$, then, by \eqref{eq:difference-1} and \eqref{eq:difference-i},
\[
\frac{y-x}{\bar{y} - \bar{x}} \geq
\frac{(\bar{y} - \bar{x}) - \frac{1}{2^n}}{\bar{y} - \bar{x}}
=
1 - \frac{1}{2^n(\bar{y} - \bar{x})} \geq 1 - \frac{1}{2} = \frac12.
\]
Hence,
\begin{align*}
\bar{j}_i^{(n)}(x,y) 
    \geq \lambda_i (\bar{y} - \bar{x})^{-1-2s} 
    \geq \lambda_i 2^{-1-2s} (y-x)^{-1-2s}.
\end{align*}
If $\bar{y} - \bar{x} =2^{-n}$, then 
\[
\bar{j}_i^{(n)}(x,y) 
    \geq \lambda_i (\bar{y} - \bar{x})^{-1-2s} 
    = \lambda_i\left(\frac{1}{2^n}\right)^{-1-2s}
    \geq \lambda_i (y-x)^{-1-2s}. 
\]
Therefore, \eqref{eq:j-bar-lower} holds for $i \in \{1,2\}$. 

Lastly, let $i \geq 3$. 
Since
\[
\bar{y} - \bar{x} \geq \left( 1 - \frac{2}{i}\right) + \frac{1}{i2^{n-1}} = \frac{i 2^{n-1} - 2^n+1}{i2^{n-1}}
\geq \frac{2}{i2^{n-1}},
\]
we have, with \eqref{eq:difference-i},
\[
\frac{y-x}{\bar{y} - \bar{x}} 
    \geq \frac{(\bar{y} - \bar{x}) - \frac{1}{i2^{n-1}}}{\bar{y} - \bar{x}}
    = 1 - \frac{1}{i2^{n-1}(\bar{y} - \bar{x})}
    \geq 1 - \frac{1}{2} = \frac12. 
\]
Therefore, we use Assumption \ref{A:kernel} to show that \eqref{eq:j-bar-lower} holds for $i \geq 3$:
\[
\bar{j}_i^{(n)}(x,y)
    = j_i^{(n)}(\bar{x}, \bar{y})
    \geq \lambda_i (\bar{y} - \bar{x})^{-1-2s}
    \geq \lambda_i 2^{-1-2s} (y-x)^{-1-2s}.
\]
\end{proof}

%%%%%%%%%%%%%%%%%%%%%%%%%%%%%%%%

\subsubsection{Regularized jump kernels}

We will also need a regularization of the jump kernels $j_i^{(n)}(x,y)$ and $\bar{j}_i^{(n)}(x,y)$ near the singularities when $i=1,2$. We include the case $i>2$ to present a unified proof of Theorem \ref{thm:Gamma} for all $i \in \N$. 

Fix a small $\delta = \delta(i)>0$ and let $n \geq n_0$ for some $n_0 = n_0(\delta)\geq 1$, to be determined. 
We may take $\delta (i) = 0$ and $n_0 = 1$ for $i>2$ since the kernels are bounded. 

Let $x,y \in \overline{V}_i^*$ be such that
$|x-y|>\delta$
and such that $x \in U_i^{(n)}(\bar{x})$ and $y \in U_i^{(n)}(\bar{y})$ for some  $\{\bar{x}, \bar{y}\} \in W_i^{(n)}$. 
If $i \in \{1,2\}$, then
\begin{equation}\label{eq:delta}
\delta < |x-y| \leq |x- \bar{x}| + |y- \bar{y}| + |\bar{x} - \bar{y}| < 
\frac{2}{2^{n+1}} + |\bar{x} - \bar{y}| \leq \frac{1}{2^{n_0}} + |\bar{x} - \bar{y}|,
\end{equation}
and we can find $n_0 = n_0(\delta)$ large enough to guarantee that $|\bar{x} - \bar{y}| > \delta/2$. 
Thus, by \eqref{eq:kernel-estimate-new}, 
 \begin{equation}\label{eq:j-bar-estimate}
\bar{j}_i^{(n)}(x,y) 
\leq j_i^{(n)}(\bar{x},\bar{y})
\leq \Lambda_i |\bar{x}-\bar{y}|^{-1-2s} 
\leq \Lambda_i \left( \frac{2}{\delta} \right)^{1+2s} \quad \hbox{for}~i=1,2.
\end{equation}
On the other hand, for $i>2$, we recall \eqref{eq:boundedkernel} and find that
\[
\bar{j}^{(n)}_i(x,y)
\leq j_i^{(n)}(\bar{x},\bar{y})
\leq \Lambda_i |\bar{x}-\bar{y}|^{-1-2s} 
\leq \Lambda_i b_i
\]
We set 
\begin{equation}\label{eq:aidelta}
a_{i,\delta}:= 
\begin{cases}
\displaystyle  4\Lambda_i \left( \frac{2}{\delta} \right)^{1+2s}
    & \hbox{if}~i=1,2\\[.25em]
\displaystyle \frac{4\Lambda_i b_i}{i}
& \hbox{if}~i>2.
\end{cases}
\end{equation}
Note that $a_{i,\delta}$ is independent of $\delta$ for $i>2$. 

With $\delta>0$, we now introduce nonsingular energies that approximate  
$\overline{\mathcal{E}}_{i}^{(n)}$ and  $\mathcal{E}_i^{(\infty)}$.
For all $n \geq n_0$ and $u \in L^2([0,1],dx)$, define 
\begin{align*}
\overline{\mathcal{E}}_{1,\delta}^{(n)}(u)
    &:= \int_0^1 \int_0^1 \bar{j}_1^{(n)}(x,y)\mathbf{1}_{\{|x-y|>\delta\}} |u(x) - u(y)|^2 \, dy \, dx\\[.5em]
\mathcal{E}_{1,\delta}^{(\infty)}(u)
    &:= \lim_{n \to \infty} \sum_{x \in V_1^{(n)}}  \sum_{y \in V_1^{(n)}} j_1^{(n)}(x,y) \mathbf{1}_{\{|x-y|>\delta\}} |u(x) - u(y)|^2 \mu_1^{(n)}(x)\mu_1^{(n)}(y)
\end{align*}
and for $i \geq 2$,
\begin{align*}
\overline{\mathcal{E}}_{i,\delta}^{(n)}(u)
    &:= \int_0^{\frac{1}{i}} \int_{1 - \frac{1}{i}}^1 \bar{j}_i^{(n)}(x,y)\mathbf{1}_{\{|x-y|>\delta\}} |u(x) - u(y)|^2  \, dy \, dx\\[.5em]
\mathcal{E}_{i,\delta}^{(\infty)}(u)
    &:= \lim_{n \to \infty} \sum_{x \in V_{i-}^{(n)}}  \sum_{y \in V_{i+}^{(n)}} j_i^{(n)}(x,y) \mathbf{1}_{\{|x-y|>\delta\}} |u(x) - u(y)|^2 \mu_i^{(n)}(x)\mu_i^{(n)}(y).
\end{align*}
Here, $\mathbf{1}_A$ denotes the characteristic function associated to a set $A \subset \R$.
Notice also that, since $|x-y|> 1-\frac{2}{i}$ for all $x \in V_{i-}^{(n)}$, $y \in V_{i+}^{(n)}$ and $i>2$, we have
\[
\overline{\mathcal{E}}_{i,\delta}^{(n)}(u) =\overline{\mathcal{E}}_{i}^{(n)}(u)
\quad \hbox{and} \quad 
\mathcal{E}_{i,\delta}^{(\infty)}(u)
= \mathcal{E}_{i}^{(\infty)}(u)
\]
for $0 \leq \delta < 1-\frac{2}{i}$.
We claim that
\begin{equation}\label{eq:bar-L2}
\overline{\mathcal{E}}_{i,\delta}^{(n)}(u) \leq a_{i,\delta}\|u\|_{L^2([0,1],dx)}^2 
\quad \hbox{and} \quad
\mathcal{E}_{i,\delta}^{(\infty)}(u) \leq 
a_{i,\delta}\|u\|_{L^2([0,1],dx)}^2 
\quad \hbox{for}~n \geq n_0. 
\end{equation}
Indeed, for $i=1$, we use \eqref{eq:j-bar-estimate} to estimate
\begin{align*}
\overline{\mathcal{E}}_{1,\delta}^{(n)}(u)  
    &\leq \Lambda_1 \left( \frac{2}{\delta}\right)^{1+2s}\int_0^1 \int_0^1 |u(x) - u(y)|^2  \, dy \, dx \\
    &\leq 2\Lambda_1 \left( \frac{2}{\delta}\right)^{1+2s}\int_0^1 \int_0^1 (|u(x)|^2 + |u(y)|^2)  \, dy \, dx 
    = a_{1,\delta} \|u\|_{L^2([0,1],dx)}^2
\end{align*}
and \eqref{eq:kernel-estimate-new} to find 
\begin{align*}
\mathcal{E}_{1,\delta}^{(\infty)}(u)
    &\leq \Lambda_1 \lim_{n \to \infty} \sum_{x \in V_1^{(n)}}  \sum_{y \in V_1^{(n)}} |x-y|^{-1-2s} \mathbf{1}_{\{|x-y|>\delta\}} |u(x) - u(y)|^2 \mu_1^{(n)}(x)\mu_1^{(n)}(y)\\
    &\leq \Lambda_1\left( \frac{1}{\delta}\right)^{1+2s} \lim_{n \to \infty} \sum_{x \in V_1^{(n)}}  \sum_{y \in V_1^{(n)}}  |u(x) - u(y)|^2 \mu_1^{(n)}(x)\mu_1^{(n)}(y)\\
    &\leq a_{1,\delta} \|u\|_{L^2([0,1],dx)}^2.
\end{align*}
The argument for $i\geq 2$ is similar (see also \eqref{eq:ellipticity} for $i>2$). 

\medskip

Recall the dense sets
\begin{equation}\label{eq:dense-set-again}
\mathcal{D}_i
    := \begin{cases}
        \lip([0,1]) & \hbox{if}~i=1\\
        C^{\frac12}([0,1]) & \hbox{if}~i=2\\
        L^{\infty}(\overline{V}_i^*) & \hbox{if}~i>2.
    \end{cases}
\end{equation}

\begin{lem}\label{lem:bar-limit}
Let $i \in \N$. 
Fix $\delta>0$. 
Under Assumption~\ref{A:kernel},  it holds that 
\[
    \lim_{n \to \infty} \bar{\mathcal{E}}_{i,\delta}^{(n)}(u) 
    = \mathcal{E}_{i,\delta}^{(\infty)}(u)
\quad \hbox{for any}~u \in \mathcal{D}_i.
\]
If $i>2$, the statement holds for $\delta=0$. 
\end{lem}

\begin{proof}
The proof follows along similar lines as Lemmas \ref{lem:samelimit}, \ref{lem:samelimit-i}, and \ref{lem:samelimit-2}.
Thus, we only write the details for $i=1$.

Let $u \in \mathcal{D}_1=\lip([0,1])$. It is enough to show that 
\begin{equation}\label{eq:bar-convergence}
\lim_{n \to \infty} \bigg|\sum_{\bar x \in V_1^{(n)}}  \sum_{\bar y \in V_1^{(n)}} j_1^{(n)}(\bar x,\bar y) \mathbf{1}_{\{|\bar x-\bar y|>\delta\}} |u(\bar x) - u(\bar y)|^2 \mu_1^{(n)}(\bar x)\mu_1^{(n)}(\bar y) - \bar{\mathcal{E}}_{1,\delta}^{(n)}(u)\bigg| = 0.
\end{equation}
Begin by writing
\begin{align*}
\bar{\mathcal{E}}_{1,\delta}^{(n)}(u)
    &=  \int_{0}^{1} \int_{0}^1 \bar{j}_1^{(n)}(x,y)\mathbf{1}_{\{|x-y|>\delta\}} |u(x) - u(y)|^2
    \, dy \, dx\\
    &=  \sum_{\bar{x} \in V_{1}^{(n)}} 
    \sum_{\bar{y} \in V_{1}^{(n)}} 
    \int_{U_1^{(n)}(\bar{x})} \int_{U_1^{(n)}(\bar{y})} \bar{j}_1^{(n)}(x,y) 
    \mathbf{1}_{\{|x-y|>\delta\}}
    |u(x) - u(y)|^2
    \, dy \, dx\\
    &=  \sum_{\bar{x} \in V_{1}^{(n)}} 
    \sum_{\bar{y} \in V_{1}^{(n)}}  j_i^{(n)}(\bar{x},\bar{y}) \mu_1^{(n)}(\bar{x}) \mu_1^{(n)}(\bar{y})
    \fint_{U_1^{(n)}(\bar{x})} \fint_{U_1^{(n)}(\bar{y})}  |u(x) - u(y)|^2 \mathbf{1}_{\{|x-y|>\delta\}} \, dy \, dx,
\end{align*}
where $\fint_A := \frac{1}{|A|} \int_A$ for a bounded, measurable set $A \subset \R$.
Therefore, 
\begin{equation}\label{eq:delta-energy-difference}
\begin{aligned}
&|\mathcal{E}_{1,\delta}^{(n)}(u\big|_{V_1^{(n)}}) - \bar{\mathcal{E}}_{1,\delta}^{(n)}(u)|\\
    &\leq   \sum_{\bar{x} \in V_{1}^{(n)}} 
    \sum_{\bar{y} \in V_{1}^{(n)}}  j_i^{(n)}(\bar{x},\bar{y}) \mu_i^{(n)}(\bar{x}) \mu_i^{(n)}(\bar{y}) \\
   &\quad\times \left||u(\bar{x}) - u(\bar{y})|^2\mathbf{1}_{\{|\bar x-\bar y|>\delta\}} - 
    \fint_{U_i^{(n)}(\bar{x})} \fint_{U_i^{(n)}(\bar{y})}|u(x) - u(y)|^2 \mathbf{1}_{\{|x-y|>\delta\}} \, dy \, dx\right|. 
\end{aligned}
\end{equation}

Using now that $u \in \lip([0,1])$, we estimate for $\bar{x}, \bar{y} \in V_1^{(n)}$
\[
|u(\bar{x}) - u(\bar{y})|^2  \leq [u]_{\lip([0,1])}^2 |\bar{x} - \bar{y}|^2.
\]
Recalling \eqref{eq:delta}, 
if $x \in U_i^{(n)}(\bar{x})$ and $y \in U_i^{(n)}(\bar{y})$ with $|x-y|>\delta$, then $|\bar{x} - \bar{y}|>\delta/2$ for $n \geq n_0(\delta)$. In particular, $\bar{x} \not= \bar{y}$ and
\[
|x-y| \leq \frac{\delta}{2} + |\bar{x}-\bar{y}| \leq 2|\bar{x}-\bar{y}|
\]
which gives 
\begin{align*}
&\fint_{U_i^{(n)}(\bar{x})} \fint_{U_i^{(n)}(\bar{y})}  |u(x) - u(y)|^2\mathbf{1}_{\{|x-y|>\delta\}} \, dy \, dx\\
&\quad\leq 
    \fint_{U_i^{(n)}(\bar{x})} \fint_{U_i^{(n)}(\bar{y})}[u]_{\lip([0,1])}^2 |x-y|^2 \mathbf{1}_{\{|x-y|>\delta\}} \, dy \, dx\\
&\quad \leq \fint_{U_i^{(n)}(\bar{x})} \fint_{U_i^{(n)}(\bar{y})} 4 [u]_{\lip([0,1])}^2|\bar{x} - \bar{y}|^2 \mathbf{1}_{\{|x-y|>\delta\}} \, dy \, dx
    \leq 4 [u]_{\lip([0,1])}^2|\bar{x} - \bar{y}|^2.
\end{align*}
Therefore, by \eqref{eq:delta-energy-difference} and Lemma \ref{lem:kernel-estimate1}, there is a $C$ independent of $n$ such that
\begin{align*}
|\mathcal{E}_{1,\delta}^{(n)}&(u\big|_{V_1^{(n)}}) - \bar{\mathcal{E}}_{1,\delta}^{(n)}(u)|\\
    &\leq  5 [u]_{\lip([0,1])}^2 \sum_{\bar{x} \in V_{1}^{(n)}} 
    \sum_{\bar{y} \in V_{1}^{(n)}}  j_i^{(n)}(\bar{x},\bar{y})|\bar{x} - \bar{y}|^2 \mu_i^{(n)}(\bar{x}) \mu_i^{(n)}(\bar{y})
    \leq C. 
\end{align*}
Applying the Lebesgue differentiation theorem in \eqref{eq:delta-energy-difference}, we prove \eqref{eq:bar-convergence}.  
\end{proof}

%%%%%%%%%%%%%%%%%%%%%%%%%
\subsection{Proof of $\Gamma$-convergence}
%%%%%%%%%%%%%%%%%%%%%%%%%

We now prove the $\liminf$-inequality for $\Gamma$-convergence and complete the proof of Theorem \ref{thm:Gamma}. 

\begin{prop}\label{prop:liminf-gamma}
Fix $i \in \N$ and suppose that Assumption \ref{A:kernel} holds. 
If
$u_n \in \ell^2(V_i^{(n)}, \mu_i^{(n)})$ and $u \in L^2(\overline{V}_i^{*},dx)$ are such that $\ext_i^{(n)} u_n \to u$ strongly in $L^2(\overline{V}_i^*,dx)$, then
\[
\liminf_{n \to \infty} \mathcal{E}_i^{(n)}(u_n) \geq \mathcal{E}_i^{(\infty)}(u).
\]
\end{prop}

\begin{proof}
We follow the proof of \cite{CKK}*{Theorem 4.5} with minor modifications for our setting and write the details to highlight where all necessary assumptions in \cite{CKK} are satisfied.

Let $u_n \in \ell^2(V_i^{(n)},\mu_i^{(n)})$ and $u \in L^2(\overline{V}_i^*,dx)$ be such that
$\ext_i^{(n)}u_n \to u$ strongly in $L^2(\overline{V}_i^*,dx)$. 
Up to taking a subsequence, we may assume 
$\lim_{n \to \infty} \mathcal{E}_i^{(n)}(u_n)$ exists and is finite and that 
\[
\sup_{n \geq 1} \bigg(\mathcal{E}_i^{(n)}(u_n) + \|u_n\|^2_{\ell^2(V_i^{(n)}, \mu_i^{(n)})}
\bigg)<\infty.
\]
By the uniform boundedness principle, $\{\ext_i^{(n)} u_n\}_{n \in \N}$ is a bounded sequence on $L^2(\overline{V_i^*}, dx)$. 
By the Banach-Saks theorem, up to a subsequence, 
\[
v_n := \frac{1}{n} \sum_{k=1}^n \ext_i^{(k)} u_k
\]
converges to $u$ in $L^2(\overline{V_i^*}, dx)$. 

\medskip

Let $\delta = \delta(i)>0$ be small 
and take $a_{i,\delta}>0$ as in \eqref{eq:aidelta}. 
Note that we may take $\delta(i) = 0$ for $i>2$.
Fix $\varepsilon>0$. 
Since $u \in L^2(\overline{V}_i^*,dx)$, we let $f \in \mathcal{D}_i$ (recall \eqref{eq:dense-set-again}) be such that
\[
\|u - f\|_{L^2(\overline{V}_i^*,dx)} \leq \frac{\varepsilon}{\sqrt{a_{i,\delta}}}.
\]
By Lemma \ref{lem:bar-limit} and \eqref{eq:bar-L2}, we find 
\begin{align*}
\limsup_{n \to \infty} &| \overline{\mathcal{E}}_{i,\delta}^{(n)}(v_n)^{\frac12}- \mathcal{E}_{i,\delta}^{(\infty)}(f)^{\frac12}|\\
&\leq \limsup_{n \to \infty} | \overline{\mathcal{E}}_{i,\delta}^{(n)}(v_n)^{\frac12}- \overline{\mathcal{E}}_{i,\delta}^{(n)}(f)^{\frac12}|
    + \limsup_{n \to \infty} | \overline{\mathcal{E}}_{i,\delta}^{(n)}(f)^{\frac12}- \mathcal{E}_{i,\delta}^{(\infty)}(f)^{\frac12}|\\
&\leq \sqrt{a_{i,\delta}} \|u-f\|_{L^2(\overline{V}_i^*,dx)} <\varepsilon, 
\end{align*}
and similarly,
\[
|\mathcal{E}_{i,\delta}^{(\infty)}(f)^{\frac12}- \mathcal{E}_{i,\delta}^{(\infty)}(u)^{\frac12}|
\leq \mathcal{E}_{i,\delta}^{(\infty)}(f-u)^{\frac12}
\leq \sqrt{a_{i,\delta}}\|u-f\|_{L^2(\overline{V}_i^*,dx)}  < \varepsilon. 
\]
Together, we find that 
\begin{equation}\label{eq:vn-eps}
\liminf_{n \to \infty}  \overline{\mathcal{E}}_{i,\delta}^{(n)}(v_n)^{\frac12} \geq \mathcal{E}_{i,\delta}^{(\infty)}(u)^{\frac12} - 2 \varepsilon. 
\end{equation}

Since $\ext_i^{(n)}u_n$ converges strongly in $L^2(\overline{V}_i^*,dx)$, 
there is an $N \in \N$ such that
\[
\|\ext_i^{(k)}u_k -\ext_i^{(n)}u_n\|_{L^2(\overline{V}_i^*,dx)} < \frac{\varepsilon}{\sqrt{a_{i,\delta}}} \quad \hbox{for all}~n \geq k \geq N. 
\]
Then, for $n \geq k >N$, we have that 
\begin{align*}
\overline{\mathcal{E}}_{i,\delta}^{(n)}(\ext_i^{(k)}u_k)^{\frac12} - \mathcal{E}_{i,\delta}^{(n)}(u_n)^{\frac12}
     &= \bar{\mathcal{E}}_{i,\delta}^{(n)}(\ext_i^{(k)}u_k)^{\frac12} - \overline{\mathcal{E}}_{i,\delta}^{(n)}(\ext_i^{(n)} u_n)^{\frac12}
    < \varepsilon. 
\end{align*}
Next, by Lemma \ref{lem:bar-limit} and taking $N$ larger if necessary, we have, for $N \leq n \leq m$, that
\begin{align*}
\overline{\mathcal{E}}_{i, \delta}^{(m)} (v_n)^{\frac12}
    &\leq \left|\overline{\mathcal{E}}_{i, \delta}^{(m)} (v_n)^{\frac12} - \overline{\mathcal{E}}_{i, \delta}^{(m)}(f)^{\frac12}\right| + \overline{\mathcal{E}}_{i, \delta}^{(m)}(f)^{\frac12}\\
    &\leq  \overline{\mathcal{E}}_{i, \delta}^{(m)} (v_n-f)^{\frac12} + {\mathcal{E}}_{i, \delta}^{(\infty)}(f)^{\frac12}+\varepsilon \\
    &\leq  \sqrt{a_{i,\delta}} \|v_n-f\|_{L^2(\overline{V}_i^*, dx)} + {\mathcal{E}}_{i, \delta}^{(\infty)}(f)^{\frac12}+\varepsilon <{\mathcal{E}}_{i, \delta}^{(\infty)}(f)^{\frac12} +2 \varepsilon,
\end{align*}
which in particular implies
\[
\sup_{m \geq N} \overline{\mathcal{E}}_{i, \delta}^{(m)}(v_N)^{\frac12}<\infty.
\]
This gives, for $n > N$, 
\begin{align*}
\overline{\mathcal{E}}_{i,\delta}^{(n)}(v_n)^{\frac12}
    &= \overline{\mathcal{E}}_{i,\delta}^{(n)}\left(\frac{N}{n}\frac{1}{N} \sum_{k=1}^N \ext_i^{(k)} u_k + \frac{1}{n} \sum_{k=N+1}^n \ext_i^{(k)} u_k \right)^{\frac12}\\
    &\leq \frac{N}{n} \overline{\mathcal{E}}_{i,\delta}^{(n)}(v_N)^{\frac12} + \frac{1}{n} \sum_{k=N+1}^n  \overline{\mathcal{E}}_{i,\delta}^{(n)}\left(\ext_1^{(k)} u_k \right)^{\frac12}\\
     &\leq \frac{N}{n} \sup_{m \geq N} \overline{\mathcal{E}}_{i,\delta}^{(m)}(v_N)^{\frac12} + \mathcal{E}_{i,\delta}^{(n)}(u_n)^{\frac12}+\varepsilon.   
\end{align*}
Therefore,
\[
\liminf_{n \to \infty} \overline{\mathcal{E}}_{i,\delta}^{(n)}(v_n)^{\frac12}
\leq \liminf_{n \to \infty} \mathcal{E}_{i,\delta}^{(n)}\left(u_n \right)^{\frac12} + \varepsilon
\]
which, recalling \eqref{eq:vn-eps}, implies 
\begin{align*}
\mathcal{E}_{i,\delta}^{(\infty)}(u)^{\frac12} 
    &\leq \liminf_{n \to \infty} \mathcal{E}_{i,\delta}^{(n)}\left(u_n \right)^{\frac12} + 3 \varepsilon
    \leq \liminf_{n \to \infty} \mathcal{E}_{i}^{(n)}\left(u_n \right)^{\frac12} + 3 \varepsilon.
\end{align*}
Letting $\varepsilon\to 0$ followed by $\delta\to 0$ gives the desired conclusion. 
\end{proof}

Combining Propositions \ref{prop:limsup} and \ref{prop:liminf-gamma}, we obtain Theorem \ref{thm:Gamma}.

%%%%%%%%%%%%%%%%%%%%%%%%%
\subsection{Proof of Mosco convergence}
%%%%%%%%%%%%%%%%%%%%%%%%%

We now prove Theorem \ref{thm:mosco}. To prove the $\liminf$-inequality, 
we use a generalization of \emph{asymptotically compact} from \cite{Mosco}, written in our setting. 

\begin{defn}
Let $i \in \N$. 
We say the sequence $(\mathcal{E}_i^{(n)}, \ell(V_i^{(n)}))$ is \emph{asymptotically compact in the generalized sense} if every sequence $\{u_n\}$ such that $u_n \in \ell^2(V_i^{(n)}, \mu_i^{(n)})$ and
\begin{equation}\label{eq:u_n-sequence-bound}
\liminf_{n \to \infty} (\mathcal{E}_i^{(n)}(u_n) + \|u_n\|_{\ell^2(V_i^{(n)}, \mu_i^{(n)})}^2) < \infty
\end{equation}
has a subsequence $\{u_{n_k}\}$ such that $\{\ext_i^{(n_k)}u_{n_k}\}$ converges strongly in $L^2(\overline{V}_i^*,dx)$. 
\end{defn}

The following characterization is proved exactly as in  \cite{Mosco}*{Lemma 2.3}.

\begin{lem}\label{lem:Gamma-Mosco}
Let $i \in \N$. 
If $(\mathcal{E}_i^{(n)}, \ell(V_i^{(n)}))$ is asymptotically compact in the generalized sense, 
then $(\mathcal{E}_i^{(n)}, \ell(V_i^{(n)}))$ Mosco converges to $(\mathcal{E}_i^{(\infty)}, \mathcal{F}_i^{(\infty)})$ in the generalized sense  if and only if 
$(\mathcal{E}_i^{(n)}, \ell(V_i^{(n)}))$ $\Gamma$-converges to $(\mathcal{E}_i^{(\infty)}, \mathcal{F}_i^{(\infty)})$ in the generalized sense.
\end{lem}

We will prove that $(\mathcal{E}_1^{(n)}, \ell(V_1^{(n)}))$ is asymptotically compact in the generalized sense when $s \in (0,\frac12)$.  
However, when $i>2$, the sequence
$(\mathcal{E}_i^{(n)}, \ell(V_i^{(n)}))$ is not necessarily asymptotically compact in the generalized sense because, in this case, \eqref{eq:u_n-sequence-bound}  is equivalent to 
\[
\liminf_{n \to \infty} \|\ext_i^{(n)}u_n\|_{L^2(\overline{V}_i^*, dx)}^2< \infty.
\]

We first establish a lower bound on the discrete energies when $i=1$ and $s \in (0,\frac12)$. 

\begin{lem}\label{lem:ext-Hs-lower}
Let $s \in (0, \frac12)$ and let Assumption \ref{A:kernel} hold. There is a constant $c>0$ depending only on $s$ such that, for all $n \in \N$ and $u \in \ell^2(V_1^{(n)}, \mu_1^{(n)})$,
\[
\mathcal{E}_1^{(n)}(u) \geq c \lambda_1 [\ext_1^{(n)}u]_{H^s([0,1])}^2. 
\]
\end{lem}

\begin{proof}
Begin by writing
\begin{align*}
\mathcal{E}_1^{(n)}(u)   
    &= \bar{\mathcal{E}}_1^{(n)}(\ext_1^{(n)}u)\\
    &= \int_0^1 \int_{\{y \in [0,1] : |x-y|< 2^{-n}\}} \bar{j}_1^{(n)}(x,y) |\ext_1^{(n)}u(x) - \ext_1^{(n)}u(y)|^2 \, dy \, dx \\
    &\quad+ \int_0^1 \int_{\{y \in [0,1] : |x-y|\geq 2^{-n}\}} \bar{j}_1^{(n)}(x,y) |\ext_1^{(n)}u(x) - \ext_1^{(n)}u(y)|^2 \, dy \, dx.
\end{align*}
and
\begin{align*}
    [\ext_1^{(n)}u]_{H^s([0,1])}^2
    &= \int_{U_1^{(n)}(\bar{x})} \int_{\{y \in [0,1] : |x-y|< 2^{-n}\}} \frac{ |\ext_1^{(n)}u(x) - \ext_1^{(n)}u(y)|^2}{|x-y|^{1+2s}} \, dy \, dx\\
    &\quad + \int_{U_1^{(n)}(\bar{x})} \int_{\{y \in [0,1] : |x-y|\geq 2^{-n}\}} \frac{ |\ext_1^{(n)}u(x) - \ext_1^{(n)}u(y)|^2}{|x-y|^{1+2s}} \, dy \, dx.
\end{align*}
By Lemma \ref{lem:extended-assumption}, we estimate the long-range interactions by
\begin{equation}\label{eq:long-range}
\begin{split}
&\int_0^1 \int_{\{y \in [0,1] : |x-y|\geq 2^{-n}\}} \bar{j}_1^{(n)}(x,y) |\ext_1^{(n)}u(x) - \ext_1^{(n)}u(y)|^2 \, dy \, dx\\ 
&\quad\geq \frac{\lambda_1}{2^{1+2s}} \int_0^1 \int_{\{y \in [0,1] : |x-y|\geq 2^{-n}\}} \frac{ |\ext_1^{(n)}u(x) - \ext_1^{(n)}u(y)|^2}{|x-y|^{1+2s}} \, dy \, dx.
\end{split}
\end{equation}

For the close-range interactions, fix $\bar{x} \in V_1^{(n)}$. We will first show that
\begin{equation} \label{eq:summary}
\begin{split}
&\int_{U_1^{(n)}(\bar{x})} \int_{\{y \in [0,1] : |x-y|< 2^{-n}\}} \bar{j}_1^{(n)}(x,y) |\ext_1^{(n)}u(x) - \ext_1^{(n)}u(y)|^2 \, dy \, dx\\
    &\quad\geq \frac{3\lambda_1}{8} (2^n)^{2s-1} \begin{cases}
    |u(\bar{x}) - u(\bar{x} - \frac{1}{2^n})|^2 + |u(\bar{x}) - u(\bar{x} +\frac{1}{2^n})|^2 & \hbox{if}~\bar{x} \in V_1^{(n)} \setminus \{0,1\} \\
    |u(1) - u(1 - \frac{1}{2^n})|^2 & \hbox{if}~ \bar{x} =1\\
    |u(0) - u( \frac{1}{2^n})|^2 & \hbox{if}~ \bar{x} = 0.
        \end{cases}
\end{split}
\end{equation}
Note that
\begin{equation}\label{eq:zero}
\bar{j}_1^{(n)}(x,y) = 0 \quad \hbox{and} \quad 
|\ext_1^{(n)}u(x) - \ext_1^{(n)}u(y)|^2 =0 \quad \hbox{for all}~x, y \in \operatorname{Int}(U_1^{(n)}(\bar{x})).
\end{equation}
Therefore, we can write
\begin{align*}
&\int_{U_1^{(n)}(\bar{x})} \int_{\{y \in [0,1] : |x-y|< 2^{-n}\}} \bar{j}_1^{(n)}(x,y) |\ext_1^{(n)}u(x) - \ext_1^{(n)}u(y)|^2 \, dy \, dx\\
&\quad =\int_{U_1^{(n)}(\bar{x})} \int_{\{y \in [0,1]\setminus U_1^{(n)}(\bar{x}) : |x-y|< 2^{-n}\}} \bar{j}_1^{(n)}(x,y) |\ext_1^{(n)}u(x) - \ext_1^{(n)}u(y)|^2 \, dy \, dx.
\end{align*}
Next notice, for $x \in U_1^{(n)}(\bar{x})$, that  
\begin{equation}\label{eq:not-in-Ux}
\begin{split}
&\{y \in [0,1]\setminus U_1^{(n)}(\bar{x}) : |x-y|< 2^{-n}\}\\
    &\qquad\quad= \begin{cases}
    \left(x - \frac{1}{2^n}, \bar{x} - \frac{1}{2^{n+1}}\right) \cup  \left(\bar{x} + \frac{1}{2^{n+1}}, x + \frac{1}{2^n}\right) & \hbox{if}~\bar{x} \in V_1^{(n)} \setminus \{0,1\}\\
     \left(\frac{1}{2^{n+1}}, x + \frac{1}{2^n}\right) & \hbox{if}~ \bar{x} = 0 \\
     \left(x - \frac{1}{2^n}, 1 - \frac{1}{2^{n+1}}\right) & \hbox{if}~ \bar{x} = 1. 
    \end{cases}
\end{split}
\end{equation}
If $\bar{x} \in V_1^{(n)} \setminus \{0,1\}$, then 
\begin{align*}
&\int_{U_1^{(n)}(\bar{x})} \int_{x - \frac{1}{2^n}}^{\bar{x} - \frac{1}{2^{n+1}}} \bar{j}_1^{(n)}(x,y) |\ext_1^{(n)}u(x) - \ext_1^{(n)}u(y)|^2 \, dy \, dx\\
&\quad= \int_{\bar{x}-\frac{1}{2^{n+1}}}^{\bar{x}+\frac{1}{2^{n+1}}} \int_{x - \frac{1}{2^n}}^{\bar{x} - \frac{1}{2^{n+1}}} j_1^{(n)}(\bar{x}, \bar{x} - \frac{1}{2^n}) |u(\bar{x}) - u(\bar{x} - \frac{1}{2^n})|^2 \, dy \, dx\\
&\quad \geq \lambda_1 (2^{-n})^{-1-2s}|u(\bar{x}) - u(\bar{x} - \frac{1}{2^n})|^2 \int_{\bar{x}-\frac{1}{2^{n+1}}}^{\bar{x}+\frac{1}{2^{n+1}}} \int_{x - \frac{1}{2^n}}^{\bar{x} - \frac{1}{2^{n+1}}} dy \, dx \\
&\quad = \frac{\lambda_1}{2} (2^n)^{2s-1}|u(\bar{x}) - u(\bar{x} - \frac{1}{2^n})|^2,
\end{align*}
and analogously 
\begin{align*}
&\int_{U_1^{(n)}(\bar{x})} \int_{\bar{x} + \frac{1}{2^{n+1}}}^{x+ \frac{1}{2^{n}}} \bar{j}_1^{(n)}(x,y) |\ext_1^{(n)}u(x) - \ext_1^{(n)}u(y)|^2 \, dy \, dx  \geq \frac{\lambda_1}{2} (2^n)^{2s-1}|u(\bar{x}) - u(\bar{x} + \frac{1}{2^n})|^2.
\end{align*}
so that \eqref{eq:summary} holds. The calculation for $\bar{x} \in \{0,1\}$ is similar. 

Now we estimate the corresponding  term in the $H^s$ seminorm of $\ext_1^{(n)}u$ from above. 
We will show that
\begin{equation}\label{eq:summary2}
\begin{split}
&\int_{U_1^{(n)}(\bar{x})} \int_{\{y \in [0,1] : |x-y|< 2^{-n}\}} \frac{ |\ext_1^{(n)}u(x) - \ext_1^{(n)}u(y)|^2}{|x-y|^{1+2s}} \, dy \, dx\\
    &\quad\leq \frac{1}{2s(1-2s)}(2^{n})^{2s-1}
\begin{cases}
    |u(\bar{x}) - u(\bar{x} - \frac{1}{2^n})|^2 + |u(\bar{x}) - u(\bar{x} +\frac{1}{2^n})|^2 & \hbox{if}~\bar{x} \in V_1^{(n)} \setminus \{0,1\} \\
    |u(1) - u(1 - \frac{1}{2^n})|^2 & \hbox{if}~ \bar{x} =1\\
    |u(0) - u( \frac{1}{2^n})|^2 & \hbox{if}~ \bar{x} = 0.
        \end{cases}.
\end{split}
\end{equation}
By \eqref{eq:zero}, we have
\begin{align*}
&\int_{U_1^{(n)}(\bar{x})} \int_{\{y \in [0,1] : |x-y|< 2^{-n}\}} \frac{ |\ext_1^{(n)}u(x) - \ext_1^{(n)}u(y)|^2}{|x-y|^{1+2s}} \, dy \, dx\\
&\quad =\int_{U_1^{(n)}(\bar{x})} \int_{\{y \in [0,1] \setminus U_1^{(n)}(\bar{x}) : |x-y|< 2^{-n}\}} \frac{ |\ext_1^{(n)}u(x) - \ext_1^{(n)}u(y)|^2}{|x-y|^{1+2s}} \, dy \, dx.
\end{align*}
Recalling \eqref{eq:not-in-Ux}, 
if $\bar{x} \in V_1^{(n)} \setminus \{0,1\}$, then 
\begin{align*}
&\int_{U_1^{(n)}(\bar{x})} \int_{x - \frac{1}{2^n}}^{\bar{x} - \frac{1}{2^{n+1}}} 
\frac{|\ext_1^{(n)}u(x) - \ext_1^{(n)}u(y)|^2}{|x-y|^{1+2s}} \, dy \, dx\\
&\quad = |u(\bar{x})- u(\bar{x}-\frac{1}{2^n})|^2\int_{\bar{x} - \frac{1}{2^{n+1}}}^{\bar{x} + \frac{1}{2^{n+1}}} \int_{x - \frac{1}{2^n}}^{\bar{x} - \frac{1}{2^{n+1}}} 
(x-y)^{-1-2s} \, dy \, dx\\
&\quad \leq \frac{1}{2s(1-2s)}(2^{n})^{2s-1}|u(\bar{x})- u(\bar{x}-\frac{1}{2^n})|^2,
\end{align*}
and similarly 
\begin{align*}
&\int_{U_1^{(n)}(\bar{x})} \int^{x + \frac{1}{2^n}}_{\bar{x} + \frac{1}{2^{n+1}}} 
\frac{|\ext_1^{(n)}u(x) - \ext_1^{(n)}u(y)|^2}{|x-y|^{1+2s}} \, dy \, dx\\
&\quad \leq \frac{1}{2s(1-2s)}(2^{n})^{2s-1}|u(\bar{x}) -u(\bar{x}+\frac{1}{2^n})|^2.
\end{align*}
The computation for $\bar{x} \in \{0,1\}$ is similar.  

With \eqref{eq:summary} and \eqref{eq:summary2} and writing the integral over $x \in [0,1]$ as a sum over the partitions $U_1^{(n)}(\bar{x})$ for $\bar{x} \in V_1^{(n)}$, we get 
\begin{align*}
&\int_0^1 \int_{\{y \in [0,1] : |x-y|< 2^{-n}\}} \bar{j}_1^{(n)}(x,y) |\ext_1^{(n)}u(x) - \ext_1^{(n)}u(y)|^2 \, dy \, dx \\
&\quad \leq 2s(1-2s) \frac{3\lambda_1}{8}\int_0^1 \int_{\{y \in [0,1] : |x-y|< 2^{-n}\}}  \frac{ |\ext_1^{(n)}u(x) - \ext_1^{(n)}u(y)|^2}{|x-y|^{1+2s}} \, dy \, dx.
\end{align*}
Together with \eqref{eq:long-range}, the result follows by taking $c = \min\{6s(1-2s)/8,  2^{-1-2s}\}$. 
\end{proof}

\begin{prop}\label{prop:a-compact}
Let $i=1$ and Assumption \ref{A:kernel} hold for $s \in (0,\frac12)$. 
Then the sequence $(\mathcal{E}_1^{(n)}, \ell(V_1^{(n)}))$ is asymptotically compact in the generalized sense.
\end{prop}

\begin{proof}
Suppose $u_n \in \ell^2(V_1^{(n)}, \mu_1^{(n)})$ is such that \eqref{eq:u_n-sequence-bound} holds. 
By Lemma \ref{lem:ext-Hs-lower} and \eqref{eq:L2-ell2}, 
\[
\liminf_{n \to \infty} ([\ext_1^{(n)}u]^2_{H^s([0,1])} + \|\ext_1^{(n)}u\|_{L^2([0,1]),dx)}^2) < \infty. 
\]
Since $H^s([0,1])$ compactly embeds into  $L^2([0,1],dx)$ (see \cite{Hitchhikers}*{Theorem 6.7}), the sequence $\ext_1^{(n)}u_n$ has a subsequence that converges strongly in $L^2([0,1], dx)$.
\end{proof}

With this, we can prove Theorem \ref{thm:mosco}.

\begin{proof}[Proof of Theorem \ref{thm:mosco}]
By Theorem \ref{thm:Gamma}, Proposition \ref{prop:a-compact}, and Lemma \ref{lem:Gamma-Mosco}, 
the discrete Dirichlet forms  $(\mathcal{E}_1^{(n)}, \ell(V_1^{(n)}))$ Mosco converge to $(\mathcal{E}_1^{(\infty)}, \mathcal{F}_1^{(\infty)})$ in the generalized sense.

\end{proof}

%%%%%%%%%%%%%%%%%%%%%%%%%%%%%%%%%%%%
\section{Limits of energies as $i \to \infty$}\label{sec:i-limit}
%%%%%%%%%%%%%%%%%%%%%%%%%%%%%%%%%%%%

The purpose of this section is to further address the novelty of having infinitely many vertices in the directed graph by considering
the limit as $i \to \infty$. Specifically, for a function $u:\overline{V}_{i_0}^* \to \R$ with $i_0 \in \N$, we take the limit of $\mathcal{E}_i^{(\infty)}(u\big|_{\overline{V}_i^*})$ and $\mathcal{E}_i^{(n)}(u\big|_{V_i^{(n)}})$ for a fixed $n \in \N$, as indicated in Figure \ref{fig:infty}. 
We assume $u \in C(\overline{V}_{i_0}^*)$ when studying the limit at a fixed stage of approximation 
to allow us to easily move between discrete spaces, but we believe that this can be weakened to $u \in L^\infty(\overline{V}_{i_0}^*)$ and also that more robust notions of convergence could be considered.

\begin{figure}[hbt]
\begin{tikzcd}
\mathcal{E}_i^{(n)}(u\big|_{V_i^{(n)}}) \quad
\arrow[r, blue, "i \to \infty",shorten <=-1.5mm, shorten >=-1.5mm] 
\arrow[d, blue, "n \to \infty", shorten <=-1mm, shorten >=-1mm]
& \quad 2|u(0) - u(1)|^2 
\arrow[d,blue,leftrightarrow,"="] \\[.5em]
\mathcal{E}_{i}^{(\infty)}(u\big|_{\overline{V}_{i}^*}) \quad 
\arrow[r, blue, "i \to \infty ", shorten <=-1.5mm, shorten >=-1.5mm]
& \quad 2|u(0) - u(1)|^2
\end{tikzcd}
\caption{Limiting energies of a fixed $u \in C(\overline{V_{i_0}}^*)$ for some $i_0>2$}
\label{fig:infty}
\end{figure}

Recall the notation $\sigma_{x,y}^{(n)} \in E_i^{(n)}$ from \eqref{eq:path-notation}. 
A sufficient condition for the limiting energies to exist is that the limit of $\delta_{\sigma_{x,y}^{(n)}}^{-1}$ as $i \to \infty$ exists, is finite, and is uniform in $(x,y) \in [0,\frac{1}{i_0}] \times [1 - \frac{1}{i_0},1]$ for some fixed $i_0$.
In the context of the jump kernel in \eqref{E:def_jump_kernel},

we make the following assumption on the constants $\lambda_i, \Lambda_i$ from Assumption \ref{A:kernel}. 

\begin{assumption}\label{A:lambdas}
The constants $0 < \lambda_i\leq \Lambda_i <\infty$ satisfy 
\begin{equation}\label{eq:Ci-condition-moved}
\lim_{i \to \infty} \frac{\lambda_i}{i^2} = 1 \quad \hbox{and} \quad \lim_{i \to \infty} \frac{\Lambda_i}{i^2} = 1. 
\end{equation}
\end{assumption}

Assumption \ref{A:lambdas} ensures the following limiting result regarding the summation of the discrete jump kernels. 

\begin{lem}\label{lem:i-infty-kernel}
Fix $n \in \N$. If Assumption \ref{A:kernel} and Assumption \ref{A:lambdas} hold for $i >2$, then
\[
\lim_{i \to \infty}  \sum_{x \in V_{i-}^{(n)}} \sum_{y \in V_{i+}^{(n)}} j_i^{(n)}(x,y) \mu_i^{(n)}(x) \mu_i^{(n)}(y) = 1.
\]
\end{lem}

\begin{proof}
With the definition of $V_{i\pm}^{(n)}$, we can write
\begin{align*}
 \sum_{x \in V_{i-}^{(n)}} \sum_{y \in V_{i+}^{(n)}}  j_i^{(n)}(x,y) \mu_i^{(n)}(x) \mu_i^{(n)}(y)
    &=  \sum_{k,\ell=0}^{2^n-1}j_i^{(n)}\left(\frac{k}{i2^n},1-\frac{\ell}{i2^n}\right)\frac{1}{(i2^n)^2}.
\end{align*}
Fix $k,\ell \in \{0,1 \dots, 2^n-1\}$. 
By Assumption \ref{A:kernel}, 
we find
\begin{align*}
j_i^{(n)}\left(\frac{k}{i2^n},1-\frac{\ell}{i2^n}\right) \frac{1}{i^2}
&\leq \frac{\Lambda_i}{i^2} \left(1-\frac{(\ell+k)}{i2^n}\right)^{-1-2s}
\leq \frac{\Lambda_i}{i^2}\left(1-\frac{2(2^n-1)}{i2^n}\right)^{-1-2s}, 
\end{align*}
and 
\begin{align*}
j_i^{(n)}\left(\frac{k}{i2^n},1-\frac{\ell}{i2^n}\right) \frac{1}{i^2}
&\geq \frac{\lambda_i}{i^2} \left(1-\frac{(\ell+k)}{i2^n}\right)^{-1-2s} \geq \frac{\lambda_i}{i^2}.
\end{align*}
Hence, by Assumption \ref{A:lambdas} 
and the comparison principle,  
\begin{equation}\label{eq:kernel-infty}
\lim_{i \to \infty} j_i^{(n)}\left(\frac{k}{i2^n},1-\frac{\ell}{i2^n}\right) \frac{1}{i^2} 
= 1.
\end{equation}
Therefore,
\begin{align*}
\lim_{i \to \infty} \sum_{x \in V_{i-}^{(n)}} \sum_{y \in V_{i+}^{(n)}}  j_i^{(n)}(x,y) \mu_i^{(n)}(x) \mu_i^{(n)}(y)
    &= \sum_{k,\ell=0}^{2^n-1} \lim_{i \to \infty} \left[ j_i^{(n)}\left(\frac{k}{i2^n},1-\frac{\ell}{i2^n}\right)\frac{1}{(i2^n)^2} \right]\\
    &= \sum_{k,\ell=0}^{2^n-1} \frac{1}{2^{2n}} 
    =1.
\end{align*}
\end{proof}

For a fixed $n \in \N$, we now analyze the limit of the sequence $\{\mathcal{E}_i^{(n)}\}_{i\in \N}$ as $i\to\infty$, which we expect to yield a functional in $\ell(\{0,1\})$.

\begin{lem}
Fix $n \in \N$. Let Assumption \ref{A:kernel} and Assumption \ref{A:lambdas} hold for $i >2$. 
If $u \in {C(\overline{V}_{i_0}^*)}$ for some $i_0 >2$, then
\[
\lim_{i \to \infty} \mathcal{E}_i^{(n)}(u\big|_{V_i^{(n)}}) = 2|u(0) - u(1)|^2.
\]
\end{lem}

\begin{proof}
Fix $i \geq i_0$. Note that $V_i^{(n)} \subset \overline{V_i}^* \subset \overline{V_{i_0}}^*$. 
For convenience, define
\[
\kappa_i :=  \sum_{x \in V_{i-}^{(n)}} \sum_{y \in V_{i+}^{(n)}} j_i^{(n)}(x,y) \mu_i^{(n)}(x) \mu_i^{(n)}(y) >0.
\]
By Lemma \ref{lem:i-infty-kernel}, notice that
\begin{equation}\label{eq:kappa-limit}
\lim_{i \to \infty} \kappa_i = 1. 
\end{equation}
We write
\begin{align*}
2|u(0) - u(1)|^2
 &=  \frac{2}{\kappa_i} \sum_{x \in V_{i-}^{(n)}} \sum_{y \in V_{i+}^{(n)}} j_i^{(n)}(x,y)|u(0) - u(1)|^2\mu_i^{(n)}(x) \mu_i^{(n)}(y).
\end{align*}
Consequently,
\begin{align*}
\big|&\mathcal{E}_i^{(n)}(u\big|_{V_i^{(n)}}) - 2(u(0) - u(1))^2\big|\\
    &\leq  2\sum_{x \in V_{i-}^{(n)}} \sum_{y \in V_{i+}^{(n)}} j_i^{(n)}(x,y)\left| (u(x) - u(y))^2 - \frac{1}{\kappa_i}(u(0) - u(1))^2\right|\mu_i^{(n)}(x) \mu_i^{(n)}(y)\\
    &=2\sum_{k,\ell =0}^{2^n-1}  j_i^{(n)}\left(\frac{k}{i^2n},1-\frac{\ell}{i^2n}\right)\left| \left(u\left(\frac{k}{i^2n}\right) - u\left(1-\frac{\ell}{i^2n}\right)\right)^2 - \frac{1}{\kappa_i}(u(0) - u(1))^2\right| \frac{1}{(i2^n)^2}.
\end{align*}
Let $k,\ell \in \{0,\dots,2^{n}-1\}$.  
{Since $u$ is continuous},  
\[
\lim_{i \to \infty} u \left( \frac{k}{i2^n}\right) = u (0) \quad \hbox{and} \quad
\lim_{i \to \infty} u \left(1- \frac{\ell}{i2^n}\right) = u (1).
\]
Therefore, with \eqref{eq:kappa-limit},
\[
\lim_{i \to \infty} \left| \left(u\left(\frac{k}{i^2n}\right) - u\left(1-\frac{\ell}{i^2n}\right)\right)^2 - \frac{1}{\kappa_i}(u(0) - u(1))^2\right| = 0. 
\] 
With this and recalling \eqref{eq:kernel-infty}, the result follows. 
\end{proof}

Next, we compare the previous limit with that of the sequence $\{\mathcal{E}_{i}^{(\infty)}\}_{i\in \N}$ as $i \to \infty$.

\begin{lem}
Let Assumption \ref{A:kernel} and Assumption \ref{A:lambdas} hold for $i >2$. 
Let $u \in \mathcal{F}_{i_0}^{(\infty)}$ for some fixed $i_0 \in \N$. Then it holds that $u\big|_{\overline{V}_{i}^*} \in \mathcal{F}_{i}^{(\infty)}$ for all $i \geq i_0$, and  
\[
\lim_{i \to \infty} \mathcal{E}_{i}^{(\infty)}(u\big|_{\overline{V}_{i}^*}) =  2|u(0) - u(1)|^2.
\]
\end{lem}

\begin{proof}
Without loss of generality, assume that $i_0>1$. The case $i_0=1$ is similar. 

Fix $i > i_0$. 
Since $\overline{V}_i^* \subset \overline{V}_{i_0}^*$, \eqref{eq:inclusion_domains} implies that $u\big|_{\overline{V}_i^*} \in \mathcal{F}_i^{(\infty)}$.

Next, observe that
\begin{align*}
\lim_{i \to \infty} \mathcal{E}_{i}^{(\infty)}(u\big|_{\overline{V}_{i}^*})
    &\leq \lim_{i \to \infty} 2\Lambda_{i} \int_0^{\frac{1}{i}} \int_{1 - \frac{1}{i}}^1 \frac{|u(x) - u(y)|^2}{(y-x)^{1+2s}} \, dy \, dx\\
    &= 2 \lim_{i \to \infty} \left(\frac{\Lambda_{i}}{i^2} \right) 
    \fint_0^{\frac{1}{i}} \fint_{1 - \frac{1}{i}}^1 \frac{|u(x) - u(y)|^2}{(y-x)^{1+2s}} \, dy \, dx
    = 2\frac{|u(0) - u(1)|^2}{(1-0)^{1+2s}},
\end{align*}
where use \eqref{eq:Ci-condition-moved} and the Lebesgue differentiation theorem in the last line.  
Similarly, from below, we estimate
\[
\lim_{i \to \infty} \mathcal{E}_{i}^{(\infty)}(u\big|_{\overline{V}_{i}^*})
    \geq 2 \lim_{i \to \infty} \left(\frac{\lambda_{i}}{i^2} \right) 
    \fint_0^{\frac{1}{i}} \fint_{1 - \frac{1}{i}}^1 \frac{|u(x) - u(y)|^2}{(y-x)^{1+2s}} \, dy \, dx 
    = 2 |u(0) - u(1)|^2.
\]
\end{proof}

In summary, under Assumption~\ref{A:lambdas}, the energies $\mathcal{E}_i^{(n)}$ and $\mathcal{E}_i^{(\infty)}$ yield the same functional on $\ell(\{0,1\})$ as $i\to\infty$.

\begin{rem}
Let $u \in L^{\infty}(V_{i_0}^{(\infty)})$ for a fixed $i_0 \in \N$. 
If the sequence $\frac{\lambda_i}{i^2} \to +\infty$, then we can show that $\mathcal{E}_i^{(n)}(u\big|_{V_i^{(n)}})$ and $\mathcal{E}_i^{(\infty)}(u\big|_{\overline{V}_i^*})$ both diverge to $+\infty$ as $i \to \infty$.
On the other hand, 
if $\frac{\Lambda_i}{i^2} \to 0$, then we can show that 
$\mathcal{E}_i^{(n)}(u\big|_{V_i^{(n)}})$ and $\mathcal{E}_i^{(\infty)}(u\big|_{\overline{V}_i^*})$ both converge to zero. 
\end{rem}

%%%%%%%%%%%%%%%%%%%%%%%%%%%%%%%%%%%%%
\begin{appendix}
%%%%%%%%%%%%%%%%%%%%%%%%%%%%%%%%%%%%%
\section{Local vs.~nonlocal energies}\label{sec:appendixA}
%%%%%%%%%%%%%%%%%%%%%%%%%%%%%%%%%%%%%%

For expository purposes, we find it helpful to present the construction of classical (local) energies on the unit interval and showcase how the usual techniques for compatibility are not applicable in the nonlocal setting. 

%%%%%%%%%%%%%%%%%%%
\subsection{Graph-directed construction of a local energy on the unit interval}
%%%%%%%%%%%%%%%%%%%

We briefly review the graph-directed construction of a local energy on the unit interval. Our presentation follows the same general layout as Sections \ref{sec:construction} and  \ref{sec:discrete}. 

\subsubsection*{Index space}
The index space is the directed graph $(1, E_L)$ consisting of a single vertex trivially denoted by $1$ and an edge set $E_L=\{e_{1,1}, e_{1,1}'\}$ with two directed edges starting and terminating at vertex $1$, see Figure \ref{fig:graph-directed-local}. On both edges $e \in E_L$, we assign the same weight $r>0$. 
A path $\sigma$ of length $n \in \N$ is a concatenation of $n$ edges in $E_L$. Let $E_L^{(n)}$ denote the collection of paths of length $n$.

\begin{figure}[hbt]
\begin{tikzpicture}[ node distance={19mm}, thick, main/.style = {draw, circle}] 
\node[main] (1) {$1$}; 
\draw[->] (1) to [out=25,in=-25,looseness=8] (1);
\draw[->] (1) to [out=155,in=205,looseness=8] (1);
\node [left=.6cm of 1] {\small $e_{1,1}$};
\node [right=.6cm of 1] {\small $e_{1,1}'$};
\end{tikzpicture}
\caption{Directed graph $(1, E_L)$}
\label{fig:graph-directed-local}
\end{figure}
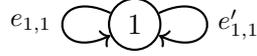

\subsubsection*{Physical space}
To each $e \in E_L$, we associate the map $\phi_e$ given in \eqref{E:def_phi_e}. 
Given a path $\sigma$, let $\phi_\sigma$ denote the corresponding composition of maps $\phi_e$, $e \in E_L$ from \eqref{eq:composition}. 
The finite sequence of approximations 
\[
V_1^{(0)} = \{0,1\}, \quad 
\bigcup_{\sigma \in E_L^{(n)}} \phi_\sigma(V_1^{(0)}) = \left\{\frac{k}{2^n}\right\}_{k=0}^{2^n} = V_1^{(n)}, \quad n \geq 1,
\]
is the same as \eqref{eq:nodes} for $i=1$ and gives rise to the same the physical space $\overline{V}_1^* = [0,1]$.

\subsubsection*{Electrical networks}

For $n \geq 0$, define a new set of wires ${W}_L^{(n)}$ by
\[
{W}_L^{(n)} = \bigcup_{\sigma \in E_L^{(n)}} \{ \phi_\sigma(0), \phi_\sigma(1)\} = 
\left\{ \left\{x, x \pm \frac{1}{2^{n}}\right\} : x \in V_1^{(n)} \setminus \{0,1\}\right\}. 
\]
The corresponding electrical network is $(V_1^{(n)}, {W}_L^{(n)})$.  
In contrast to \eqref{eq:wires}, each node in  $(V_1^{(n)}, {W}_L^{(n)})$ is only connected to its nearest neighbor with respect to the Euclidean distance (see Figure \ref{fig:local-networks}). 
The resistance on a wire in $W_L^{(n)}$ at stage $n \geq 0$ is given by $r^n$ where $r$ is the edge weight on $e \in E_L$. 

\subsubsection*{Approximating energies}
Define $\mathcal{E}_L^{(0)}(u) = (u(0) - u(1))^2$.
For $n \in \N$, the energy $(\mathcal{E}_L^{(n)}, \ell(V_1^{(n)}))$ associated with $(V_1^{(n)}, W_L^{(n)})$ is given by
\[
\mathcal{E}_L^{(n)}(u) = \sum_{\sigma \in E_L^{(n)}} r^{-n} \mathcal{E}_L^{(0)}(u \circ \phi_\sigma), \quad u \in \ell(V_1^{(n)})
\]
and one can check that
\[
 \mathcal{E}_L^{(n)}(u)   
    = \sum_{\{x,y\} \in W_L^{(n)}} r^{-n} (u(x) - u(y))^2, \quad u \in \ell(V_1^{(n)}).
\]

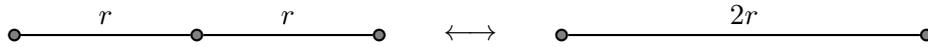
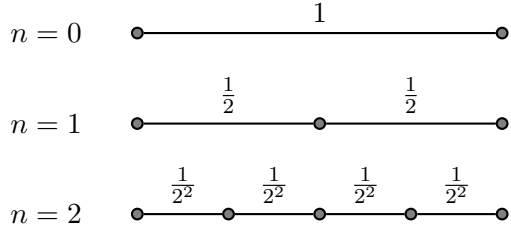
\begin{figure}[h]
\centering
\subfloat[Network reduction from stage 1 to stage 0 \label{fig:local-reduction}]{
 \begin{tikzpicture}[thick,scale=1.2, main/.style ={circle, draw, fill=black!50,
                        inner sep=0pt, minimum width=4pt}]
%%%first graph
\node[main]  (0) at (0,0) {};
\node[main]  (1) at (2,0) {};
\node[main]  (2) at (4,0) {};
\path
	(0) edge node[above] {$r$} (1);
\path
	(1) edge node[above] {$r$} (2);
%%%
\node at (5,0) {$\longleftrightarrow$};
%%%second graph
\node[main]  (0) at (6,0) {};
\node[main]  (2) at (10,0) {};
\path
	(0) edge node[above] {$2r$} (2);
\node[opacity=0] at (0,-.5) {};
\end{tikzpicture}
}

\medskip

\subfloat[Electrical networks at stages $n=0,1,2$
\label{fig:local-networks}]{
 \begin{tikzpicture}[thick,scale=1.2, main/.style ={circle, draw, fill=black!50,
                        inner sep=0pt, minimum width=4pt}]
%%labels
\node at (-1,0) {$n=0$};
\node at (-1,-1) {$n=1$};
\node at (-1,-2) {$n=2$};
%%in=0
\node[main]  (0) at (0,0) {};
\node[main]  (4) at (4,0) {};
    \path (0) edge node[above] {$1$} (4);
%%n=1
\node[main]  (0) at (0,-1) {};
\node[main]  (2) at (2,-1) {};
\node[main]  (4) at (4,-1) {};
    \path (0) edge node[above] {$\frac12$} (2);
    \path (2) edge node[above] {$\frac12$} (4);
%%n=2
\node[main]  (0) at (0,-2) {};
\node[main]  (1) at (1,-2) {};
\node[main]  (2) at (2,-2) {};
\node[main]  (3) at (3,-2) {};
\node[main]  (4) at (4,-2) {};
    \path (0) edge node[above] {$\frac{1}{2^2}$} (1);
    \path (1) edge node[above] {$\frac{1}{2^2}$} (2);
    \path (2) edge node[above] {$\frac{1}{2^2}$} (3);
    \path (3) edge node[above] {$\frac{1}{2^2}$} (4);
\node[opacity=0] at (-1.75,-2.5) {};
\node[opacity=0] at (5.75,-2.5) {};
\end{tikzpicture}
}
\caption{Electrical networks associated to discrete approximating energies of a local energy on the unit interval}
%\label{fig:local}
\end{figure}

\subsubsection*{Compatibility}

Since the wires at stage $n=1$ are connected in series, we reduce as in Figure \ref{fig:local-reduction} to find that the networks at stages $n=0$ and $n=1$ are electrically equivalent if and only if $r = \frac12$. 
At any higher stage $n\geq 1$, all wires are connected in series, so by assigning the resistance $r^n =\frac{1}{2^n}$ to each wire in ${W}_L^{(n)}$, it follows that the sequence of networks 
$\{(V_1^{(n)}, W_L^{(n)})\}_{n \geq 0}$ is electrically equivalent, see Figure \ref{fig:local-networks}. 
Consequently, the sequence of energies $(\mathcal{E}_L^{(n)}, \ell(V_1^{(n)}))$ satisfy
\[
 \mathcal{E}_L^{(n)}(u)   
    = \sum_{x \in V_1^{(n)}} \sum_{\substack{y \in V_1^{(n)}\\
    \text{s.t.}~|x-y|=2^{-n}}} 2^{n} (u(x) - u(y))^2, \quad u \in \ell(V_1^{(n)})
\]
and are compatible.  
Moreover, the resulting effective resistance distance is equal to the Euclidean distance. 
One can check using the definition of Riemann integration that
\[
\lim_{n \to \infty}
 \mathcal{E}_L^{(n)}(u)
    = \lim_{n \to \infty} 2 \sum_{k=0}^{2^n-1} \left|  \frac{u \left( \frac{k}{2^n}\right) - u \left( \frac{k+1}{2^n}\right)}{2^{-n}}\right|^2 2^{-n}
    =  2 \int_0^1 | \nabla u|^2 \, dx. 
\]

%%%%%%%%%%%%%%%%%%%
\subsection{Compatibility challenges in the nonlocal setting}
%%%%%%%%%%%%%%%%%%%

Here, we 
explain why the typical electrical network approaches for compatibility, including those from the previous subsection, do not readily apply in our context.  
Recall that compatibility of energy forms is characterized by  electrical equivalence of the corresponding electrical networks.

%%%%%%%

\subsubsection*{Network reduction rules}

A key feature for constructing local energies on the  unit interval is the ability to divide the network approximations into appropriate subnetworks that partition the wires of the network and that pairwise share at most one 
node (from an electrical perspective, this node represents the shared point of current flow). For instance, in Figure \ref{fig:local-networks} we see that the network at stage $n=2$ consists of two scaled copies of the network at stage $n=1$. 
Because the two copies (subnetworks) are connected by a single node, one may apply network reduction rules within each subnetwork without impacting the other subnetwork. 
Moreover, because of the self-similarity of the copies, one may iterate these reduction rules from stage to stage. 
This is a specific example of the typical procedure used to construct (local) energies on post critically finite (p.c.f.) fractals in \cite{Kigami01}.

To construct nonlocal energy forms on the unit interval, we build networks which are complete/complete bipartite (see Section \ref{sec:construction}).  
As the number of nodes increases, the additional wires quickly preclude the ability to construct appropriate wire-partitions. Thus, after stage $n=1$, there is no context in which to apply simple network reduction rules (series/parallel/delta-wye).

%%%%%%%

\subsubsection*{Solving matrix equations}

Although simple network reduction rules cannot be applied to our complete network setting, conditions for electrical equivalence can still be determined using the (combinatorial) graph Laplacian for the electrical network. 
For this, consider a wire $\{x,y\} \in W_i^{(n)}$ at stage $n \geq 0$ associated with the path $\sigma_{x,y}^{(n)} \in E_i^{(n)}$ in the index space. The inverse of the resistance is called the conductance and is denoted by 
\[
c_i^{(n)}(x,y) = \delta_{\sigma_{x,y}^{(n)}}^{-1}. 
\]
The weighted degree of $x \in V_i^{(n)}$ is the sum of the conductances of all wires connected to $x$ and is denoted by 
\begin{equation}\label{eq:degree}
c_i^{(n)}(x) = \sum_{\{y \in V_i^{(n)} : \{x,y\} \in W_i^{(n)}\}}c_i^{(n)}(x,y). 
\end{equation}
The graph Laplacian associated to the network $(V_i^{(n)}, W_i^{(n)})$ is the matrix $L_i^{(n)}$ with the values $c_i^{(n)}(x)$ along the diagonal and $-c_i^{(n)}(x,y)$ in the off-diagonal entries. 
For example, when $n=2$, the graph Laplacian for $i=1$ is given by
\begin{equation}\label{eq:Laplace-example}
L_1^{(2)} =  \begin{pmatrix}
c_1^{(2)}(0)
%\frac{1}{r_{11}^2} +\frac{1}{r_{11}r_{12}}+\frac{1}{r_{12}r_{23}} +\frac{1}{r_{12}r_{24}} 
& - \frac{1}{r_{11}^2} 
& - \frac{1}{r_{11}r_{12}}
& - \frac{1}{r_{12}r_{23}} 
& -\frac{1}{r_{12}r_{24}}\\[.5em]
-\frac{1}{r_{11}^2} 
&c_1^{(2)}\left(\frac14\right)
%& \frac{2}{r_{11}^2} + \frac{1}{r_{12}r_{22}}+ \frac{1}{r_{12}r_{23}}
& -\frac{1}{r_{11}^2}
&-\frac{1}{r_{12}r_{22}}
&-\frac{1}{r_{12}r_{23}}\\[.5em]
-\frac{1}{r_{11}r_{12}} 
&-\frac{1}{r_{11}^2} 
&c_1^{(2)}\left(\frac12\right)
%& \frac{2}{r_{11}r_{12}} +\frac{2}{r_{11} r_{11}} 
&-\frac{1}{r_{11}^2} 
&-\frac{1}{r_{11}r_{12}} \\[.5em]
-\frac{1}{r_{12} r_{23}} 
&- \frac{1}{r_{12}r_{22}}
&- \frac{1}{r_{11}^2}
&c_1^{(2)}\left(\frac34\right)
%&\frac{2}{r_{11}^2} + \frac{1}{r_{12}r_{22}}+\frac{1}{r_{12} r_{23}} 
& -\frac{1}{r_{11}^2}\\[.5em]
-\frac{1}{r_{12}r_{24}} 
& -\frac{1}{r_{12}r_{23}}
&-\frac{1}{r_{11}r_{12}}
& -\frac{1}{r_{11}^2}
&c_1^{(2)}(1)
%&\frac{1}{r_{11}^2} + \frac{1}{r_{11}r_{12}} +\frac{1}{r_{12}r_{23}}+\frac{1}{r_{12}r_{24}} 
\end{pmatrix},
\end{equation}
where $c_1^{(2)}(x)>0$ are such that the row sums are zero.

The effective resistance (see Definition \ref{defn:ER}) between two nodes $x,y \in V_i^{(n)}$ 
can be written as
\begin{equation*}\label{eq:R-MPinverse}
R_i^{(n)}(x,y) = {b_{xy}^{(n)}}^T(L_i^{(n)})^\dagger b_{xy}^{(n)},
\end{equation*}
where $(L_i^{(n)})^\dagger$ is the Moore-Penrose inverse of the corresponding graph Laplacian and the $b_{xy}^{(n)}$ are differences of standard basis vectors corresponding to $x$ and $y$ at stage $n$.
To determine constraints on the class of resistances that give electrical equivalence between stages $n-1$ and $n$, one would need to symbolically solve the equation 
\begin{equation}\label{eq:MP}
{b_{xy}^{(n)}}^T\left(L_i^{(n)}\right)^\dagger b_{xy}^{(n)}={b_{xy}^{(n-1)}}^T\left(L_i^{(n-1)}\right)^\dagger b_{xy}^{(n-1)},
\end{equation}
for all $x,y \in V_i^{(n-1)}$. Applying this method directly was computationally infeasible.
Even when $n=2$ and $i=1$, MATLAB was not able to solve \eqref{eq:MP} for the values of $L_1^{(2)}$ given in \eqref{eq:Laplace-example}.

\subsubsection*{Star mesh transform} 
As an alternative to solving equation \eqref{eq:MP}, we sought to employ a more complex network reduction rule known as the star-mesh transform. 
The star-mesh transform naturally aligns with our setting as it is used to reduce networks by removing one node at a time from a network, along with all connected edges \cite{LyonsPeres}. 

Let us outline the details for an electrical network $(V,W)$ with effective resistance $R(x,y)$ and conductances $c(x,y)$. 
Fix a node $x_0\in V$. 
Let $R'$ denote the effective resistance of the network $(V\setminus \{x_0\},W')$ where $W'$ is the collection of wires $\{x,y\}$, such that $x,y \in V \setminus \{x_0\}$ with $x\not= y$, determined by the conductances
\[
c'(x,y)=c(x,y)+\frac{c(x,x_0)c(x_0,y)}{c(x_0)}.
\]
Then $R'(x,y)=R(x,y)$ for all $x, y\in V\setminus\{x_0\}$ and $c'(x)=c(x) - \frac{c(x_0,x)^2}{c(x_0)}$.

The star-mesh transform allows one to remove a single vertex from a network without changing the effective resistances between all other vertices \cite{LyonsPeres}. 
In our setting, one would apply the star-mesh transform $|V_i^{(n+1)}|-|V_i^{(n)}| \geq 2^n$ times to determine conditions on the class of weights $r_{j,k}$ that ensure electrical equivalence between the networks $(V_i^{(n)}, W_i^{(n)})$ and $(V_i^{(n-1)}, W_i^{(n-1)})$.  
However, these constraint equations quickly become unwieldy in practice because of the new $c(x_0)$ term introduced at each application. 
Moreover, due to the graph-directed structure, the weights $r_{j,k}$ appear in the resistances in the networks $(V_i^{(n)}, W_i^{(n)})$ for all $1 \leq i \leq j$ and $n \geq N$ for some $N = N(i) \geq 0$.
Consequently, for electrical equivalence to be achieved between consecutive stages in a given component, the weight $r_{j,k}$ must satisfy a collection of constraints.

%%%%%%%%%%%%%%%%%%%%%%%%%%%%%%%%%%%%%%%%%%%%
\section{
Effective resistance distance estimates}\label{sec:appendixB}
%%%%%%%%%%%%%%%%%%%%%%%%%%%%%%%%%%%%%%%%%%%%

When the sequence of discrete energies is 
compatible, the  
effective resistances $R_1^{(n)}(x,y)$ are 
equal on common domain values (c.f.~Definitions~\ref{defn:ER},~\ref{defn:EE}, and~\ref{defn:Compatible}). In this case, we may define a single effective resistance distance on $V_1^*$ as
$$
R_1(x,y)=R_1^{(n)}(x,y)~when~
~x, y \in V_1^{(n)}.
$$

Assuming compatible energies, we obtain upper and lower bounds on this effective resistance distance.

\begin{lem}
Let $i=1$. If $\{(V_1^{(n)}, W_1^{(n)})\}_{n=0}^{\infty}$ are sequence of electrically equivalent networks, then for all $n 
\geq 0$, it holds that
\[
R_1(x,y)=R_1^{(n)}(x,y) \leq 2 \sum_{k=0}^n r_{1,1}^k \quad \hbox{for all}~x,y \in V_1^{(n)}. 
\]
Consequently, if $r_{1,1} \in (0,1)$, then $R_1(x,y)$ is uniformly bounded from above by $2/(1-r_{1,1})$. 
\end{lem}

\begin{proof}
For $n=0$, we simply note that
\[
R_1^{(0)}(0,1)= 1 < 2 r_{1,1}^0. 
\]
Let $n >0$ and assume inductively that the statement holds for $n-1$. We will show the statement holds for $n$. 
Let $x,y \in V_1^{(n)}$ with $x \not= y$. We break into cases based on whether the nodes $x$ and $y$ are born at stage $n$ or at a previous stage. 

\medskip

\noindent{\bf Case 1}. Suppose that $x,y \in V_1^{(n-1)}$. 
By electrical equivalence and the inductive hypothesis,
\[
R_1^{(n)}(x,y)
    = R_1^{(n-1)}(x,y) \leq 2 \sum_{k=0}^{n-1} r_{1,1}^k \leq 2 \sum_{k=0}^n r_{1,1}^k. 
\]

\medskip

\noindent{\bf Case 2}. Suppose that $x \in V_1^{(n-1)}$, $y \in V_1^{(n)} \setminus V_1^{(n-1)}$.
Without loss of generality, suppose that $x < y$. 
Note that $y - 2^{-n} \in V_1^{(n-1)}$. Therefore, by electrical equivalence and the inductive hypothesis,
\begin{align*}
R_1^{(n)}(x,y)
    &\leq R_1^{(n)}(x,y - 2^{-n}) + R_1^{(n)}(y-2^{-n},y) \\
    &= R_1^{(n-1)}(x,y - 2^{-n}) + R_1^{(n)}(y-2^{-n},y) \\
    &\leq 2 \sum_{k=0}^{n-1} r_{1,1}^k+ R_1^{(n)}(y-2^{-n},y). 
\end{align*}
Since the effective resistance between $x$ and $y$ is bounded by the resistance on the wire connecting $x$ and $y$,
\[
R_1^{(n)}(y-2^{-n},y) \leq \delta_{\sigma_{y-2^{-n},y}^{(n)}} = r_{1,1}^{n}. 
\]
Combining, we have
\[
R_1^{(n)}(x,y)
    \leq  2 \sum_{k=0}^{n-1} r_{1,1}^k+  r_{1,1}^{n} \leq 2 \sum_{k=0}^{n} r_{1,1}^k.
\]

\medskip

\noindent{\bf Case 3}. Suppose that  $x,y \in V_1^{(n)} \setminus V_1^{(n-1)}$. 
Without loss of generality, suppose that $x < y$. 
Note that $x+2^{-n},y - 2^{-n} \in V_1^{(n-1)}$. As in Case 2, we find 
\begin{align*}
R_1^{(n)}(x,y)
    &\leq R_1^{(n)}(x,x+2^{-n}) +R_1^{(n)}(x+2^{-n},y - 2^{-n})+ R_1^{(n)}(y-2^{-n},y) \\[.5em]
    &\leq \delta_{\sigma_{x,x+2^{-n}}^{(n)}} 
        + 2\sum_{k=0}^{n-1} r_{1,1}^k
        + \delta_{\sigma_{y-2^{-n},y}^{(n)}}
        = 2\sum_{k=0}^{n} r_{1,1}^k.
\end{align*}

We have exhausted all cases. Consequently, the lemma holds by induction. 
\end{proof}

Recall the definition of the weighted degrees $c_i^{(n)}(x)$ introduced in \eqref{eq:degree}. 

\begin{lem}\label{lem:lowerbound}
Let $i=1$. If $\{(V_1^{(n)}, W_1^{(n)})\}_{n=0}^{\infty}$ are sequence of electrically equivalent networks, then for all $n 
\geq 0$, it holds that
\[
R_1(x,y)\geq \frac{1}{\min\{c_1^{(n)}(x), c_1^{(n)}(y)\}}\quad\hbox{for all}~\{x, y\} \in W_1^{(n)}.
\]
\end{lem}

\begin{proof}
Let $\{x,y\} \in W_1^{(n)}$. Define $S=V_i^{(n)}\setminus \{x\}$ and notice that $S$ and its complement, $S^c$, determine an edgecut between $x$ and $y$, that is, any path connecting $x$ to $y$ must contain an edge connecting a vertex in $S$ to a vertex in $S^c$. Applying the Nash-Williams Inequality \cites{Nash-Williams,LyonsPeres}, we obtain
$$
R_1^{(n)}(x,y)\geq \frac{1}{\sum_{\{x,y\} \in W_1^{(n)}: x\in S^c, y \in S\}} c_i^{(n)}(x,y)}=\frac{1}{c_i^{(n)}(x)}.
$$
 We complete the same process for $S=V_i^{(n)}\setminus \{y\}$ and the result follows.
\end{proof}

\begin{rem} \label{rem:lowerbound1}
Since $c_i^{(n)}(x)$ increases 
as $n$ increases,
the sharpest lower bound in Lemma \ref{lem:lowerbound} is obtained by taking the minimum over all $n$ such that $\{x,y\} \in W_1^{(n)}$.
\end{rem}

\begin{rem} 
Being that the networks $(V_i^{(n)},W_i^{(n)})$ are complete, there are no other possible edge cuts once $\{x,y\}$ is fixed. 
Thus, the Nash-Williams Inequality cannot provide us with a larger lower bound than in Remark \ref{rem:lowerbound1}. 
\end{rem}

Lastly, we show that the effective resistance distance on each finite level approximation is bounded by the Euclidean distance when $i=1$ and a geodesic-type distance when $i>1$ at a fixed stage.

\begin{lem}\label{lem:R-euclidean}
Suppose Assumption \ref{A:kernel} holds and that  $\{(V_i^{(n)}, W_i^{(n)})\}_{n=0}^{\infty}$ are sequences of electrically equivalent networks. 
Then, for $i=1$,
\[
R_1(x,y) = R_1^{(n)}(x,y) \leq \frac{(2^{n+1})^2}{\lambda_1} |x-y|^{1+2s} \quad \hbox{for all}~x,y \in V_1^{(n)},
\]
and for $i>1$, if  $x,y \in V_i^{(n)}$ with $x< y$, then
\[
R_i(x,y) = R_i^{(n)}(x,y) 
\leq \frac{(2^ni)^2}{\lambda_i} 
\begin{cases}
(y-x)^{1+2s}& \hbox{if}~x \in V_{i-}^{(n)},~y \in V_{i+}^{(n)} \\[.5em]
\displaystyle \min_{z \in V_{i+}^{(n)}} \left(|x-z|^{1+2s} + |y-z|^{1+2s}\right) & \hbox{if}~x,y \in V_{i-}^{(n)}\\
\displaystyle \min_{z \in V_{i-}^{(n)}} \left(|x-z|^{1+2s} + |y-z|^{1+2s}\right) & \hbox{if}~x,y \in V_{i+}^{(n)}.
\end{cases}
\]
\end{lem}

\begin{proof}
First take $i=1$. Recall that $
\mu_1^{(n)}(x) \geq 2^{-(n+1)}$
for all $x \in V_1^{(n)}$. 
For any $x,y \in V_1^{(n)}$, we use Remark \ref{rem:conductance-bound} to find
\[
R_1^{(n)}(x,y) 
    \leq \delta_{\sigma_{x,y}^{(n)}}
    \leq \frac{1}{\lambda_1\mu_1^{(n)}(x)\mu_1^{(n)}(y)} |x-y|^{1+2s}
    \leq \frac{(2^{n+1})^2}{c_1}|x-y|^{1+2s}.
\]

Now take $i>1$.  
If $x \in V_{i-}^{(n)}$, $y \in V_{i+}^{(n)}$, then
\[
R_i^{(n)}(x,y) 
    \leq \delta_{\sigma_{x,y}^{(n)}}
    \leq \frac{1}{\lambda_i\mu_i^{(n)}(x)\mu_i^{(n)}(y)} |x-y|^{1+2s}
    = \frac{(2^{n}i)^2}{\lambda_i}(y-x)^{1+2s}.
\]
On the other hand, if $x,y \in V_{i-}^{(n)}$, then there is no wire connecting $x$ and $y$, so we find
\begin{align*}
R_i^{(n)}(x,y)
    &\leq \min_{z \in V_{i+}^{(n)}} \left( R_i^{(n)}(x,z)+R_i^{(n)}(y,z)\right) \\
    &\leq  \min_{z \in V_{i+}^{(n)}} \left( \delta_{\sigma_{x,z}^{(n)}}+\delta_{\sigma_{z,y}^{(n)}}\right)\\
    &\leq \min_{z \in V_{i+}^{(n)}} \left(\frac{(2^{n}i)^2}{\lambda_i}(z-x)^{1+2s} + \frac{(2^{n}i)^2}{\lambda_i}(z-y)^{1+2s}\right)\\
    &= \frac{(2^{n}i)^2}{\lambda_i} \min_{z \in V_{i+}^{(n)}} \left((z-x)^{1+2s} + (z-y)^{1+2s}\right).
\end{align*} 
For $x,y \in V_{i+}^{(n)}$, we similarly find
\[
R_i^{(n)}(x,y)
    \leq \frac{(2^{n}i)^2}{\lambda_i} \min_{z \in V_{i-}^{(n)}} \left((x-z)^{1+2s} + (y-z)^{1+2s}\right).
\]
\end{proof}

\end{appendix}

%%%%%%%%%%%%%%%%%%%%%%%%%%
\section*{Acknowledgments}
%%%%%%%%%%%%%%%%%%%%%%%%%%

It is a pleasure to thank Tushar Das for regular  discussions surrounding this research and for suggestions that helped improve the presentation of the article. 
Part of this research was conducted during a SQuaRE at the American Institute of
Mathematics; the authors thank AIM for the support. 
PAR was partly supported by the NSF grant DMS 2140664. 
MV acknowledges the support of Australian
Laureate Fellowship FL190100081 ``Minimal surfaces, free boundaries and partial differential
equations.''

%%%%%%%%%%%%%%%%%%%%%%%%%%%%%%%%%%%%
\bibliographystyle{imsart-number}
\bibliography{ref}
%%%%%%%%%%%%%%%%%%%%%%%%%%%%%%%%%%%%

\end{document}